\documentclass{elsarticle}

\usepackage[utf8]{inputenc}
\usepackage{amsmath}
\usepackage{amssymb, amsthm, units, bm, esint}
\usepackage[T1]{fontenc}
\usepackage{srcltx}
\usepackage{graphicx}
\usepackage{subfig}
\usepackage{accents}
\usepackage{natbib}
\usepackage{hyperref}
\usepackage[colorinlistoftodos]{todonotes}
\usepackage{geometry}
\usepackage{soul}
\newtheorem{thm}{Theorem}[section]
\newtheorem*{thm*}{Theorem}
 \newtheorem{cor}[thm]{Corollary}
 \newtheorem{lem}[thm]{Lemma}
 \newtheorem*{lem*}{Lemma}
 \newtheorem{prop}[thm]{Proposition}
 \newtheorem{prop*}{Proposition}[section]
 
 \theoremstyle{definition}
 
 \theoremstyle{remark}
 
 \numberwithin{equation}{section}

\def\tens#1{\mathbb{#1}}
\def\vec#1{\boldsymbol{#1}}
\def \R{\mathbb{R}}
\def \N{\mathbb{N}}

\newcommand{\diff}{\mathop{}\!\mathrm{d}}
\newcommand{\wstar}{\overset{*}{\rightharpoonup}}

\def\bT{\tens{T}}

\def\bD{\tens{D}}
\def\bG{\tens{G}}
\def\bO{\tens{O}}

\def\bI{\tens{I}}

\def\bF{\tens{F}}

\def\bB{\tens{B}}
\def\bA{\tens{A}}

\def\bH{\tens{H}}

\def\bQ{\tens{Q}}

\def\bu{\vec{u}}
\def\bv{\vec{v}}
\def\bw{\vec{w}}

\def\bn{\vec{n}}

\def\bq{\vec{q}}

\def\be{\begin{equation}}
\def\ba{\begin{array}}
\def\ea{\end{array}}
\def\ee{\end{equation}}

\DeclareMathOperator{\diver}{div}
\DeclareMathOperator{\Diver}{Div}

\begin{document}

\begin{frontmatter}

\title{On three-dimensional flows of thermo-viscoelastic fluids of Giesekus type\tnoteref{mytitlenote}}
\tnotetext[mytitlenote]{Miroslav Bul\'{i}\v{c}ek and  Tom\'{a}\v{s} Los supported by the project No. 25-16592S financed by GAČR. Miroslav Bul\'{i}\v{c}ek is a member of the Nečas Center for Mathematical Modeling. Jakub Woźnicki was supported by National Science Centre, Poland through project no. 2023/49/N/ST1/02737. }

\author[Bulicek-address]{Miroslav Bul\'{i}\v{c}ek}
\ead{mbul8060@karlin.mff.cuni.cz}
\author[Bulicek-address]{Tom\'{a}\v{s} Los}
\ead{los@karlin.mff.cuni.cz}
\author[1,2]{Jakub Wo\'{z}nicki}
\ead{jw.woznicki@student.uw.edu.pl}

\address[Bulicek-address]{Charles University, Faculty of Mathematics and Physics,  Sokolovsk\'{a}~83, 186~75 Praha~8, Czech Republic}
\address[1]{University of Warsaw, Faculty of Mathematics, Informatics and Mechanics, Stefana Banacha 2, 02-097 Warszawa, Poland}
\address[2]{Institute of Mathematics of Polish Academy of Sciences, Jana i J\k edrzeja \'Sniadeckich 8, 00-656 Warsaw, Poland}


\begin{abstract}
Viscoelastic rate-type fluid models constitute a fundamental framework for the mathematical description of complex materials exhibiting coupled elastic and viscous effects, with a wide range of applications in engineering, biomaterials, and medicine. In realistic regimes, thermal effects are essential and lead to strongly coupled systems in which heat conduction and temperature-dependent constitutive laws play a decisive role.

In this paper, we develop a thermodynamically consistent model for heat-conducting viscoelastic rate-type fluids. We establish the existence of a global weak solution in the full three-dimensional setting. In contrast to the existing literature, no smallness, regularity, or structural restrictions on the initial data are imposed beyond natural energy and entropy bounds, and no additional regularising mechanisms such as artificial stress diffusion are required.

The analysis is based on a weak–strong framework combining energy and entropy balances with compactness tools, allowing us to treat the full nonlinear coupling between the fluid velocity, the temperature, and the elastic stress.

\end{abstract}


\end{frontmatter}

\section{Introduction}

The main goal of the paper is to establish a long-time, large-data existence result for a system of partial differential equations describing the flow of an incompressible viscoelastic rate-type heat-conducting fluid model in a bounded $d$-dimensional domain $\Omega \subset \mathbb{R}^d$ over the time interval $(0,T)$ with $T>0$. Denoting $Q_T:=(0,T)\times \Omega$, the system of equations we are interested in is the following system assumed to be satisfied in $Q_T$:
\begin{align}
\label{1tBurg}
\diver \bv&=0,\\
\label{2tBurg}
\rho\partial_t \bv + \rho\diver (\bv\otimes \bv) - \diver \tens{T}&=\rho\vec{f},\\
\label{3tBurg}
\partial_t \bB +\mbox{Div} (\bB\otimes \bv) - \nabla \bv \bB -\bB (\nabla \bv)^T+ \tens{H}&=\tens{O},\\
\label{4tBurg}
\rho\partial_t e+\rho\diver(e\bv)+\diver \bq-\tens{T}:\bD&=0.
\end{align}
Here, $\rho >0$ is the constant density, $\bv:Q_T \to \mathbb{R}^d$ is the velocity field, $p:Q_T\to \mathbb{R}$ denotes the pressure, $\tens{T}:Q_T\to \mathbb{R}^{d\times d}_{\textrm{sym}}$ denotes the Cauchy stress tensor, $\vec{f}:Q_T \to \mathbb{R}^d$ denotes the density of the external body forces, $\tens{B}:Q_T\to \mathbb{R}^{d\times d}_{\textrm{sym}}$ describes the left Cauchy--Green tensor related to the natural configuration, and $\tens{H}:Q_T\to \mathbb{R}^{d\times d}_{\textrm{sym}}$ is responsible for the dissipative effects in the extra stress tensor and relaxation processes, while $\bq:Q_T \to \mathbb{R}^d$ is the heat flux and $e:Q_T\to \mathbb{R}_+$ denotes the internal energy of the fluid. Equation \eqref{1tBurg} is the incompressibility condition, the balance of linear momentum is described by \eqref{2tBurg}, the evolution of $\tens{B}$ is described by \eqref{3tBurg}, which contains the upper convected Oldroyd derivative, and the balance of internal energy is given by \eqref{4tBurg}. The system must be completed by the constitutive relations. Here, we assume that the primary variable is also the temperature $\theta:Q_T\to \mathbb{R}_+$, and the constitutive relations are of the form:
\begin{align}
\bq&=-\kappa(\theta)\nabla \theta, \label{heat}\\
\tens{H}&=\tau(\theta)\left(\bB^2 - \bB\right)+\alpha(\theta)(\tens{B}-\tens{I}),\label{Giesekus}\\
\label{CauchtBurg}
\tens{T}&=-p \bI+2\nu(\theta) \bD+ 2\rho g(\theta) \bB, \\
\label{etBurg}
e&=c_v \theta+ (g(\theta)-\theta g'(\theta))(\mbox{tr } \bB - d - \ln \det \bB), \\
\label{BeqFFT}
\bB&=\bF \bF^T, \\
\det \bF&>0.  \label{detFform}
\end{align}
The heat flux is assumed to satisfy the classical Fourier law with the heat conductivity $\kappa(\theta)$, the tensorial quantity $\tens{H}$ represents the combination of the Giesekus and Oldroyd models for dissipation with the relaxation parameters $\tau(\theta)$ and $\alpha(\theta)$, and the Cauchy stress is given as the Newtonian part~ $\nu(\theta) \bD$, represented by the viscosity $\nu(\theta)$ and by the symmetric velocity gradient 
$$
\bD:=\frac{\nabla \bv + (\nabla \bv)^T}{2},
$$
and the extra stress tensor $2\rho g(\theta) \bB$. Moreover, we require the matrix $\tens{B}$ to be positive definite, which is guaranteed by \eqref{BeqFFT}. Moreover, the symbol $\tens{O}$ denotes the zero tensor in $\mathbb{R}^{d\times d}$.  Finally, the internal energy $e$ is given by \eqref{etBurg}. 

The system of equations above must be completed by appropriate initial and boundary conditions. Hence,
for the boundary conditions, the paper deals with the following set of boundary conditions 
\begin{equation}\label{tBurgbc}
\begin{split}
\bv\cdot \bn=0, \quad (\tens{T}\bn)_{\tau} = -\gamma_*\bv_{\tau}, \qquad \nabla \theta\cdot \bn = 0 \qquad \textrm{ on } (0,T) \times \partial \Omega.   
\end{split}
\end{equation}
Here,  $\bI$ is the identity tensor, $\vec{z}_\tau:= \vec{z}-(\vec{z}\cdot \bn) \bn$ stands for the tangential part of a vector $\vec{z}\in \R^{ d}$ with respect to $\partial \Omega$, where $\bn$ denotes the unit outward normal vector.The above boundary conditions describe the so-called Navier slip boundary condition for the velocity field, where $\gamma_* \ge 0$ is the slip parameter and the zero flux boundary condition for the internal energy. 

To complete the system, we prescribe the following initial conditions for  $\bv$, $\bF$, and $\theta$ 
\begin{equation}\label{tic}
\begin{split}
\bv(0,\vec{\cdot})=\bv_0, \quad \bF(0,\vec{\cdot})=\bF_{0}, \quad 
\theta(0, \vec{\cdot})= \theta_0 \quad \textrm{ in }\Omega
\end{split} 
\end{equation}
where $\bv_0$, $\bF_0$ and $\theta_0$ are given. Furthermore, the initial conditions for the internal energy and for the tensor~$\bB$ are then given as 
$$
\begin{aligned}
\bB_0(0,\vec{\cdot}):=\bF_0 \bF_0^T,\quad 
e_0(0,\vec{\cdot}):= c_v \theta_0 + (g(\theta_0)-\theta_0 g'(\theta_0))(\mbox{tr } \bB_0-d-\ln \det \bB_0) \qquad \textrm{ in } \Omega.
\end{aligned}
$$

In the above model, we introduced several material parameters and functions. Our main goal is to establish the theory for a general large class of material parameters. Hence, we consider that $\rho, c_v \in (0,\infty)$ and $\gamma_*\in [0,\infty)$ are given constants and that $g, \nu, \tau, \kappa: \mathbb{R}{+}\to\mathbb{R}{+}$ are given continuous functions satisfying for some $K\ge1$.
\begin{align}
        K^{-1}&\leq \nu(\theta)\leq K, \label{estnu}
    \\
     K^{-1}&\leq \tau(\theta)\leq K, \label{esttau}
    \\ 
     K^{-1}&\leq \kappa(\theta)\leq K,\label{estkappa}\\
     0&\le \alpha(\tau)\le K,\label{alphaest}\\
     0&\le g(\theta)\le K. \label{estg}
\end{align}
In addition, we require $g\in C^2(\mathbb{R}_+)$ and that it satisfies 
\begin{align}
g''(\theta)&\le 0\le g'(\theta) \qquad \textrm{ for all }\theta\in (0,\infty), 
 \label{propg}  \\
    \limsup_{\theta\to\infty} \theta  g'(\theta)&=:L<\infty. \label{growthg01}
\end{align}
Our main result of the paper is formulated as follows.
\begin{thm*}
Let $\Omega\subset \mathbb{R}^d$ be Lipschitz and let the initial conditions belong to the natural energy-entropy function space. Assume that the parameters satisfy the assumptions \eqref{estnu}--\eqref{growthg01}. Then there exists a weak dissipative global-in-time solution to the system \eqref{1tBurg}--\eqref{tic}.     
\end{thm*}
Although the above theorem precisely describes the key result of the paper, it is formulated rather vaguely. Therefore, in the remainder of this introduction, we provide a detailed description of what we mean by a dissipative entropy solution and identify the correct function spaces for initial data. Furthermore, we give an indication of the thermodynamical background of the above system and also of the role of various boundary conditions. We also highlight the key novelty compared to the existing literature. The final results are then precisely and rigorously proven in Section~\ref{estandwss} (see Theorem~\ref{THM1}) and Section~\ref{Sec4} (see Theorem~\ref{THM2}). While in Section~\ref{estandwss}  we deal only with the sequential stability of the solution, Section~~\ref{Sec4} provides the complete existence proof.

\subsection{Thermodynamical background and several notions of the solution}


There is a long journey in building the physical background for the studied model, which goes back to \cite{Burg}, who derived a model for describing viscoelastic and relaxation effects in fluids. Such an approach was then extended by the celebrated work of Oldroyd~\cite{oldroyddef}, who put viscoelastic rate-type models on a rigorous theoretical background. Later, due to several inconsistencies between observations, Giesekus~\cite{GIESEKUS} proposed a new model with better properties, and this model is now frequently used. In fact, it is the model described by \eqref{1tBurg}--\eqref{etBurg}, with a constant temperature instead. There are also much more general classes of such fluids, see e.g. \cite{johnson1977amodel}. Furthermore, stress diffusion effects can be incorporated, which then simplify the analysis of the corresponding model. However, a reasonable thermodynamical background for such models, including temperature effects, was missing. The first theoretical results dealing with such models including temperature effects are due to \cite{Rs20}, where the complete thermodynamical background is carefully built and the essential quantities are identified, which are later also used in the analytical parts. The Rajagopal and Srinivasa approach led to a huge development not only at the level of modelling but also at the level of analytical results. We further refer the reader to \cite{Rs20,Rs24,MTR5,MP28} for references on particular applications.

Moreover, this approach, where the Helmholtz free energy plays a central role, was then adopted to develop a large cascade of possible models with full complexity; see, for example, \cite{Hron17,MRT,Ma18}. We may also refer to~\cite{DEO99} for the first systematic approach or to \cite{HuLe}, where the associated entropy estimates were derived for the first time.

To follow this idea, we introduce the Helmholtz free energy
\begin{equation}
\psi(\theta,\bB):= -c_v \theta (\ln \theta -1) + g(\theta)\tilde{\psi}(\bB),
\label{Helm}
\end{equation}
where $\tilde{\psi}$ is of the classical form\footnote{Note that it may have many different forms depending on the context. The only physically motivated requirement is that it has a minimum for $\bB=\bI$ and it tends to infinity whenever $|\bB| \to \infty$ or $|\bB^{-1}|\to \infty$.}
\begin{equation}
\label{deftildepsi}
    \tilde{\psi}(\bB)= \mbox{tr } \bB-d-\ln \det \bB.
\end{equation}
Then the general concept for dealing with other thermodynamical potentials is the following. We define the entropy $\eta$ as
\begin{equation} \label{etaH}
\eta(\theta,\bB):= -\frac{\partial \psi(\theta,\bB)}{\partial \theta},
\end{equation}
and the internal energy as
\begin{equation}\label{eH}
e(\theta,\bB):=\psi(\theta,\bB) +\theta \eta (\theta,\bB).
\end{equation}
With our particular choice of the free energy \eqref{deftildepsi}, we see that $e$ has the form \eqref{etBurg}. Moreover, it follows from the definition of $\tilde{\psi}$ that $\tilde{\psi}(\bB)\ge 0$. Then, using \eqref{etBurg}, we see that $e$ is a nondecreasing function of the temperature if and only if the function $g$ is concave, see assumption \eqref{propg}. Furthermore, the entropy $\eta$ is given by
\begin{align}
        \label{defentpsi}
        \eta(\theta,\bB)&:= c_v \ln \theta - g'(\theta) \tilde{\psi}(\bB).
\end{align}
Next, we define the total energy as the sum of the kinetic and internal energies, i.e.
\begin{align}
        \label{deftotE}
        E(\theta,\bB,\bv)&:= \frac{|\bv|^2}{2}+e(\theta,\bB).
\end{align}

We now derive the equations for these two quantities, and our main goal is to replace \eqref{4tBurg} by something formally equivalent at the level of classical solutions. The reason for this consideration is that the sought solution will not be regular enough (or we do not know whether it exists) to give any meaning to \eqref{4tBurg}. To derive the equation for $E$, it is enough to take the scalar product of \eqref{2tBurg} with $\bv$, sum it with \eqref{4tBurg}, and, using the identity \eqref{1tBurg}, we obtain
\begin{equation}
    \label{evE}
        \rho\partial_t E + \rho\diver\left((E+p) \bv\right)-\diver \left(\kappa(\theta)\nabla \theta\right)= \diver \big((2\nu(\theta) \bD + 2\rho g(\theta)\bB) \bv\big),
\end{equation}

To derive the equation for $\eta$, we divide \eqref{4tBurg} by $\theta$ to obtain
\begin{equation}\label{pentder1}
\begin{split}
    \frac{\rho}{\theta} \left(\partial_t e+\diver(e\bv)\right)+\diver \left(\frac{\bq}{\theta}\right)=\frac{\tens{T}:\bD}{\theta} - \frac{\bq \cdot \nabla \theta}{\theta^2} = \frac{2\nu(\theta)|\bD|^2+ 2\rho g(\theta)\bD:\bB}{\theta} + \frac{\kappa(\theta)|\nabla \theta|^2}{\theta^2},
\end{split}
\end{equation}
where we used the constitutive relations
\eqref{CauchtBurg} and \eqref{heat}. Then, we take the scalar product of \eqref{3tBurg} with $\partial_{\bB}\psi(\theta,\bB)$ to deduce
\begin{equation}
\left(\partial_t \bB +\mbox{Div} (\bB\otimes \bv)\right):\partial_{\bB}\psi(\theta,\bB) -\left( \nabla \bv \bB +\bB (\nabla \bv)^T\right): \partial_{\bB}\psi(\theta,\bB)+ \tens{H}:\partial_{\bB}\psi(\theta,\bB)=\tens{O}.\label{helpeta} 
\end{equation}
Since $\diver \bv = 0$, we can rewrite the first term on the left-hand side with the help of the definitions of $\eta$ and $e$ in \eqref{etaH} and \eqref{eH} as
$$
\begin{aligned}
\left(\partial_t \bB +\mbox{Div} (\bB\otimes \bv)\right):\partial_{\bB}\psi(\theta,\bB)&=\left(\partial_t \psi(\theta,\bB) +\diver (\psi(\theta,\bB)\bv \right) +\left(\partial_t \theta +\diver (\theta \bv)\right) \eta(\theta,\bB)\\
&=\left(\partial_t (e(\theta,\bB)-\theta \eta(\theta,\bB)) +\diver ((e(\theta,\bB)-\theta \eta(\theta,\bB))\bv \right) +\left(\partial_t \theta +\diver (\theta \bv)\right) \eta(\theta,\bB)\\
&=\left(\partial_t e(\theta,\bB)+\diver(e(\theta,\bB)\bv)\right) -\theta \left(\partial_t\eta(\theta,\bB) +\diver (\eta(\theta,\bB)\bv) \right).
\end{aligned}
$$
We also evaluate the second and the third terms on the left-hand side of \eqref{helpeta} as follows. We note that it follows from \eqref{Helm} and \eqref{deftildepsi} that $\partial_{\bB}\psi(\theta,\bB)= g(\theta)(\bI - \bB^{-1})$, provided that $\bB$ is positive definite. Using this, the fact that $\diver \bv =0$ and the definition of $\bH$ in \eqref{Giesekus}, we see that 
$$
\begin{aligned}
 \left(\bH - \nabla \bv \bB -\bB (\nabla \bv)^T\right): \partial_{\bB}\psi(\theta,\bB)&=g(\theta) \left(\tau(\theta)\left(\bB^2 - \bB\right)+\alpha(\theta)(\tens{B}-\tens{I})- \nabla \bv \bB -\bB (\nabla \bv)^T\right):\left(\bI-\bB^{-1}\right)\\
 &=g(\theta)\left( \tau(\theta)|\bB-\bI|^2 + \alpha(\theta) |\bB^{\frac12}-\bB^{-\frac12}|^2 +2\bD:\bB\right).
\end{aligned}
$$
Finally, we insert this relation into \eqref{helpeta}, multiply the result by $\rho/\theta$, and subtract it from \eqref{pentder1} to deduce the entropy identity
\begin{equation}\label{eventr}
    \rho\partial_t \eta+\rho\diver (\eta \bv)-\diver \left(\frac{\kappa(\theta) \nabla \theta}{\theta}\right)= \frac{\kappa(\theta)|\nabla \theta|^2}{\theta^2}+\frac{2\nu(\theta)|\bD|^2+\rho\tau(\theta)  g(\theta)|\bB-\bI|^2+ \rho \alpha(\theta)g(\theta)|\bB^{\frac12} - \bB^{-\frac12}|^2}{ \theta}.
\end{equation}

Hence, we formally derived two identities: one for the total energy \eqref{evE} and the other for the entropy \eqref{eventr}. Note that both are formally equivalent to \eqref{4tBurg}. However, this represents a very fundamental change of perspective when dealing with these identities. The internal energy equation is, in fact, deduced from the equation for the total energy, which plays the role of the second law of thermodynamics. Therefore, \eqref{evE} should be preferred, and one should focus on it and try to formulate it in a weak sense. However, contrary to the classical Navier--Stokes equations, in \eqref{evE} a pressure appears that must be introduced as an integrable quantity. This is the key drawback, and we know that this might not be possible, for example, for Dirichlet boundary conditions. On the other hand, one can follow the ideas in \cite{FeMa06} for the spatially periodic setting or \cite{BuFeMa09,BuMaRa09} for the case of Navier's slip boundary conditions.

Since pressure $p$ might not exist for Dirichlet boundary conditions, we can go further and require only that the total energy satisfies the so-called energy inequality. That is, integrating \eqref{evE} over $\Omega$, using integration by parts and the boundary conditions \eqref{tBurgbc}, we obtain (we can even replace the equality sign by an inequality)
\begin{equation}
    \label{evEO}
 \frac{d}{dt} \rho \int_{\Omega}  E + \gamma_* \int_{\partial \Omega} |\bv|^2 \le 0,
\end{equation}
where no pressure appears. In fact, here we lose the formal equivalence with the internal energy balance. Therefore, we complete the system with \eqref{eventr}, where we replace the equality sign by the inequality sign $\ge$, which is perfectly motivated by the second law of thermodynamics. Note that if we have a smooth solution to \eqref{1tBurg}--\eqref{3tBurg}, which in addition satisfies \eqref{evEO} and \eqref{eventr} with the inequality sign ``$\ge$", then it must also satisfy \eqref{4tBurg}. Hence, we see that we have ``deduced'' a notion of a solution that is compatible with the classical solution if the latter exists. Such a solution will be called a dissipative weak solution. We want to emphasise that, in the context of heat-conducting fluids, such an approach was introduced by  E.~Feireisl, and we refer the interested reader to the classical books \cite{Fe04,Feireisl2017}.

This is our preferred choice for the definition of a solution. To simplify the subsequent discussion, we now set $\rho\equiv 1$ and also $\alpha(\theta)\equiv 0$. Note that this does not change the analysis at all.

\subsection{Reformulation of the rate-type model}

The next difficulty we have to deal with arises in equation \eqref{3tBurg}. Thanks to the energy estimates (derived later), we shall find that the tensorial quantities that are not under derivatives are only $L^1$-integrable, and therefore it may not be easy to identify the limit and to construct the solutions. Therefore, we introduce another quantity that will formally play the role of the square root of $\bB$. Here, we closely follow the idea of Masmoudi~\cite{Mas} and introduce the following
\begin{equation}
    \label{Fev}
    \partial_t \bF + \mathrm{Div } (\bF\otimes \bv) - \nabla \bv \bF + \frac{\tau (\theta)}{2} (\bF \bF^T \bF-\bF)=\bO,
\end{equation}
Assume for a moment that $\bv$, $\theta$, and $\bF$ are smooth and satisfy \eqref{Fev}. Then, define
$$
\bB:= \bF \bF^T,
$$
and we show that $\bB$ satisfies \eqref{3tBurg}. Indeed, multiplying \eqref{Fev} from the right by $\bF^T$ and using that $\diver \bv = 0$, we obtain
$$
\partial_t \bF \bF^T + \bv \cdot (\nabla \bF \bF^T )  - \nabla \bv \bF \bF^T + \frac{\tau (\theta)}{2} (\bF \bF^T \bF -\bF)\bF^T=\bO.
$$
Next, we take the transpose of equation \eqref{Fev} and multiply the result from the left by $\bF$, obtaining
$$
 \bF \partial_t \bF^T + \bv \cdot (\bF \nabla \bF^T)   - \bF \bF^T (\nabla \bv)^T + \frac{\tau (\theta)}{2} \bF (\bF^T \bF \bF^T-\bF^T)=\bO.
$$
Hence, summing the above identities and using the fact that $\bF \bF^T =\bB$, we deduce \eqref{3tBurg}.

Note that the above formal procedure can be made rigorous by using the renormalisation technique developed by Di~Perna and Lions, see~\cite{DiPerna1989}. Indeed, following their approach step by step, one can show that \eqref{Fev} implies \eqref{3tBurg}, provided that $\bF\in L^4(Q_T)^{d\times d}$ and $\bv \in L^2(0,T; W^{1,2}(\Omega)^d)$ with $\diver \bv =0$, which is exactly the setting of the present paper; see also \cite{BuMaLo22,BuMaLo24}.

\subsection{Novelty and bibliographical remarks}

The system under consideration is undoubtedly one of the most studied in the context of viscoelastic fluid models, even in the constant-temperature setting, where it is already quite difficult. However, for the iconic model, the Oldroyd-B model, which corresponds to the setting with $\alpha(\theta)\equiv 1$ and $\tau(\theta)\equiv 0$, there is no global-in-time large-data result. The only global-in-time existence result is known for a simplified model in the so-called co-rotational case, see~\cite{LiMa}. Thus, the only available results for the Oldroyd-B model are obtained by adding an artificial stress diffusion term. For such models, one may establish a global-in-time theory, see e.g. \cite{barrett2011existence,debiec2025onaclass,KrPoSa15}. 

For the Giesekus model, the path towards a global-in-time theory was outlined in~\cite{masmoudi2011global}, albeit without details. Thus, the first rigorous results for Giesekus models are due to \cite{BaBuMa21}, where an additional stress diffusion term is added, while the original Giesekus model was finally treated in~\cite{BuMaLo22,BuMaLo24}; see also~\cite{BuFeMa19} for relevant compressible but very simplified models.

Finally, when switching to models with temperature dependence, not much is known. Indeed, starting with purely viscous fluids, one may follow the basic strategy outlined mainly by Feireisl (for compressible fluids) and Feireisl and M\'{a}lek (for incompressible fluids) and deduce a relatively satisfactory theory. The interested reader is referred to the following articles \cite{MR1453181,BoMu92,FeMa06,BuFeMa09,BuMaRa09} for relevant results, studies, and techniques used.

Finally, for problems involving both elasticity and temperature effects, the very first rigorous global-in-time large-data result is presented in~\cite{BuMaPrSu21}, where, however, only spherical stresses were assumed, i.e. $\mathbb{B}= b\mathbb{I}$. A significant improvement of this result towards more realistic models is developed in~\cite{bathory2024coupling}, where the authors again consider a modification of the Giesekus model by adding stress diffusion into the system and treating only a very simplified form of the internal energy, namely the case when $g(\theta)=\theta$, which then leads to many reductions in the related system. A very relevant study for our problem is \cite{BuMaSu26}, where the existence result is established for models in which the macroscopic equation for $\bB$ is replaced by its microscopic Fokker--Planck version. Finally, the only available result dealing with the Giesekus heat-conducting fluid is established in \cite{BuWo26}, where, however, only the two-dimensional case is treated and where the authors assume rather non-realistic conditions on the heat conductivity; namely, assumption \eqref{estkappa} is replaced by 
$$
K^{-1}(1+\theta)^r \le \kappa(\theta)\le K(1+\theta)^r 
$$
for some sufficiently large $r$. Hence, our main novelty here is that we are able to treat the full three-dimensional (in fact, any-dimensional) case, we do not require any additional mollification due to stress diffusion, and, most importantly, we do not need to impose any unrealistic growth assumptions on the material parameters.

\subsection{Notation, structure of the paper, and auxiliary tools}

To emphasize, we recall that throughout the whole paper we assume (and do not mention it anymore) that $\Omega\subset \R^d$ is a bounded open domain with a Lipschitz boundary, with dimension $d\ge 2$, $T\in(0,\infty)$ is the length of an arbitrary but fixed time interval, and $Q_T:=(0,T)\times \Omega$ is the space--time cylinder.

Further, in the paper we use the standard notions for Lebesgue, Sobolev and Bochner spaces. In case we want to emphasize the vector- or tensor-valued character of a function, this is indicated by the exponent, e.g. $L^p(\Omega)^d$ for vector-valued functions, and $L^{p}(\Omega)^{d\times d}$ for tensor-valued functions. To denote zero trace, we use the abbreviation $W^{1,p}_0(\Omega)$ for Sobolev spaces. Due to the special character of the requirement on the velocity field, we also define the function spaces 
$$
\begin{aligned}
W_{\vec{n}}^{1,p}&:=\{\bv \in W^{1,p}(\Omega)^d: \; \bv \cdot \bn = 0\textrm{ on } \partial \Omega\},\\
W_{\vec{n},\diver}^{1,p}&:=\{\bv \in W_{\vec{n}}^{1,p}: \; \diver \bv =0 \textrm{ in }\Omega\},\\
L^2_{\bn, \diver}&:= \overline{W_{\vec{n},\diver}^{1,2}}^{\|\, \|_{L^2}}
\end{aligned}
$$
Next, for any Banach space $X$, we denote its dual by $X^*$. Furthermore, the space of Radon measures on a set $Z$ is denoted by $\mathcal{M}(Z)$. For $F\in X^*$ and $f\in X$, we use $\langle F, f\rangle_X$ to denote duality pairing, and in the absence of ambiguity we suppress the use of the subscript $X$. Furthermore, $p':=p/(p-1)$ denotes the dual exponent.

The structure of the rest of the paper is as follows. In Section \ref{estandwss}, we derive a priori estimates of hypothetical smooth solutions to \eqref{1tBurg}--\eqref{tic}, and prove their weak sequential stability. In Section~\ref{Sec4}, we establish the existence of a solution. At the beginning of each section, we give the precise definition of each statement.

The rest of this introductory part is devoted to recalling several useful lemmas used in the paper. We start with the introduction of the rescaled entropy.
\begin{prop}
\label{otherform}
   If $\theta> 0$ and $\det \bB> 0$ in $Q_T$, and $(\bv, p, \bB,\theta)$ is a smooth solution to \eqref{1tBurg}--\eqref{3tBurg}, then \eqref{4tBurg} is equivalent to \eqref{eventr} and also to
\small\begin{align}
              &\rho\partial_t \eta_\lambda + \rho \diver (\eta_\lambda \bv)+\left(\rho\tau(\theta) |\bB-\bI|^2-2\rho\bB:\bD\right)(g'(\theta)\theta^\lambda-h_\lambda (\theta))-\diver\left(\frac{\kappa(\theta) \nabla \theta}{\theta^{1-\lambda}}\right)\nonumber\\
               & \qquad=(1-\lambda)   \frac{\kappa(\theta)|\nabla \theta|^2}{\theta^{2-\lambda}}+\frac{2\nu(\theta)|\bD|^2+ \rho \tau(\theta)g(\theta)|\bB-\bI|^2}{ \theta^{1-\lambda}}\label{evetalambda}
     \end{align} 
with 
\begin{align}
        \label{defetalambda}
        \eta_{\lambda}&:=c_v\frac{\theta^\lambda}{\lambda}-h_{\lambda}(\theta)\tilde{\psi}(\bB),\quad \lambda\in (0,1),
\end{align}
where $h_\lambda(\theta)$ denotes a primitive function to $\theta^\lambda g''(\theta)$.
\end{prop}

\begin{proof}
    
We have already proved the equivalence of \eqref{4tBurg} and \eqref{eventr}. Therefore, we focus on the second equivalence, namely that \eqref{4tBurg} is equivalent to \eqref{evetalambda} for arbitrary $\lambda\in (0,1)$. Multiplying \eqref{4tBurg} by $\theta^{\lambda-1}$, using \eqref{CauchtBurg}, \eqref{etBurg} and \eqref{1tBurg}, we obtain 
         \begin{equation*}
         \begin{split}
              & \frac{\rho c_v}{\lambda}(\partial_t \theta^\lambda+\nabla \theta^\lambda\cdot \bv)-\rho(\partial_t \theta+\nabla \theta\cdot \bv) h_{\lambda}'(\theta) \tilde{\psi}(\bB) \\
              & \quad +\rho\left(\partial_t \tilde{\psi} (\bB)+\nabla \tilde{\psi} (\bB)\cdot \bv\right)  (g(\theta)-\theta g'(\theta))\theta^{\lambda-1}-\diver\left(\frac{\kappa(\theta) \nabla \theta}{\theta^{1-\lambda}}\right)\\
              & \qquad =(1-\lambda)   \frac{\kappa(\theta)|\nabla \theta|^2}{\theta^{2-\lambda}}+\frac{2\nu(\theta)|\bD|^2+2\rho g(\theta)\bB:\bD}{ \theta^{1-\lambda}},
         \end{split}
     \end{equation*}
     which can be rewritten as (use also \eqref{1tBurg})
      \begin{equation*}
         \begin{split}
              & \frac{\rho c_v}{\lambda}(\partial_t \theta^\lambda+\nabla \theta^\lambda\cdot \bv)-\rho\left(\partial_t  (h_{\lambda}(\theta) \tilde{\psi}(\bB))+ \nabla(h_{\lambda}(\theta) \tilde{\psi}(\bB))\cdot \bv\right)\\ & \ +\rho\left(\partial_t \tilde{\psi} (\bB)+\nabla \tilde{\psi}(\bB)\cdot \bv\right) (h_{\lambda}(\theta)- g'(\theta)\theta^{\lambda})+\rho(\partial_t \tilde{\psi}(\bB)+\nabla \tilde{\psi}(\bB)\cdot \bv) g(\theta) \theta^{\lambda-1}\\
              & \ -2\rho g(\theta)\theta^{\lambda-1}(\bB-\bI):\bD-\diver\left(\frac{\kappa(\theta) \nabla \theta}{\theta^{1-\lambda}}\right)=(1-\lambda)  \frac{\kappa(\theta)|\nabla \theta|^2}{\theta^{2-\lambda}}+\frac{2\nu(\theta)|\bD|^2}{ \theta^{1-\lambda}}.
         \end{split}
     \end{equation*} 
     At this point, we use \eqref{Fev} multiplied scalarly by $2\rho(h_{\lambda}(\theta)- g'(\theta)\theta^{\lambda})(\bF-\bF^{-T})$, resp. by $2\rho g(\theta)\theta^{\lambda-1}(\bF-\bF^{-T})$, to treat the third, resp. the fourth and fifth term on the left-hand side, and (we use also the symmetry of $\bB$ and \eqref{1tBurg}) conclude \eqref{evetalambda}. All the arrangements were equivalent, hence \eqref{evetalambda} holds if and only if \eqref{4tBurg} holds.
\end{proof}

We continue with recalling several convergence tools. 
\begin{prop}[Chacon's biting lemma, see \cite{BaMu89}]\label{biting}
Let $Q\subset \mathbb{R}^{d+1}$ be a measurable set. Assume that a~sequence $\{f_k\}_{k\in \N}$ satisfies
\begin{equation}\label{L1bound}
\sup_{k\in \N} \|f_k\|_{L^1(Q)}<\infty.
\end{equation}
Then there exist: \emph{(i)} $f \in L^1(Q)$, \emph{(ii)} a non-decreasing family of measurable sets $\{E_j\}_{j\in \N}$, $E_j\subset Q$ for all $j\in \N$, satisfying $\lim_{j\to\infty}|Q\setminus E_j|=0$,
and \emph{(iii)} a subsequence of $\{f_k\}_{k\in \N}$, which we do not relabel, 
such that, for all $j\in \N$, the following holds:
\begin{equation}\label{bitingesta}
    f_k \rightharpoonup f \quad \textrm{weakly in } L^1(E_j).
\end{equation}
\end{prop}

In what follows, whenever we write
$$
f_k \rightharpoonup f \textrm{ in the biting sense in }L^1(Q),
$$
this means that, for the sequence $\{f_k\}_{k\in \N}$ satisfying \eqref{L1bound}, we apply Chacon's biting lemma (Proposition~\ref{biting}) and obtain the corresponding $f\in L^1(Q)$ and a sequence of sets $\{E_j\}_{j\in \N}$ with the properties described above (together with a suitable subsequence of $\{f_k\}_{k\in \N}$, not relabeled) such that \eqref{bitingesta} holds. Moreover, it follows from the characterization of $L^1$-weakly convergent sequences that $\{f_k\}_{k\in \N}$ is uniformly equi-integrable on $E_j$ for any $j\in \N$. This means that, for any $\varepsilon >0$ and any $j\in \N$, there exists $\delta>0$ such that, for all measurable sets $U\subset E_j$ satisfying $|U|\le \delta$, there holds
\begin{equation}
\label{unifL1}
\sup_{k\in \N} \int_{U}|f_k| \le \varepsilon.
\end{equation}
To be precise, the uniform bound~\eqref{L1bound} implies, for a suitable subsequence of $\{f_k\}_{k\in \N}$ (not relabeled), the validity of estimate~\eqref{unifL1}, which, for a further (not relabeled) subsequence of $\{f_k\}_{k\in \N}$, implies the convergence result~\eqref{bitingesta}. On the other hand, the convergence result~\eqref{bitingesta} implies the validity of \eqref{unifL1}.

The next convergence result is an application of Proposition~\ref{biting} to a Stokes-like system.
\begin{prop}\label{bitingT}
Let 
$\{\tilde{\bw}_k\}_{k\in \N}$ and $\{\tilde{\bG}_k\}_{k\in \N}$ be sequences satisfying
\begin{align}\label{bitingC}
\begin{aligned}
\tilde{\bw}_k &\rightharpoonup \tilde{\bw} &&\textrm{weakly in }L^2(0,T;W_{\vec{n},\diver}^{1,2}),\\
\partial_t \tilde{\bw}_k &\rightharpoonup \partial_t\tilde{\bw} &&\textrm{weakly in }L^2(0,T;(W_{\vec{n},\diver}^{1,2})^*),\\
\tilde{\bG}_k &\rightharpoonup \tilde{\bG} &&\textrm{weakly in }(L^{2}(Q_T))^{d\times d}
\end{aligned}
\end{align}
and fulfilling, for a.a. $t\in (0,T)$ and all $\varphi \in W_{\vec{0},\diver}^{1,2}$, the following identity
\begin{equation}\label{formB}
\langle \partial_t \tilde{\bw}_k, \vec{\varphi}\rangle + \int_{\Omega} \tilde{\bG}_k : \nabla \vec{\varphi} =0.
\end{equation}
Then there holds
\begin{align}\label{Br}
\tilde{\bG}_k : \nabla \tilde{\bw}_k \rightharpoonup \tilde{\bG} : \nabla \tilde{\bw} \qquad \textrm{in the biting sense in } L^1(Q_T).
\end{align}
\end{prop}
\begin{proof}
    See \cite{BuMaLo24}, Theorem 1.7 and Remark at the end of Section 2. 
\end{proof}
Finally, the tool that will play the crucial role in the identification of the weak limit is the following.
\begin{prop}\label{div-curl}\textup{\textbf{(Div-curl lemma, \cite{conti2011thedivcurl})}}
    Let $O\subset\R^n$ be an open and bounded domain with a Lipschitz boundary and let $p,q\in (1+\infty)$ with $\frac{1}{p} + \frac{1}{q} = 1$. Suppose that $\{\bu_k\}_{k\in \N}\subset (L^p(O))^n$, $\{\bv_k\}_{k\in \N}\subset (L^q(O))^n$ are sequences satisfying
    $$
    \bu_k \rightharpoonup \bu\text{ weakly in }(L^p(O))^n,\quad \bv_k\rightharpoonup \bv\text{ weakly in }(L^q(O))^n,
    $$
    and
    $$
    \{\bu_k\cdot \bv_k\}_{k\in \N} \text{ is uniformly equiintegrable in }Q_T.
    $$
    Finally, assume that 
    \begin{align*}
    \diver \bu_k &\rightarrow \diver \bu \text{ strongly in }(W^{1,\infty}_0(O))^*,\\
    \mathrm{curl}\,\bv_k &\rightarrow \mathrm{curl}
    \,\bv \text{ strongly in }((W^{1,\infty}_0(O))^{n \choose 2})^*. 
    \end{align*}
    Then,
    $$
    \bu_k\cdot \bv_k\rightharpoonup \bu\cdot \bv\text{ weakly in }L^1(O).
    $$
\end{prop}

To conclude this part, it is essential to mention that the paper relies not only on the results mentioned above, but also on the following two classical techniques. First, we mention the theory of renormalisation for the transport equation developed in~\cite{DiPerna1989}, which is used when dealing with the equation for $\mathbb{B}$. Second, the compensated compactness method for heat-conducting fluids, introduced by Feireisl, see~\cite{Fe04}, and further generalised and extended to other settings, see~\cite{Feireisl2017,FeNo22} and~\cite{BuJuPoZa22}. These methods are used to deduce compactness of the temperature~$\theta$ and the tensor~$\mathbb{B}$.

\section{A priori estimates and weak sequential stability}\label{estandwss}
This section is devoted to the derivation of a priori estimates and to the proof of weak sequential stability. To simplify the notation and the presentation, we provide only the three-dimensional version of the theorem, but the corresponding result holds in any dimension with only minor modifications in the proof.
\begin{thm}\label{THM1}
   Let the material parameters satisfy \eqref{estnu}--\eqref{growthg01}. Assume that for every $k\in \mathbb{N}$ the smooth functions $\{\bv_k, p_k, \theta_k, \bF_k\}$ solve \eqref{1tBurg}--\eqref{tBurgbc} and \eqref{Fev}. Assume that $\det \bF_k >0$ in $Q_T$. Finally, assume that 
\begin{equation}\label{initCCC}
    \begin{aligned}
       \bv_k(0)&\to \bv_0 &&\textrm{strongly in }L^{2}_{\bn, \diver},\\
       \bF_k(0)& \to \bF_0 &&\textrm{strongly in } L^2(\Omega)^{d\times d},\\
       \theta_k(0)&\to \theta_0 &&\textrm{strongly in }L^1(\Omega),\\
       e(\theta_k(0),\bB_k(0))&\to e(\theta_0,\bB_0) &&\textrm{strongly in }L^1(\Omega),\\
       \eta(\theta_k(0),\bB_k(0))&\to \eta(\theta_0,\bB_0) &&\textrm{strongly in }L^1(\Omega).
    \end{aligned}
\end{equation}
Then there exists a subsequence (not relabeled) such that 
\small
\begin{align}
    \bv_k& \rightharpoonup^* \bv && \mbox{weakly-* in } L^\infty(0,T; L^2_{\bn, \diver})\cap L^2(0,T; W_{\bn, \diver}^{1,2}),\label{vkweakT}\\
    \bF_k& \rightharpoonup^* \bF && \mbox{weakly-* in } L^\infty(0,T; (L^2(\Omega))^{3\times3})\cap L^4(Q_T)^{3\times 3},\label{FkweakT}\\
    \theta_k& \rightharpoonup^* \theta  && \mbox{weakly-* in } L^\infty(0,T; L^1(\Omega))\cap L^{\frac{5}{4}}(0,T; W^{1, \frac{5}{4}} (\Omega))\cap L^{\frac{5}{3}}(Q_T), \\
    \partial_t \bv_k & \rightharpoonup \partial_t \bv && \mbox{weakly in }L^{\frac{4}{3}}(0,T; (W_{\vec{n}, \diver}^{1,2})^*),\label{dervkweakT}\\
    \partial_t \bF_k & \rightharpoonup \partial_t \bF && \mbox{weakly in }L^{\frac{4}{3}}(0,T; ((W^{1,2}(\Omega))^{3\times3})^*).
\end{align}
Moreover, for all $\bw\in L^{\infty}(0,T; W^{1,\infty}_{\bn,\diver})$ and all $\bA \in L^{4}(Q_T)^{3\times 3}$, the following identities hold:
\small
\begin{align}
        \label{vincompT}
        \int_0^T\langle \partial_t \bv, \bw\rangle 
        &= \int_{Q_T} (\bv\otimes \bv):\nabla \bw
        - \int_{Q_T} (2\nu(\theta)\bD+g(\theta)\bF \bF^T):\nabla \bw
        - \tau_*\int_{\Sigma_T}\bv_\tau \cdot \bw_{\tau},\\
        \label{FincompT}
        \int_0^T\langle \partial_t \bF, \bA\rangle 
        &= \int_{Q_T} (\bF\otimes \bv):\nabla \bA
        +\int_{Q_T} \left(\nabla \bv \bF-\frac{\tau(\theta)}{2}(\bF \bF^T \bF-\bF)\right):\bA,
\end{align}
Finally, the entropy inequality and the total energy inequality are satisfied in the following sense: for all nonnegative $\varphi \in \mathcal{C}^1_0([0,T); \mathcal{C}^1(\Omega))$ and all nonnegative $\xi\in \mathcal{C}^1_0([0,T))$, we have
\begin{equation}\label{dfneq}
\begin{aligned}
\int_{Q_T}\eta(\theta,\bB) \partial_t \varphi 
&\ge \int_{\Omega} \eta(\theta_0,\bB_0)\varphi(0)   
+\int_{Q_T} \eta \bv\cdot \nabla \varphi 
-\int_{Q_T}\frac{\kappa(\theta) \nabla \theta}{\theta} \cdot \nabla \varphi\\
&\quad +\int_{Q_T} \varphi \left( \frac{\kappa(\theta)|\nabla \theta|^2}{\theta^2}
+\frac{2\nu(\theta)|\bD|^2+\rho\tau(\theta)  g(\theta)|\bB-\bI|^2}{ \theta} \right).
\end{aligned}
\end{equation}
and
\begin{equation}\label{dfnEq}
\begin{aligned}
\int_0^T \xi' \left(\int_{\Omega} \frac{|\bv|^2}{2} + e(\theta,\bB)\right) 
+ \int_0^T \xi \left(\gamma_*\int_{\partial \Omega} |\bv|^2\right) 
\le \xi(0)\int_{\Omega} \frac{|\bv_0|^2}{2} + e(\theta_0,\bB_0).
\end{aligned}
\end{equation}
The initial conditions are attained in the following sense:
\begin{equation}
    \lim_{t\to 0_+} \left(\|\bv(t)-\bv_0\|_2 + \|\theta(t)-\theta_0\|_1 + \|\bF(t)-\bF_0\|_2 \right)=0.\label{aatt}
\end{equation}

\end{thm}

\subsection{A priori estimates}\label{estimates}
Before we establish the a priori estimates for the problem \eqref{1tBurg}--\eqref{tic}, resp. for \eqref{1tBurg}--\eqref{2tBurg}, \eqref{Fev}, \eqref{4tBurg}--\eqref{tic}, we define an appropriate primitive function to $\theta^\lambda g''(\theta)$, i.e. the function $h_\lambda(\theta)$ acting in the equation \eqref{defetalambda}, and show, under certain growth assumptions on the function $g$, its properties, which will be used in the derivation of the a priori estimates.

\begin{lem}
\label{suithlambdatheta}
    Let $g\in C^2(\R)$ be a nondecreasing concave function satisfying, for some $\delta\in (0,1)$, the condition 
     \begin{align}
        \lim_{\theta\to\infty} \theta^{1+\delta} g'(\theta)&=: L_\delta\in [0,\infty). \label{growthg010}
    \end{align}Let the function $h_\lambda=h_\lambda(\theta)$, for $\lambda\in (0,1)$ and $\theta\in (0,\infty)$, be defined as
     \begin{equation}
     \label{defhlambdatheta}
         h_\lambda (\theta):=  \int_\theta^\infty -z^\lambda g''(z) \mathrm{d} z.
     \end{equation}
     Then $h_\lambda(\theta)$ is always nonnegative, bounded, and there holds 
     \begin{equation}
         \label{hlambdathetapf}
         h_\lambda'(\theta)= \theta^\lambda g''(\theta) \qquad \mbox{for all }\theta\in (0,\infty).
     \end{equation}
     In addition,  there exists a finite positive $C=C(g)$ such that for any $\theta\in (0,\infty)$
     \begin{equation}
        \label{hlambdaarr}
        \begin{split}
        h_{\lambda} (\theta)\leq \theta^\lambda g'(\theta)+ 
         \frac{\lambda }{(1-\lambda)}C(g)\theta^{\lambda-\delta-1}.
        \end{split}
    \end{equation}
    \end{lem}

    \begin{proof}
     The nonnegativity of $h_\lambda(\theta)$ for any $\theta\in (0,\infty)$ follows immediately from the concavity of $g$ (even the integrand in \eqref{defhlambdatheta} is nonnegative). 
     
     To prove \eqref{hlambdathetapf} it suffices to show that the integral in \eqref{defhlambdatheta} is convergent for any fixed $\theta\in (0,\infty)$. 
     Employing the integration by parts and \eqref{growthg01} (together with the property $g'\in C([0,\infty))$), we have, for any $\theta\in (0,\infty)$ and some finite positive $C=C(g)$, that
    \begin{equation*}
        \begin{split}
        h_{\lambda} (\theta)&= [-z^\lambda g'(z)]_{\theta}^{\infty}+\int_\theta^\infty \lambda z^{\lambda-1} g'(z) \mathrm{d} z\\&= \theta^\lambda g'(\theta)+\lambda \int_\theta^\infty z^{1+\delta} g'(z) z^{\lambda-\delta-2}\mathrm{d} z\\
        &\leq \theta^\lambda g'(\theta)+{ \lambda C(g)}\frac{\theta^{\lambda-\delta-1}}{(1+\delta-\lambda)},
        \end{split}
    \end{equation*}
    from which it follows that the integral in \eqref{defhlambdatheta} is convergent, thus \eqref{hlambdathetapf} holds. Moreover, we have proved the estimate \eqref{hlambdaarr}.
    
    The boundedness of the function $h_\lambda$ is obvious. Indeed, for any $\theta\in (0,\infty)$
    $$0\leq h_\lambda(\theta)\leq \int_0^\infty -z^\lambda g''(z) \mathrm{d} z$$ (both inequalities following from the concavity of $g$), which is a convergent integral since the integral in \eqref{defhlambdatheta} is convergent, $\lambda\in (0,1)$ and $g''\in C([0,\infty))$.

\end{proof}

\medskip

In the following, the (finite positive) constants of uniform bounds, whose exact values are not essential for our aims, are denoted as $C,\tilde{C}, \hat{C}, \overline{C}$, $C_i$, $i\in \N$, $K$, $\tilde{K}$, $\hat{K}$, $\overline{K}$, $\varepsilon$, $\tilde{\varepsilon}$. Their values can change throughout the text.

 To simplify the presentation, we consider the three-dimensional setting, i.e. the flow domain $\Omega\subset \R^3$.

\begin{prop}
\label{propest}
    Let $(\bv, p, \bF, \bB, e, \theta)$ be a smooth solution to the problem \eqref{1tBurg}--\eqref{2tBurg}, \eqref{Fev}, \eqref{4tBurg}--\eqref{tic}. Let for simplicity $\rho=c_v=\gamma_*=1$, $\vec{f}\equiv \vec{0}$, $\theta>0$ and $\det \bB>0$ in $Q_T$ and let there exist $K\in (0,\infty)$ such that
   
    \eqref{estnu}--\eqref{estkappa}
    
    Then the following uniform estimates hold ($\tilde{\delta}\in (0,\frac14]$ is arbitrary):
    \small\begin{align}
        \label{est1} \sup_{t\in (0,T)} (\|\bv(t)\|_2^2+\|\theta(t)\|_1+\|e(t)\|_1)&\leq C(\|\bv_0\|_2^2+\|e_0\|_1), \\
        \label{est2} \sup_{t\in (0,T)} \|\bF(t)\|_2^2+\|\nabla \bv\|_{2, Q_T}+ \|\bF\|_{4, Q_T} &\leq \tilde{C}(\delta, T, \Omega, g, \bv_0, \bB_0, \theta_0),\\
        \label{est3} \|\theta\|_{\frac{5}{3}-\tilde{\delta}, Q_T}+\|\nabla \theta\|_{\frac{5}{4}-\tilde{\delta}, Q_T}&\leq \hat{C}(\delta, \tilde{\delta}, T, \Omega, g, \bv_0, \bB_0, \theta_0),\\
        \label{est3a0}\sup_{t\in (0,T)} (\|\ln \theta(t)\|_1 + \|\ln \det \bB(t)\|_2)+\|\nabla\ln \theta\|_{2,Q_T}&\leq \overline{C}(\delta, T,\Omega, g, \bv_0, \bB_0, \theta_0).
    \end{align}
\end{prop}

\begin{proof}
    \normalsize The estimate of $\|\bv(t)\|_2$ and $\|e(t)\|_1$ by the right-hand side of \eqref{est1} can be achieved, supposed that $e\geq 0$ in $Q_T$, by summing \eqref{2tBurg} multiplied scalarly by  $\bv$ with \eqref{4tBurg}, integrating the result over $(0,t)\times\Omega$, where $t\in (0,T)$ is arbitrary, and using \eqref{1tBurg}, \eqref{tBurgbc}, \eqref{tic}, the integration by parts and the symmetry of $\bT$ (following from \eqref{CauchtBurg} and \eqref{BeqFFT}). To obtain the validity of \eqref{est1} it remains to prove that also $\|\theta(t)\|_1$ is estimated by its right-hand side and that $e\geq 0$ in $Q_T$. Since $g\in C^2(\R)$ is  concave (which implies that $g(\theta)-\theta g'(\theta)$ is nondecreasing) and $g(0)\geq 0$, we have that 
\begin{equation}\label{posgmgp}
    g(\theta)-\theta g'(\theta)\geq g(0)\geq 0\quad \mbox{in } Q_T.
\end{equation} 
Moreover, since for 

we have
\begin{equation}
\label{dertildepsi}
    \tilde{\psi}'(\bB)= \bI- \bB^{-1}, \quad \tilde{\psi}''(\bB)= \bB^{-1}\otimes \bB^{-1},
\end{equation}
we see that $\tilde{\psi}$ is convex and achieves its minimum at $\bB=\bI$. This minimum is equal to zero, hence
\begin{equation}
\label{postildepsi}
    \tilde{\psi}(\bB)\geq 0\quad \mbox{in } Q_T.
\end{equation}
As we assume that $\theta>0$, we get, by \eqref{etBurg}, \eqref{posgmgp} and \eqref{postildepsi}, that $e> 0$ in $Q_T$ and that the term $\|\theta(t)\|_1$ is also estimated by the right-hand side of \eqref{est1}. Thus \eqref{est1} holds.

The proof of \eqref{est2} is split into the following three paragraphs. We may use Proposition \ref{otherform} and Lemma \ref{suithlambdatheta} since all their assumptions are satisfied. In the first paragraph, we show that, for any $M\in (0,\infty)$, there holds 
\begin{equation}
    \label{entrinequality}
    \int_{ \{\theta\leq M\}} (|\bD|^2 + |\bB-\bI|^2)\leq C(M, g, \bv_0, \bB_0, \theta_0).
\end{equation}
Integrating \eqref{eventr}, with $\eta$ defined in \eqref{defentpsi}, by parts over $\Omega$ and using \eqref{tBurgbc}, we obtain ($\kappa$, $\nu$ and $\tau$ depend on $\theta$ and $\bB$)
\begin{equation}\label{entr}
    \int_{\Omega}\partial_t (\ln \theta - g'(\theta) \tilde{\psi}(\bB))=\int_{\Omega} \left(\frac{\kappa|\nabla \theta|^2}{\theta^2}+\frac{2\nu|\bD|^2+\tau g(\theta) |\bB-\bI|^2}{\theta}\right).
\end{equation}
Integrating \eqref{entr} over time and using the nonnegativity of $g'(\theta) \tilde{\psi}(\bB)$, the property $\ln y\leq y$ valid for all $y\in \R$, the estimate \eqref{est1} and the conditions \eqref{tic}, we see  that (here, $\tilde{C}$ depends on $\bB_0$ and $\theta_0$, and $\hat{C}$ depends on $\bv_0$, $\bB_0$ and $\theta_0$)
\begin{equation}
\label{pomest2}
    \int_{Q_T} \frac{2\nu|\bD|^2+\tau g(\theta)|\bB-\bI|^2}{\theta}\leq  \sup_{t\in (0,T)}\int_{\Omega}\ln \theta (t) + \tilde{C}\leq \sup_{t\in (0,T)}\|\theta(t)\|_1+\tilde{C}\leq \hat{C}.
\end{equation}
From \eqref{estnu}, \eqref{esttau}, \eqref{pomest2} and the properties of $g$\footnote{Since $g\in C^2([0, \infty))$ is non constant, non decreasing and concave, we have that $g'>0$ in some right neighbourhood of $0$, hence, using also the nonnegativity of $g$, we have that $g(s)>0$ for any $s>0$, and if $g(\theta)\to 0$ as $\theta\to 0+$, then the l'Hospital rule gives
$$\lim_{\theta\to0+}(g(\theta)/\theta)=  g'(0)>0.$$} we conclude the validity of \eqref{entrinequality}.

Now, we improve the estimate \eqref{entrinequality} and show, for any $\lambda\in (0,1)$, that
\begin{equation}
\label{mainpomest2}
    \int_{Q_T} \frac{|\bD|^2}{\theta^{1-\lambda}}\leq C(\delta, \lambda, T, \Omega, g, \bv_0, \bB_0, \theta_0),
\end{equation} 
where $\delta\in (0,1)$ satisfies \eqref{growthg01}.
Integrating \eqref{evetalambda}, with $\lambda\in (0,1)$ arbitrary, $\eta_\lambda$ defined in \eqref{defetalambda} and $h_\lambda(\theta)$ defined in \eqref{defhlambdatheta}, over $\Omega$, and using the integration by parts and \eqref{tBurgbc}, we get
     \begin{equation}
     \label{pom1est3}
         \begin{split}
              &\int_{\Omega} \partial_t \left(\frac{\theta^\lambda}{\lambda}-h_{\lambda}(\theta) \tilde{\psi}(\bB)\right) + \int_{\Omega}\left(\tau |\bB-\bI|^2-2\bB:\bD\right)(g'(\theta)\theta^\lambda-h_\lambda (\theta))\\
              & \quad =(1-\lambda) \int_{\Omega}  \frac{\kappa|\nabla \theta|^2}{\theta^{2-\lambda}}+\int_{\Omega}\frac{2\nu|\bD|^2+\tau g(\theta)|\bB-\bI|^2}{ \theta^{1-\lambda}}.
         \end{split}
     \end{equation} 
     Integrating \eqref{pom1est3} over time, using also \eqref{estnu}, \eqref{esttau}, \eqref{tic}, the nonnegativity of $ \tilde{\psi} (\bB)$, $h_\lambda(\theta)$ and $g'(\theta)$ (see \eqref{postildepsi} and \eqref{defhlambdatheta} with $g$ being a concave nondecreasing function) and the Young inequality, we get
     \small\begin{equation}
         \label{pom1est3integr}
         \begin{split}
         \int_0^t \int_\Omega \frac{ |\bD|^2+ g(\theta)|\bB-\bI|^2}{\theta^{1-\lambda}}\leq \tilde{K}\int_0^t \int_\Omega (|\bB-\bI|^2+|\bD|^2)(g'(\theta)\theta^\lambda+h_\lambda(\theta))\\
          + \frac{\tilde{K}}{\lambda}\int_{\Omega} \theta^\lambda(t)+\tilde{C}(\bB_0, \theta_0).
         \end{split}
     \end{equation}
     \normalsize Let us note that $t\in (0,T)$ in \eqref{pom1est3integr} is arbitrary and all integrands in \eqref{pom1est3integr} are nonnegative, hence, 
     applying the fact that the nondecreasing $g$ satisfies $g(s)>0$ for any $s>0$ (see the footnote on the previous page), the boundedness of $g'(\theta) \theta^\lambda$ and $h_\lambda(\theta)$ (see Lemma \ref{suithlambdatheta}), the estimate \eqref{est1} and the Hölder inequality, we can write for any $M\in (0,\infty)$  
     \small\begin{equation}
         \label{pom1est3integr2}
         \begin{split}
         &\int_{Q_T} \frac{|\bD|^2}{\theta^{1-\lambda}} + \int_{\{\theta\geq M\}} \frac{|\bB-\bI|^2}{\theta^{1-\lambda}}\leq 
         \left(1+\frac{1}{g(M)}\right)\left(\hat{K}(g)\int_{\{\theta\leq M\}} (|\bB-\bI|^2+|\bD|^2)\right.\\& \left.\qquad+\tilde{K} \int_{\{\theta\geq M\}} (|\bB-\bI|^2+|\bD|^2)(g'(\theta)\theta^\lambda+h_\lambda(\theta))+\hat{C}(\lambda, T, \Omega, \bv_0, \bB_0, \theta_0)\right).
         \end{split}
     \end{equation}
     The first integral on the right-hand side of \eqref{pom1est3integr2} is estimated by \eqref{entrinequality}. The second integral on the right-hand side of \eqref{pom1est3integr2} can be absorped by its left-hand side. Indeed, using \eqref{growthg01} (and the fact $g'\in C([0,\infty))$), we have
     \begin{equation}
         \label{pom1est3integr3}
         \begin{split}
         \int_{\{\theta\geq M\}}(|\bB-\bI|^2+|\bD|^2)g'(\theta)\theta^\lambda
         \leq \int_{\{\theta\geq M\}}\frac{(|\bB-\bI|^2+|\bD|^2)}{\theta^{1-\lambda}M^\delta}g'(\theta)\theta^{1+\delta}\\
         \leq \frac{\overline{C}(g)}{M^\delta} \int_{\{\theta\geq M\}} \frac{|\bB-\bI|^2+|\bD|^2}{\theta^{1-\lambda}}, 
         \end{split}
     \end{equation}
     and, thanks to \eqref{hlambdaarr} and \eqref{pom1est3integr3}, it holds
     \begin{equation}
         \label{pom1est3integr4}
         \int_{\{\theta\geq M\}}(|\bB-\bI|^2+|\bD|^2)h_\lambda(\theta)\leq \frac{C_1(\lambda,g)}{M^\delta} \int_{\{\theta\geq M\}} \frac{|\bB-\bI|^2+|\bD|^2}{\theta^{1-\lambda}}.
     \end{equation}
     Hence, if $M\in (0, \infty)$ is sufficiently large, we get from \eqref{entrinequality} and \eqref{pom1est3integr2}--\eqref{pom1est3integr4} that
     \begin{equation}
         \label{pom1est3integr5}
         \int_{\{\theta\geq M\}} \frac{(|\bD|^2+|\bB-\bI|^2)}{\theta^{1-\lambda}}\leq C_2(\delta, \lambda, M, T, \Omega, g, \bv_0, \bB_0, \theta_0),
     \end{equation}
     which, again together with \eqref{entrinequality} and \eqref{pom1est3integr2}--\eqref{pom1est3integr4}, gives the desired \eqref{mainpomest2}.
     \normalsize 
     
     Now, we are ready to prove \eqref{est2}. Since the non constant function $g$ is nonnegative, nondecreasing and bounded, there exists
$$\lim_{\theta\to \infty} g(\theta)=: g_\infty\in (0, \infty).$$ Multiplying (scalarly) \eqref{2tBurg} by $\bv$ and \eqref{Fev} by $2g_\infty\bF$, summing the results, integrating by parts over time and space and using \eqref{tBurgbc}, \eqref{tic} and the symmetry of $\bD$ and $\bF \bF^T$, we obtain for all $t\in (0,T)$
\begin{equation}
\label{pom2est2}
\begin{split}
    \frac{\|\bv(t)\|_2^2}{2} + g_\infty\|\bF(t)\|_2^2+2\int_0^t \int_{\Omega}     \nu |\bD|^2+2\int_0^t \int_{\Omega} (g(\theta)-g_\infty)\bF \bF^T: \bD     \\
    +\int_0^t \int_{\partial\Omega} |\bv_\tau|^2+  g_\infty\int_0^t \int_{\Omega} \tau(|\bF \bF^T|^2-|\bF|^2)= 
    \frac{\|\bv_0\|_2^2}{2}+ g_\infty \|\bF_0\|_2^2.
\end{split}
\end{equation}
Applying the Young inequality on the last term in the first line, we get  for any $\varepsilon\in (0,\infty)$ and certain $\tilde{K}=\tilde{K}(\varepsilon)\in (0,\infty)$
\begin{equation}
\label{pom3est2}
     (g(\theta)-g_\infty) \bF \bF^T: \bD\leq  \varepsilon |g(\theta)-g_\infty|  |\bF \bF^T|^2+ \tilde{K}(\varepsilon)|g(\theta)-g_\infty| |\bD|^2. 
\end{equation}
 Again, by the means of the Young inequality applied on the term $\tau |\bF|^2$, we deduce from \eqref{pom2est2}, \eqref{pom3est2} and the fact $\bB_0=\bF_0 \bF_0^T$ (and \eqref{estnu}, \eqref{esttau}, the Korn inequality and the inequality $|\bF|^4\leq 3|\bF \bF^T|^2$) that,  to prove \eqref{est2}, it suffices to show that
\begin{equation}
\label{pom4est2}
    \int_{Q_T}|g(\theta)-g_\infty| |\bD|^2\leq \tilde{C}(\delta, T, \Omega, g, \bv_0, \bB_0, \theta_0),
\end{equation}
where $\delta\in (0,1)$ satisfies \eqref{growthg01}.
    However, taking into account \eqref{mainpomest2} and the fact that $g$ is nondecreasing and bounded, we see that \eqref{pom4est2} holds if for some $\overline{K}(\delta)\in (0,\infty)$
    and $L_\delta$ defined in \eqref{growthg01}
    $$
        \lim_{\theta\to \infty} \theta^\delta (g_\infty-g(\theta))= \overline{K}(\delta)L_\delta\in [0,\infty).
    $$
    The latter holds due to the l'Hospital rule and \eqref{growthg01}. Indeed,
    $$
    \lim_{\theta\to \infty} \theta^\delta (g_\infty-g(\theta))=\lim_{\theta\to \infty} \frac{ g_\infty-g(\theta)}{1/\theta^\delta}= \lim_{\theta\to \infty} \left(\frac{ -g'(\theta)}{-\delta/\theta^{1+\delta}}\right)=\frac{L_\delta}{\delta}.
    $$
    The proof of \eqref{est2} is finished.

     Next, we prove the estimate \eqref{est3}. 
    Employing the boundedness of $g'(\theta)\theta^\lambda$ and $h_\lambda(\theta)$, the nonnegativity of $h_\lambda(\theta)\tilde{\psi}(\bB)$ and the estimate \eqref{est2} together with \eqref{BeqFFT} and \eqref{tic}, we see that \eqref{pom1est3} integrated over time yields for all $t\in (0,T)$ and $\lambda\in (0,1)$ 
    \begin{equation}
        \label{pom2est3}
        (1-\lambda)\int_0^t\int_{\Omega}  \frac{\kappa|\nabla \theta|^2}{\theta^{2-\lambda}}\leq \frac{1}{\lambda} \int_{\Omega} \theta^\lambda(t)+C(\delta, \lambda, T, \Omega, g, \bv_0, \bB_0, \theta_0),
    \end{equation}
    which, thanks to \eqref{estkappa}, \eqref{est1} and the Hölder inequality, implies that
    \begin{equation}
        \label{est1nablathetalambda}
        \int_{Q_T} \frac{|\nabla \theta|^2}{\theta^\lambda}\leq \tilde{C}(\delta, \lambda, T, \Omega, g, \bv_0, \bB_0, \theta_0)\quad \mbox{for all } \lambda\in (1,2).
    \end{equation}
    The estimate \eqref{est1nablathetalambda} gives that $\nabla (\theta^{\frac{2-\lambda}{2}})$ is bounded in $L^2(Q_T)$ for any $\lambda\in (1,2)$, and since $\theta$ is bounded in $L^\infty(0,T; L^1(\Omega))$, we have that ($\hat{C}$ depends on $\delta, \lambda$, $T$, $\Omega$, $g$, $\bv_0$, $\bB_0$ and $\theta_0$)
    \begin{equation}
        \|\theta^{\frac{2-\lambda}{2}}\|_{L^\infty\left(0,T; L^\frac{2}{2-\lambda}(\Omega)\right)}+\|\theta^{\frac{2-\lambda}{2}}\|_{L^2(0,T; W^{1,2}(\Omega))}\leq \hat{C}\quad \mbox{for all } \lambda\in (1,2).
    \end{equation}
    Hence, since $W^{1,2}(\Omega)\hookrightarrow L^6(\Omega)$, we have that $\theta^{\frac{2-\lambda}{2}}$ is bounded in $L^\beta (Q_T)$ if, for a suitable $\alpha\in (0,1)$, it holds
    \begin{equation}
    \label{interpol1}
        \frac{\alpha}{\infty}+\frac{1-\alpha}{2}= \frac{1}{\beta},\qquad 
        \frac{\alpha}{\frac{2}{2-\lambda}}+\frac{1-\alpha}{6}= \frac{1}{\beta}.
    \end{equation}
    The equalities in \eqref{interpol1} hold true for
    \begin{equation}
        \label{interpol2}
        \alpha= \frac{2}{8-3\lambda}, \qquad \beta = \frac{2(8-3\lambda)}{6-3\lambda}.
    \end{equation}
    Thus, for all $\lambda\in (1,\frac53]$, we have that
    \begin{equation}
        \label{interpol3}
        \|\theta^{{\frac{2-\lambda}{2}}}\|_{\beta,Q_T}= \|\theta\|_{\frac{8-3\lambda}{3},Q_T}^\frac{8-3\lambda}{3\beta}\leq \overline{C}(\delta, \lambda, T, \Omega, g, \bv_0, \bB_0, \theta_0), 
    \end{equation}
    which means that $\|\theta\|_{\frac53-\tilde{\delta}, Q_T}$ is bounded for any $\tilde{\delta}\in (0,\frac23]$.

    To complete the proof of \eqref{est3}, we use the Young inequality and write, for any $\lambda\in (1,\frac53]$, that
    \begin{align}
    \label{est3alfin}
        \int_{Q_T} |\nabla \theta|^\mu= \int_{Q_T} \left( \frac{|\nabla \theta|^2}{\theta^\lambda}\right)^\frac{\mu}{2} \theta^{\frac{\mu\lambda}{2}}\leq \int_{Q_T} \frac{|\nabla \theta|^2}{\theta^\lambda}+\theta^\frac{\mu \lambda}{2 \sigma}= \int_{Q_T} \frac{|\nabla \theta|^2}{\theta^\lambda} + \theta^{\frac{2-\lambda}{2}\beta},
    \end{align}
 where $\mu, \sigma$ are such that 
 \begin{equation}\label{interpol4}
     \frac{\mu}{2}+\sigma=1, \qquad \frac{\mu \lambda}{2\sigma}= \frac{2-\lambda}{2}\beta,
\end{equation}
or equivalently,
\begin{equation}
    \label{interpol5}
    \sigma= \frac{2\lambda}{2\lambda+\beta(2-\lambda)}=\frac{3\lambda}{8}, \qquad \mu= \frac{2\beta(2-\lambda)}{2\lambda+\beta(2-\lambda)}= \frac{8-3\lambda}{4}.
\end{equation}
Let us note that the last expression in \eqref{est3alfin} is bounded due to \eqref{est1nablathetalambda} and \eqref{interpol3}. Thus \eqref{est3alfin} together with the expression for $\mu$ in \eqref{interpol5} give the boundedness of $\|\nabla \theta\|_{\frac{5}{4}-\tilde{\delta}, Q_T}$ for any $\tilde{\delta}\in (0, \frac{1}{4}]$, which completes the proof of \eqref{est3}. 

Last, we prove \eqref{est3a0}. The estimate of $\|\ln \theta(t)\|_1$ independent of $t\in (0,T)$ follows from \eqref{pomest2} and the fact that $\ln y\leq y$ for any $y\in \R$, i.e., we have ($C$ depends on $\bB_0$ and $\theta_0$, $\tilde{C}$ depends on $\bv_0$, $\bB_0$ and $\theta_0$)
     \begin{equation}
     \label{estlnt}
         \|\ln \theta(t)\|_1=\int_{\Omega} ((\ln \theta(t))_+ + (\ln \theta(t))_-)\leq 2\|\theta(t)\|_1+C\leq \tilde{C},
     \end{equation}
     where, for any function $f: \R\to \R$, we define
     \begin{equation}
     \label{deffypm}
        f(y)_+:= \max\{0, f(y)\},\quad f(y)_-:= \max\{0, -f(y)\},\qquad y\in \R.
     \end{equation}
     Moreover, from \eqref{estkappa}, \eqref{est1}, \eqref{entr} integrated over time together with \eqref{tic}, the nonnegativity of $g'(\theta) \tilde{\psi}(\bB)$ and the fact that $\ln \theta\leq \theta$ in $Q_T$, we deduce that
     \begin{equation}
     \label{estgradln}
         \|\nabla \ln\theta\|_{2,Q_T}= \left\| \frac{\nabla \theta}{\theta}\right\|_{2,Q_T}\leq C(\bv_0, \bB_0, \theta_0).
     \end{equation}
     To estimate $|\ln \det \bB(t)|$ in $L^{2}(\Omega)$ independently of $t\in (0,T)$, we first multiply \eqref{Fev} scalarly by $2\bF^{-T}$ and use \eqref{1tBurg} and \eqref{BeqFFT} to obtain
     \begin{equation}
     \label{evlndetB}
          \partial_t \ln \det \bB + \diver((\ln \det \bB) \bv) +\tau\mbox{tr} (\bB- \bI)= 0.
     \end{equation}
       Multiplying the equation \eqref{evlndetB} by $\ln \det \bB$, integrating the result over $(0,t)\times \Omega$, where $t\in (0,T)$ is arbitrary, and using the Hölder inequality and \eqref{tic}, we get
      \begin{equation*}
          \int_{\Omega} (\ln \det \bB(t))^2\leq \int_{\Omega} (\ln \det \bB_0)^2+C \|\ln \det \bB\|_{2, Q_T} \|\bB-\bI\|_{2, Q_T}.
      \end{equation*}
      Applying $\sup_{t\in (0,T)}$ on the left-hand side and the Young inequality on the right-hand side, we obtain that
      \begin{equation*}
          \sup_{t\in (0,T)}\|\ln \det \bB(t)\|_2^2\leq \|\ln \det\bB_0\|_2^2+ \varepsilon \sup_{t\in (0,T)}\|\ln \det \bB(t)\|_2^2 + \tilde{C}\|\bB-\bI\|_{2, Q_T}^2
      \end{equation*} 
      with $\varepsilon=\varepsilon(T)$ sufficiently small and $\tilde{C}=\tilde{C}(\varepsilon)$, which, together with \eqref{est2} and \eqref{BeqFFT}, gives 
      \begin{equation}
          \label{estlndetBtL2}
          \sup_{t\in (0,T)} \|\ln \det \bB(t)\|_2\leq \hat{C}(\delta, T, \Omega, g, \bv_0, \bB_0, \theta_0).
      \end{equation}
      This estimate together with \eqref{estlnt} and \eqref{estgradln}, yields \eqref{est3a0}. 

\end{proof}

\medskip

\begin{cor}\label{corpropest}
    Let the assumptions of Proposition \ref{propest} be satisfied. Then, for $\delta\in (0,1)$ introduced in \eqref{growthg01} and any $\tilde{\delta}>0$, there holds
    \small\begin{align}
        \|\partial_t \bv\|_{L^\frac{4}{3}(0,T; (W_{\bn, \diver}^{1,2})^*)}+\|\partial_t \bF\|_{L^\frac43(0,T; ((W^{1,2}(\Omega))^{3\times 3})^*)}&\leq C(\delta, T,\Omega, g, \bv_0, \bB_0, \theta_0),\label{est1der}\\
        \|\partial_t \bB\|_{L^1(0,T; ((W^{1,3+\tilde{\delta}}(\Omega))^{3\times3})^*)}&\leq C(\delta, \tilde{\delta}, T, \Omega, g, \bv_0, \bB_0, \theta_0), \label{est1dera}\\
        \|\partial_t e\|_{L^1(0,T; (W^{1,10+\tilde{\delta}}(\Omega))^*)}&\leq C(\delta, \tilde{\delta}, T, \Omega, g, \bv_0, \bB_0, \theta_0).\label{est2der}
    \end{align}
\end{cor}

\begin{proof}
    \normalsize Let us note that the assumptions of Proposition \ref{propest} imply that  the assumptions of Proposition \ref{otherform} and Lemma \ref{suithlambdatheta} are also fulfilled. To prove \eqref{est1der} and \eqref{est1dera}, we multiply \eqref{2tBurg} by a  $\bw\in L^4(0,T; W_{\bn, \diver}^{1,2})$, \eqref{Fev} by an $\bA\in L^4(0,T; (W^{1,2}(\Omega))^{3\times 3})$ and \eqref{3tBurg}, which is satisfied thanks to Proposition~ \ref{otherform}, by an $\tilde{\bA}\in L^\infty (0,T; (W^{1, 3+\tilde{\delta}} (\Omega))^{3\times 3})$, use \eqref{CauchtBurg}, integrate all three resulting equations over $Q_T$ and use the integration by parts to obtain
    \begin{align}
        \label{pomestder1}\int_0^T\langle \partial_t \bv, \bw\rangle &= \int_{Q_T} (\bv\otimes \bv):\nabla \bw- \int_{Q_T} (2\nu\bD+g(\theta)\bB):\nabla \bw-\int_{\Sigma_T} \bv_\tau\cdot \bw_\tau,\\
        \label{pomestderF}\int_0^T\langle \partial_t \bF, \bA\rangle &= \int_{Q_T} (\bF\otimes \bv):\nabla \bA+ \int_{Q_T} \left(\nabla \bv \bF-\frac{\tau}{2}(\bF \bF^T \bF-\bF)\right):\bA,\\
        \label{pomestder1b}\int_0^T\langle \partial_t \bB, \tilde{\bA}\rangle &= \int_{Q_T} (\bB\otimes \bv):\nabla \tilde{\bA}+ \int_{Q_T} \left(\nabla \bv \bB+\bB (\nabla \bv)^T-\tau(\bB^2-\bB)\right):\tilde{\bA}.
    \end{align}
    The estimates \eqref{est1} and \eqref{est2} together with standard Sobolev embeddings and interpolation imply that 
    \small \begin{equation}
        \label{estvinterpol}
        \|\bv\|_{L^2(0,T; (L^6(\Omega))^3)}+\|\bv\|_{L^4(0,T;(L^3(\Omega))^3)}+\|\bv\|_{\frac{10}{3}, Q_T}\leq C(\delta, T,\Omega, \bv_0, \bB_0, \theta_0).
    \end{equation}
    \normalsize Then, \eqref{est1der} and \eqref{est1dera} are  direct consequences of  \eqref{BeqFFT}, \eqref{est2}, \eqref{pomestder1}--\eqref{estvinterpol}, the embeddings $W^{1,2}(\Omega)\hookrightarrow L^4(\Omega)$ and $W^{1, 3+\tilde{\delta}}(\Omega)\hookrightarrow L^\infty(\Omega)$, the boundedness of $g(\theta)$ (and $\nu$, $\tau$), the trace theorem and the Hölder inequality.

    To prove \eqref{est2der}, we multiply \eqref{4tBurg} by a $\varphi\in L^\infty(0,T; W^{1, 10+\tilde{\delta}}(\Omega))$, use \eqref{CauchtBurg}, \eqref{etBurg} and \eqref{deftildepsi}, and integrate the result by parts over $Q_T$ to get
    \begin{equation}
    \label{pomestder2}\begin{split}
    \int_0^T \langle \partial_t e, \varphi\rangle = \int_{Q_T} \left(\theta +(g(\theta)-\theta g'(\theta)) \tilde{\psi}(\bB)\right) \bv\cdot \nabla\varphi-\int_{Q_T} \kappa\nabla \theta\cdot\nabla\varphi\\+\int_{Q_T} (2\nu|\bD|^2+g(\theta)\bB:\bD)\varphi.
    \end{split}
    \end{equation}
    Since $W^{1, 10+\tilde{\delta}}(\Omega)\hookrightarrow L^\infty(\Omega)$, the last integral on the right-hand side is uniformly bounded thanks to \eqref{BeqFFT} and \eqref{est2} (use also the boundedness of $\nu$ and $g(\theta)$ and the Hölder inequality). The first integral on the right-hand side of \eqref{pomestder2} is uniformly bounded due to \eqref{BeqFFT}, \eqref{est2}, \eqref{est3}, \eqref{est3a0}, \eqref{deftildepsi}, the facts that $\bv$ is uniformly bounded in $(L^\frac{10}{3}(Q_T))^3$ (see \eqref{estvinterpol}) and $g(\theta)-\theta g'(\theta)$ is uniformly bounded in $L^\infty (Q_T)$ (as it is nonnegative and $g$ is nondecreasing and bounded) and the Hölder inequality. The uniform boundedness of the second integral on the right-hand side follows from \eqref{est3}.  
\end{proof}

\medskip

\subsection{Weak sequential stability}\label{wss}
Let $\{(\bv_k, p_k, \bF_k, \bB_k, e_k, \theta_k)\}_{k\in \N}$ be smooth solutions to \eqref{1tBurg}--\eqref{2tBurg}, \eqref{Fev}, \eqref{4tBurg}--\eqref{tic} with the initial conditions $\bv_0$, $\bF_0$, $\bB_0$, $e_0$, $\theta_0$ same for all $k\in \N$. Let the material constants be equal to one, let $\vec{f}\equiv \vec{0}$.  Let the functions $\nu, \tau, \kappa: \R\times \R^{3\times 3}\to \R$ and $g:\R\to \R$ satisfy the assumptions of Proposition \ref{propest}. Then, thanks to Proposition \ref{propest} and Corollary \ref{corpropest}, there exist functions $\bv$, $\bF$, $e$, $\theta$ such that, for a suitably extracted subsequence of $\{\bv_k,\bF_k,e_k,\theta_k\}_{k\in \N}$, which we do not relabel, and any $\tilde{\delta}>0$, there holds as $k\to \infty$ (here, for $q>1$, the symbol $q)$ stands for any $\tilde{q}\in[1, q)$)
\small\begin{align}
    \bv_k& \rightharpoonup^* \bv && \mbox{weakly-* in } L^\infty(0,T; L^2_{\bn, \diver})\cap L^2(0,T; W_{\bn, \diver}^{1,2}),\label{vkweak}\\
    \bF_k& \rightharpoonup^* \bF && \mbox{weakly-* in } L^\infty(0,T; (L^2(\Omega))^{3\times3})\cap (L^4(Q_T))^{3\times 3},\label{Fkweak}\\
    \theta_k& \rightharpoonup^* \theta  && \mbox{weakly-* in } L^\infty(0,T; L^1(\Omega))\cap L^{\frac{5}{4})}(0,T; W^{1, \frac{5}{4})} (\Omega))\cap L^{\frac{5}{3})}(Q_T), \\
    \partial_t \bv_k & \rightharpoonup\ \partial_t \bv && \mbox{weakly in }L^\frac{4}{3}(0,T; (W_{\vec{n}, \diver}^{1,2})^*),\label{dervkweak}\\
    \partial_t \bF_k & \rightharpoonup\ \partial_t \bF && \mbox{weakly in }L^\frac{4}{3}(0,T; ((W^{1,2}(\Omega))^{3\times3})^*),\\
    \partial_t e_k & \rightharpoonup^* \partial_t e && \mbox{weakly-* in }\mathcal{M}([0,T]; (W^{1, 10+\tilde{\delta}}(\Omega))^*).
\end{align}
\normalsize
For various \emph{nonlinear} (not necessarily tensor)  quantities $\bQ_k$ (such as $\nabla\bv_k \bF_k$, $\bF_k\bF_k^T\bF_k$, $|\bF_k|^2$ and some others introduced below), which are uniformly bounded in some Lebesgue spaces $L^r(Q_T)$, $r\geq 1$, there is an element $\overline{\bQ}\in L^r(Q_T)$ for $r>1$, $\overline{\bQ}\in \mathcal{M}(\overline{Q_T})$ for $r=1$, and a subsequence of $\{\bQ_k\}_{k\in \N}$ (denoted again $\{\bQ_k\}_{k\in \N}$) so that the following holds as $k\to \infty$:
\begin{align}
\label{conven}
\bQ_k& \rightharpoonup\ \overline{\bQ} \quad\textrm{ weakly in } L^r(Q_T) && \textrm{ for } r>1,\\
\bQ_k& \rightharpoonup^* \overline{\bQ} \quad\textrm{ weakly-* in } \mathcal{M}(Q_T) && \textrm{ for } r=1. 
\end{align} 
Using this convention, and the weak convergence results above together with the Aubin-Lions compactness lemma and the trace theorem, we get that $(\bv, \bF, e, \theta)$ satisfies, for all $\bw\in L^4(0,T; W_{\bn,\diver}^{1,2})$, $\bA\in L^4 (0,T; (W^{1,2}(\Omega))^{3\times 3})$ and $\varphi\in C_c^1((-\infty,T); W^{1, 10+\tilde{\delta}}(\Omega))$, the following identities:
\small\begin{align}
        \label{vincomp}\int_0^T\langle \partial_t \bv, \bw\rangle &= \int_{Q_T} (\bv\otimes \bv):\nabla \bw- \int_{Q_T} (2\nu\bD+\overline{g(\theta)\bF \bF^T}):\nabla \bw-\int_{\Sigma_T}\bv_\tau \cdot \bw_{\tau},\\
        \label{Fincomp}\int_0^T\langle \partial_t \bF, \bA\rangle &= \int_{Q_T} (\bF\otimes \bv):\nabla \bA+\int_{Q_T} \left(\overline{\nabla \bv \bF}-\frac{\tau}{2}(\overline{\bF \bF^T \bF}-\bF)\right):\bA,
    \end{align}
    and
    \be
        \label{eincomp}\begin{split}
        -\int_{Q_T}e \partial_t \varphi - \int_{\Omega} e_0 \varphi(0) &= \int_{Q_T} \left(\theta +\overline{(g(\theta)-\theta g'(\theta)) \tilde{\psi}(\bF \bF^T)}\right) \bv\cdot \nabla\varphi \\& \qquad +\int_{Q_T} (2\nu\overline{|\bD|^2}+\overline{g(\theta) \bF \bF^T:\bD})\varphi-\int_{Q_T}\kappa\nabla\theta\cdot\nabla\varphi.
        \end{split}
        \ee
\normalsize
Our aim is to show that the nonlinear terms converge to proper functions, i.e. that we can remove the bars in \eqref{vincomp}--\eqref{eincomp}, where the equation corresponding to \eqref{eincomp} holds at least with inequality such that the left-hand side is greater than or equal to the right-hand side (we are not able to prove the equality due to the term $\overline{|\bD|^2}$ acting in \eqref{eincomp}) and that $e$ satisfies \eqref{etBurg}. To complete this, we have to show that a subsequence of $\{\theta_k\}_{k\in \N}$ is convergent  almost everywhere in $Q_T$ and that the sequence $\{\bF_k\}_{k\in \N}$ is compact in $(L^1(Q_T))^{3\times 3}$ (we show the compactness of $\{\bF_k\}_{k\in \N}$ in $(L^2(Q_T))^{3\times 3}$).

\normalsize\subsubsection{Almost everywhere convergence of the temperatures \texorpdfstring{$\{\theta_k\}_{k\in\N}$}{thetak}}\label{subsection:almost_everywhere_conv_theta}

To prove the almost everywhere convergence of (a subsequence of) $\{\theta_k\}_{k\in \N}$ we shall employ the Div-curl lemma \ref{div-curl}. In this subsection all the constants of uniform bounds that depend on the data of our problem are denoted by $C$. 

Our first step is to focus on the entropy equation \eqref{eventr}. Recall that
\begin{equation}
    \label{defetakpsiBk}
\eta_k:=  \ln \theta_k - g'(\theta_k) \tilde{\psi}(\bB_k), \quad \tilde{\psi} (\bB_k):= \mathrm{tr }\ \bB_k - 3 - \ln \det \bB_k.
\end{equation}
By \eqref{BeqFFT} and \eqref{est2} we have that
\begin{align}\label{bounds_on_trace_B_k_l2}
    \|\mathrm{tr}\ \bB_k\|_{2, Q_T}^2\leq \|\bF_k\|_{4, Q_T}^4 \leq C,
\end{align}
which, combined with \eqref{est3a0} and the boundedness of $g'$, readily implies
\begin{align}\label{bound_on_entropy_k_l2}
        \|\eta_k\|_{2, Q_T}\leq  C.
\end{align}
Then, by \eqref{bound_on_entropy_k_l2}, \eqref{estvinterpol} and H\"{o}lder's inequality, we obtain
    \begin{align}
    \label{betakvk}
        \|\eta_k\,\bv_k\|_{\frac{5}{4}, Q_T} \leq  C.
    \end{align}
Now denote
    $$
    \vec{q}_k := \eta_k\,\bv_k - \kappa\nabla\ln\theta_k.
    $$
In particular, the bound \eqref{betakvk} together with \eqref{estkappa} and \eqref{est3a0} implies
    \begin{align}
    \label{bqk}
        \|\vec{q}_k\|_{\frac{5}{4}, Q_T} \leq C.
    \end{align}
On the other hand, we may utilize the equation \eqref{eventr}, together with the bounds  \eqref{estkappa}, \eqref{pomest2} and \eqref{estgradln}, to deduce 
    \begin{multline}
    \label{bdivetakqk}
        \|\diver_{t,x}(\eta_k, \vec{q}_k)\|_{1, Q_T} = \|\partial_t\eta_k + \diver_x \vec{q}_k\|_{1, Q_T}\\ = \left\| \kappa\frac{|\nabla_x\theta_k|^2}{\theta^2_k} + \nu\frac{|\bD_k|^2}{\theta_k} + \tau g(\theta_k)\frac{|\bB_k - \mathbb{I}|^2}{\theta_k}\right\|_{1, Q_T}
        \leq C.
    \end{multline}
Thus, by 
\eqref{bound_on_entropy_k_l2}, \eqref{bqk}, \eqref{bdivetakqk} and the Sobolev embeddings, there exist some $\overline{\eta}\in L^{\frac{5}{4}}(Q_T)$ and $\overline{\vec{q}} \in (L^{\frac{5}{4}}(Q_T))^3$ such that (up to subsequences)
    \begin{align}
    \label{etakweak}
        \eta_k &\rightharpoonup \overline{\eta} \quad\text{ weakly in }L^\frac{5}{4}(Q_T),\\
        \label{qkweak}\vec{q}_k &\rightharpoonup \overline{\vec{q}} \quad\text{ weakly in }(L^\frac{5}{4}(Q_T))^3,\\
        \label{divetakqkcomp}\{\diver_{t,x}(\eta_k, \vec{q}_k)\}_{k\in\N}&\text{ is compact in } (W^{1,5}_0(Q_T))^*.
    \end{align}
To proceed, we need an explicitly stated equation for $\tilde{\psi}(\bB_k)$. By \eqref{1tBurg}, \eqref{BeqFFT} and \eqref{dertildepsi} we can multiply \eqref{Fev} scalarly by $2(\bF_k - \bF_k^{-1})$ to obtain
$$
\partial_t \tilde{\psi}(\bB_k) + \diver_x (\tilde{\psi}(\bB_k) \bv) + \tau|\bB_k - \mathbb{I}|^2 = 2(\bB_k - \mathbb{I}) : \bD_k.
$$
Hence, in a similar way, with the use of the equation above, the bounds \eqref{est3a0}, \eqref{bounds_on_trace_B_k_l2}, \eqref{estvinterpol}, as well as \eqref{esttau}, \eqref{est2} and the Sobolev embeddings, there exist $\overline{\tilde{\psi}(\bB)} \in L^{\frac{5}{4}}(Q_T)$ such that (up to a subsequence)
    \begin{align}
    \label{psiBkweak}
        \tilde{\psi}(\bB_k) &\rightharpoonup \overline{\tilde{\psi}(\bB)}\,\,\,\quad\text{ weakly in }L^{\frac{5}{4}}(Q_T),\\
        \label{psiBkvkweak}
        \tilde{\psi}(\bB_k)\, \bv_k &\rightharpoonup \overline{\tilde{\psi}(\bB)}\, \bv\quad\text{ weakly in }(L^{\frac{5}{4}}(Q_T))^3,\\
        \label{divpsiBkvkcomp}
        \{\diver_{t,x}(\tilde{\psi}(\bB_k), \tilde{\psi}(\bB_k)\,\bv_k)\}_{k\in\N}&\text{ is compact in }(W^{1,5}_0(Q_T))^*.
    \end{align}
Now consider the sequence $\{\theta_k^{\frac{1}{5}}\}_{k\in\N}$. By \eqref{est3} there holds
    \begin{align}
    \label{bthetak15}
        \|\theta_k^{\frac{1}{5}}\|_{6, Q_T}\leq C,
    \end{align}
and by \eqref{est1nablathetalambda}
    \begin{align*}
        \|\nabla_x(\theta_k^{\frac{1}{5}})\|_{2, Q_T} = \frac{1}{5}\|\theta_k^{-\frac{4}{5}}\nabla_x\theta_k\|_{2, Q_T} \leq C.
    \end{align*}
Thus
    \begin{align}
    \label{bcurlthetak15}
        \|\mathrm{curl}_{t,x}\,(\theta_k^{\frac{1}{5}}, 0, 0, 0)\|_{2, Q_T} \leq \|\nabla_x(\theta_k^{\frac{1}{5}})\|_{2, Q_T} \leq C.
    \end{align}
In particular, by 
\eqref{bthetak15}, \eqref{bcurlthetak15} and the Sobolev embeddings, there exists some $\overline{\theta^{\frac{1}{5}}}\in L^{6}(Q_T)$ such that (up to a subsequence)
    \begin{align}
    \label{thetak15weak}
        \theta_k^{\frac{1}{5}}&\rightharpoonup \overline{\theta^{\frac{1}{5}}} \quad\text{ weakly in }L^{6}(Q_T),\\
        \label{curlthetak15comp}
        \{\mathrm{curl}_{t,x}\,(\theta_k^{\frac{1}{5}},0,0,0)\}_{k\in\N}&\text{ compact in }((W^{1,2}_0(Q_T))^6)^*.
    \end{align}
We may combine \eqref{etakweak}--\eqref{divpsiBkvkcomp}, \eqref{thetak15weak}--\eqref{curlthetak15comp} and use the Div-curl lemma \ref{div-curl} to deduce
    \begin{align}
        \eta_k\,\theta_k^{\frac{1}{5}} &\rightharpoonup \overline{\eta}\,\overline{\theta^{\frac{1}{5}}} &&\text{ weakly in }L^{1}(Q_T),\label{after_div_curl_conv_eta_multiplied_theta}\\
        \tilde{\psi}(\bB_k)\,\theta_k^{\frac{1}{5}}&\rightharpoonup\overline{\tilde{\psi}(\bB)}\,\overline{\theta^{\frac{1}{5}}}&&\text{ weakly in }L^1(Q_T)\label{after_div_curl_conv_f_multipled_theta}.
    \end{align}
    
Next, let us note that $w\mapsto -g'(w)$ and $w\mapsto w^{\frac{1}{5}}$ are nondecreasing functions. Thus, due to the nonnegativity of $\tilde{\psi}(\bB_k)$, we have in $Q_T$ for any $w\in\R$
    $$
    0\leq \tilde{\psi}(\bB_k)(-g'(\theta_k) + g'(w))(\theta_k^{\frac{1}{5}} - w^\frac{1}{5}).
    $$
In particular, for $w := \left(\overline{\theta^{\frac{1}{5}}}\right)^5$, we get a.e. in $Q_T$
    \begin{align*}
      0\leq \tilde{\psi}(\bB_k)\left(-g'(\theta_k) + g'\left(\left(\overline{\theta^{\frac{1}{5}}}\right)^5\right)\right)\left(\theta_k^{\frac{1}{5}} - \overline{\theta^{\frac{1}{5}}}\right),  
    \end{align*}
which can be rewritten as follows
    \begin{align*}
        -\tilde{\psi}(\bB_k)g'(\theta_k)\overline{\theta^{\frac{1}{5}}} \leq -\tilde{\psi}(\bB_k)g'(\theta_k)\theta_k^{\frac{1}{5}} + \tilde{\psi}(\bB_k) g'\left(\left(\overline{\theta^{\frac{1}{5}}}\right)^5\right)\left(\theta_k^{\frac{1}{5}} - \overline{\theta^{\frac{1}{5}}}\right).
    \end{align*}
Integrating the last inequality over $Q_T$, using the boundedness of $g'$, the uniform boundedness of $\{\tilde{\psi}(\bB_k)\}_{k\in \N}$ in $L^2(Q_T)$ and $\{\theta_k^{\frac{1}{5}}\}_{k\in \N}$ in $L^6(Q_T)$ (see \eqref{est3a0}, \eqref{bounds_on_trace_B_k_l2} and \eqref{bthetak15}) and the Hölder inequality, taking the limit $k\to \infty$ and using the convention \eqref{conven}, we obtain, due to \eqref{after_div_curl_conv_f_multipled_theta}, that
    \begin{align}\label{ineq:for_monotonicity_trick_f_of_B}
        \int_{Q_T} -\overline{g'(\theta)\tilde{\psi}(\bB)}\,\,\,\overline{\theta^{\frac{1}{5}}}\leq \int_{Q_T} -\overline{g'(\theta)\tilde{\psi}(\bB)\theta^{\frac{1}{5}}}.
    \end{align}
Let us also recall that the sequence $\{\ln \theta_k\}_{k\in \N}$ is uniformly bounded in $L^2(Q_T)$. Having this and \eqref{conven}, we may apply \eqref{defetakpsiBk}, \eqref{after_div_curl_conv_eta_multiplied_theta} and \eqref{ineq:for_monotonicity_trick_f_of_B} to get
    \begin{equation}\label{ineq:for_monotonicity_trick_ln_theta}
        \begin{split}
            \int_{Q_T} \overline{\ln(\theta)}\,\overline{\theta^{\frac{1}{5}}} &= \int_{Q_T} \overline{\eta}\,\overline{\theta^{\frac{1}{5}}} + \int_{Q_T} \overline{g'(\theta)\tilde{\psi}(\bB)}\,\,\,\overline{\theta^{\frac{1}{5}}}\\
            &= \int_{Q_T}\overline{\eta\,\theta^{\frac{1}{5}}} + \int_{Q_T} \overline{g'(\theta)\tilde{\psi}(\bB)}\,\,\,\overline{\theta^{\frac{1}{5}}}\\
            & = \int_{Q_T} \overline{\ln(\theta)\theta^{\frac{1}{5}}} - \int_{Q_T} \overline{g'(\theta)\tilde{\psi}(\bB)\theta^{\frac{1}{5}}}
            + \int_{Q_T}\overline{g'(\theta)\tilde{\psi}(\bB)}\,\overline{\theta^{\frac{1}{5}}}\\
            &\geq \int_{Q_T} \overline{\ln(\theta)\theta^{\frac{1}{5}}}.
        \end{split}
    \end{equation}
Notice that $w\mapsto \ln(w)$ and $w\mapsto w^{\frac{1}{5}}$ are increasing functions. Thus, for any $w\in L^1(Q_T)$ such that $w>0$ a.e. in $Q_T$ and $\ln w \in L^\frac54(Q_T)$, there holds
    \begin{align*}
        \int_{Q_T} (\ln(\theta_k) - \ln(w))(\theta_k^{\frac{1}{5}} - w^{\frac{1}{5}}) \geq 0.
    \end{align*}
    Due to \eqref{ineq:for_monotonicity_trick_ln_theta} we may let $k\to \infty$ in the above inequality to deduce
    \begin{align}\label{ineq:monotonicity_trick_for_arbitrary_w}
        \int_{Q_T} (\overline{\ln(\theta)} - \ln(w))(\overline{\theta^{\frac{1}{5}}} - w^{\frac{1}{5}}) \geq 0.
    \end{align}
Now, by \eqref{est3}, \eqref{est3a0} and the weak lower-semicontinuity of the convex function
        $\phi \mapsto \int_{Q_T} e^\phi$, we may in \eqref{ineq:monotonicity_trick_for_arbitrary_w} consider 
    $$
        w := e^{\overline{\ln\theta} - \lambda h}, \quad\lambda > 0, h\in L^\infty(Q_T) \text{ - arbitrary}
    $$
and obtain
    $$
        \int_{Q_T} (\overline{\theta^{\frac{1}{5}}} - e^{\frac{1}{5}(\overline{\ln\theta} - \lambda h)})h\geq 0.
    $$
Letting $\lambda\to 0^+$, using Lebesgue's Dominated Convergence Theorem and  the arbitrariness of $h$, we get
    \begin{align}
    \overline{e^{\frac{1}{5}\ln\theta}} = \overline{\theta^{\frac{1}{5}}} = e^{\frac{1}{5}\overline{\ln\theta}} \text{ a.e. in }Q_T.\label{eq:equations_weak_limits_mono_trick_div_curl}
    \end{align}
Since $w\mapsto e^{\frac{1}{5}w}$ is an increasing function, the above equalities imply that
\small
    \be
    \label{lnthe15pos}
    (\overline{\ln(\theta)} - \ln(\theta_k))(\overline{\theta^{\frac{1}{5}}} - \theta_k^{\frac{1}{5}})=(\overline{\ln(\theta)} - \ln(\theta_k))(e^{\frac{1}{5}\overline{\ln\theta}} - e^{\frac{1}{5}\ln\theta_k}) \geq 0 \text{ a.e. in } Q_T.
    \ee
    \normalsize
Now let us set $w := \theta_k$ in \eqref{ineq:monotonicity_trick_for_arbitrary_w}. Thanks to  \eqref{ineq:for_monotonicity_trick_ln_theta} and \eqref{lnthe15pos}, we obtain
    \begin{align*}
    \lim_{k\to+\infty}\int_{Q_T} |(\overline{\ln(\theta)} - \ln(\theta_k))(\overline{\theta^{\frac{1}{5}}} - \theta_k^{\frac{1}{5}})| 
    =\lim_{k\to+\infty}\int_{Q_T} (\overline{\ln(\theta)} - \ln(\theta_k))(\overline{\theta^{\frac{1}{5}}} - \theta_k^{\frac{1}{5}})  = 0.
    \end{align*}
Hence, up to a subsequence, there holds
    \begin{align*}
       (\overline{\ln(\theta)} - \ln(\theta_k))(e^{\frac{1}{5}\overline{\ln\theta}} - e^{\frac{1}{5}\ln\theta_k}) = (\overline{\ln(\theta)} - \ln(\theta_k))(\overline{\theta^{\frac{1}{5}}} - \theta_k^{\frac{1}{5}}) \rightarrow 0 \text{ a.e. in }Q_T.
    \end{align*}
    Since $w \mapsto e^{\frac{1}{5}w}$ is a strictly increasing function, the above is only possible if
    \begin{align*}
        \ln(\theta_k) \rightarrow \overline{\ln\theta} \text{ a.e. in }Q_T,
    \end{align*}
    which in turn implies that
    \begin{align}\label{almost_everywhere_convergence_theta_n}
        \theta_k \rightarrow \theta := e^{\overline{\ln\theta}} \text{ a.e. in }Q_T.
    \end{align}

\normalsize\subsubsection{Compactness of \texorpdfstring{$\{\bF_k\}_{k\in \N}$}{Fk} in \texorpdfstring{$(L^2(Q_T))^{3\times 3}$}{L2}}\label{subsection:strong_convergence_of_F}
The proof follows the scheme introduced in \cite{Mas} or \cite{BuMaLo24}. Formally, the evolutionary equation for $\bF_k$ is multiplied scalarly by $\bF_k$, the limited equation for $\bF$ is multiplied scalarly by $\bF$ and we estimate the difference $|\bF_k|^2-|\bF|^2$. Note that the sequence $\{\bv_k, |\bF_k|\}_{k\in \N}$ is uniformly bounded in $ L^2(0,T;W_{\bn,\diver}^{1,2})\times L^4(Q_T)$. For the quantity $\overline{|\bF|^2}-|\bF|^2$ and an arbitrary non-negative $\varphi\in C_c^\infty((-\infty,T)\times\Omega)$, we aim to get the following Gronwall type inequality 
\begin{equation}
    \label{Gronwall}
    -\int_{Q_T} (\overline{|\bF|^2}-|\bF|^2) \partial_t \varphi-\int_{Q_T} (\overline{|\bF|^2}-|\bF|^2)\bv \cdot \nabla \varphi\leq \int_{Q_T} L(\overline{|\bF|^2}-|\bF|^2)\varphi,
\end{equation}
where $L$ is an $L^2(Q_T)$ function specified bellow. Then we apply Proposition~1.5 from \cite{BuMaLo24} and conclude that $\overline{|\bF|^2}-|\bF|^2=0$ almost everywhere in $Q_T$, which is equivalent to the compactness of $\{\bF_k\}_{k\in \N}$ in $(L^2(Q_T))^{3\times3}$, see \cite{BuMaLo24}. 

Let us note that, while deriving \eqref{Gronwall}, the most complicated terms to handle are $\int_{Q_T}\nabla \bv_k: \bF_k \bF_k^T \varphi$, acting in the equation \eqref{Fev} with $\bF_k$ in the role of $\bF$ multiplied scalarly by $\varphi \bF_k$ and integrated over $Q_T$, and   $\int_{Q_T} \overline{(\nabla \bv) \bF}:\bF \varphi$, acting in the mollified (by the standard mollifying kernel $\omega_{\hat{\delta}}$, $\hat{\delta}>0$, for more details see \cite{BuMaLo22} or \cite{BuMaLo24}) form of \eqref{Fincomp}  with $\bA$ being equal to the mollification of $\varphi \bF$, where in the equation with the mollified terms the limit in the mollifying parameter $\hat{\delta}\to 0+$ is taken. To treat the mentioned terms, we shall employ, among other things, the evolutionary equation for $\bv_k$ multiplied scalarly by $\bv_k-\bv$. 
However, in this equation, the convective term 
$(\bv_k-\bv)\diver(\bv_k\otimes \bv_k)$ is not uniformly bounded in $(L^1(Q_T))^3$. To avoid this obstacle we regularize the equation for $\bF_k$ by $G_m(|\bF_k|^2)$, where $G_m(s)$ is a cut-off function defined, for $m\in \N$, as
\begin{equation}
    \label{defGm}
    G_m(s)= G(s/m), \quad G(s)=\left\{\begin{matrix}1& \mbox{if } |s|\leq 1\\
    0& \mbox{if } |s|>2\end{matrix}\right.,\quad G\in C^\infty(\R)
\end{equation}
such that $G$ is nonincreasing in $[0,\infty)$, 
and in the equation for $\bv_k$ we use the tools connected with the existence of biting limits, namely Propositions \ref{biting} and~\ref{bitingT}. Note that such a procedure was already used in the analysis of the Giesekus model in the isothermal three-dimensional setting, see \cite{BuMaLo24}.

\medskip

The proof of the compactness of $\{\bF_k\}_{k\in \N}$ in $(L^2(Q_T))^{3\times3}$ consists of two main steps. 

\textbf{Step 1.}
Subtracting the mollified form of \eqref{Fincomp} with $\bA$ being the mollification of $\varphi \bF$ from \eqref{Fev} with $\bF_k$ in the role of $\bF$ multiplied scalarly by $\varphi G_m(|\bF_k|^2) \bF_k$ and integrated over $Q_T$, and taking (besides the limit in the mollifying parameter) the limits $k\to\infty$ and $m\to \infty$, we shall get, after certain arrangements, the following inequality, valid for all non-negative  $\varphi\in \mathcal{C}_c^\infty((-\infty,T)\times\Omega)$:
\begin{equation}
    \label{step1compFsigma}
    \begin{split}
    -&\int_{Q_T} \left(\overline{|\bF|^2}-|\bF|^2\right)\partial_t \varphi - \int_{Q_T} \left(\overline{|\bF|^2}-|\bF|^2\right) \bv\cdot \nabla \varphi \\
    &\qquad \leq \int_{Q_T} \left(\overline{|\bF|^2}-|\bF|^2\right)\varphi-\ell_{1}+2\ell_{2},
\end{split}
\end{equation}
where
\begin{align}
\label{lb}
 \ell_{1}&:= \lim_{m\to \infty}\lim_{k\to \infty} \int_{Q_T} \left(|\bF_k \bF_k^T|^2 G_m(|\bF_k|^2)-\bF_k\bF_k^T\bF_k:\bF\right)\varphi,
    \\
\label{la}
    \ell_{2}&:= \lim_{m\to \infty}\lim_{k\to \infty} \int_{Q_T} \left(\nabla \bv_k: G_m(|\bF_k|^2) \bF_k \bF_k^T- \nabla\bv_k\bF_k:\bF\right)\varphi.
\end{align}

\textbf{Step 2.}
We shall prove that $\ell_{1}\ge 0$ and
\begin{equation}
\label{step2compFsigma}
    \ell_{2}\leq \int_{Q_T} \tilde{L} \left(\overline{|\bF|^2}-|\bF|^2\right)\varphi
\end{equation}
for all non-negative  $\varphi\in \mathcal{C}_c^\infty((-\infty,T)\times\Omega)$,
where $\tilde{L}$ is an $L^2(Q_T)$ function specified below. This, together with the results from Step 1, implies the validity of \eqref{Gronwall}, and thus completes the proof of the compactness of $\{\bF_k\}_{k\in \N}$ in $(L^2(Q_T))^{3\times 3}$, as clarified in the beginning of this subsection.

\medskip

We concentrate just on the proof of \eqref{step2compFsigma} since the previous results can be derived following step by step the proofs of the corresponding properties in the isothermal setting, see \cite{BuMaLo24}, Proof of the compactness of $\{\bF_k\}_{k\in \N}$ in $(L^2(Q_T))^{3\times 3}$, Steps 1a, 1b, 2a.

\begin{proof}[\textbf{Proof of \eqref{step2compFsigma}.}]
    First we decompose $l_2$ into three parts $L_1+L_2+L_3,$
    where
        \begin{align}
    L_1&:= \lim_{m\to \infty} \lim_{k\to \infty}\int_{Q_T} \left(\nabla \bv_k: G_m(|\bF_k|^2)\bF_k \bF_k^T-\nabla \bv:\bF_k \bF_k^T\right)\varphi,\label{L1}\\
    L_2&:= \lim_{k\to\infty}\int_{Q_T} \left(\nabla \bv:(\bF_k \bF_k^T - \bF \bF^T)\right)\varphi,\label{L2}\\
    L_3&:= \lim_{k\to\infty}\int_{Q_T} \left(\nabla \bv\bF:\bF- \nabla \bv_k \bF_k: \bF\right)\varphi.\label{L3}
\end{align}
For $L_2$, due to \eqref{Fkweak} and the fact that $\nabla \bv\in (L^2(Q_T))^{3\times 3}$, there holds
\begin{equation}
    \label{estL2} L_2\leq \lim_{k\to \infty} \int_{Q_T} |\nabla \bv| |\bF_k-\bF|^2 \varphi=\int_{Q_T} |\nabla \bv| \left(\overline{|\bF|^2}-|\bF|^2\right)\varphi.
\end{equation}

To treat $L_1$ and $L_3$, let us first decompose $\bv_k:= \bv_k^1+\bv_k^2$, where $\bv_k^1\in L^2(0,T; W_{\bn, \diver}^{1,2})$ solves, for all $\bw\in W_{\bn,\diver}^{1,2}$ and almost all $t\in (0,T)$, the problem
\begin{equation}
    \label{stokes1}
    \langle\partial_t \bv_k^1, \bw\rangle+\int_{\Omega} \nabla \bv_k^1: \nabla \bw+\int_{\partial \Omega} \bv_{k_\tau}^1\cdot \bw_{\tau}= -\int_{\Omega} g(\theta_k)\bF_k \bF_k^T:\nabla \bw
\end{equation}
with the initial condition $\bv_k^1(0, \vec{\cdot})= \bv_0$ almost everywhere in $\Omega$. Since $\bv_k^2=\bv_k-\bv_k^1$, we know that $\bv_k^2\in L^2(0,T; W_{\bn, \diver}^{1,2})$ solves, for all $\bw\in W_{\bn, \diver}^{1,2}$ and a.a. $t\in (0,T)$, the problem
\begin{equation}
    \label{stokes2}
    \langle\partial_t \bv_k^2, \bw\rangle+\int_{\Omega} \nabla \bv_k^2: \nabla \bw+\int_{\partial \Omega} \bv_{k_\tau}^2\cdot \bw_{\tau}= \int_{\Omega} (\bv_k\otimes \bv_k):\nabla \bw
\end{equation}
with the initial condition $\bv_k^2(0,\vec{\cdot})=\vec{0}$ almost everywhere in $\Omega$. 
Since $\{\bF_k \bF_k^T\}_{k\in \N}$ is uniformly bounded in $(L^2(Q_T))^{3\times3}$, the standard theory of Stokes problems applied on the equation \eqref{stokes1} gives the existence of $\bv^1$ such that (for suitable not relabeled subsequence of $\{\bv_k^1\}_{k\in \N}$) 
\begin{align}
    \label{vk1weak}\bv_k^1 & \rightharpoonup \bv^1 && \mbox{weakly in } L^2(0,T; W_{\bn, \diver}^{1,2}),\\
    \label{dervk1weak}\partial_t \bv_k^1 & \rightharpoonup \partial_t\bv^1 && \mbox{weakly in } L^2(0,T; (W_{\bn, \diver}^{1,2})^*).
\end{align}
Due to \eqref{vkweak}, \eqref{dervkweak} and the Aubin-Lions compactness lemma, we have that
\begin{equation}
    \label{vkstrong}
    \bv_k \to \bv \quad \mbox{ strongly in } (L^{\frac{10}{3})}(Q_T))^3,
\end{equation}
thus the standard theory of Stokes problems (together with \eqref{vkweak} and \eqref{vk1weak}) applied on the equation \eqref{stokes2} gives the existence of $\bv^2$ satisfying (for suitable not relabeled subsequence of $\{\bv_k^2\}_{k\in \N}$)
\begin{align}
\label{vk2weak}\bv_k^2 & \rightharpoonup \bv^2 && \mbox{weakly in } L^2(0,T; W_{\bn, \diver}^{1,2}),\\
    \label{nablavk2strong}\nabla \bv_k^2 & \to \nabla\bv^2 && \mbox{strongly in } (L^{2)}(Q_T))^{3\times 3}.
\end{align}
We shall deal with biting limits of terms involving $\nabla \bv^1_k$. We define the 
sequence $\{b_k\}_{k\in \N}$ as
\begin{equation}\label{CISLO}
b_k:= |\nabla \bv_k^1|^2+|\bF_k|^4. 
\end{equation}
Thanks to \eqref{est2} and \eqref{vk1weak}, the sequence $\{b_k\}_{k\in \N}$ is uniformly bounded in $L^1(Q_T)$. 
Hence, Chacon's biting lemma (Proposition~\ref{biting}) and the characterisation of $L^1$-weakly converging sequences imply, for a suitable not relabeled subsequence of $\{b_k\}_{k\in \N}$, the existence of a nondecreasing sequence of measurable sets $\{E_j\subset Q_T\}_{j\in \N}$ with the property $|Q_T \setminus E_j|\to 0$ as $j\to \infty$ and satisfying that, for any $j\in \N$ and any $\varepsilon>0$, there exists $\tilde{\delta}>0$ such that, for any measurable set $U\subset E_j$ with $|U|<\tilde{\delta}$, there holds
\begin{equation}
\sup_{k\in \N} \int_{U} b_k \le \varepsilon.\label{EQI}
\end{equation}
Now we decompose the terms $L_1$ and $L_3$ in the following way:
\begin{equation}
\label{L13dec}
L_1=L_1^{1}+L_1^{2}, \qquad L_3=L_3^{1}+L_3^{2},
\end{equation}
where 
\begin{align*}
    &L_1^{1}:= \lim_{m\to \infty} \lim_{k\to \infty}\int_{E_j} \left(\nabla \bv^1_k: G_m(|\bF_k|^2)\bF_k \bF_k^T-\nabla \bv^1:\bF_k \bF_k^T\right)\varphi,\\
    &L_3^{1}:= \lim_{k\to\infty}\int_{E_j} \left(\nabla \bv^1\bF:\bF- \nabla \bv^1_k \bF_k: \bF\right)\varphi.
\end{align*}

We start with a suitable identification of $L_1^1$. We first realize that ($C$ depends on the data)
\begin{equation}
    \label{Gmn1}\left|\{G_m (|\bF_k|^2)\neq 1\}\right|\leq \left|\{|\bF_k|^2\geq m\}\right|\leq \int_{\{|\bF_k|^4\geq m^2\}} \frac{|\bF_k|^4}{m^2}\leq \frac{C}{m^2},\end{equation}
    which, due to the uniform equi-integrability of $\{b_k\}_{k\in \N}$ on $E_j$ (see \eqref{EQI}), the Young inequality and the fact  that $G_m$ is bounded, enables to rewrite $L_1^1$ as
    \begin{equation}
        \label{L11eq} L_1^1= \lim_{k\to \infty} \int_{E_j} (\nabla \bv_k^1-\nabla \bv^1): \bF_k \bF_k^T \varphi.  
    \end{equation}
    From the uniform equi-integrability of $\{b_k\}_{k\in \N}$ on $E_j$ and the boundedness of $g$ it also follows that the sequence 
    $\{\tilde{\bG}_k:\nabla \bv_k^1\}_{k\in \N}$, where $$\tilde{\bG}_k:= g(\theta_k) \bF_k \bF_k^T+\bD_k^1,$$ is uniformly equi-integrable on $E_j$ as well. Hence, due to the properties of $E_j$, the convergence results \eqref{vk1weak} and \eqref{dervk1weak}, the weak convergence of (a subsequence of) $\{\bG_k\}_{k\in \N}$ in $(L^2(Q_T))^{3\times3}$, the equation \eqref{stokes1} and Proposition \ref{bitingT}, we have, for an appropiate not relabeled subsequence, that
    \begin{equation}
        \tilde{\bG}_k: \nabla \bv_k^1 \rightharpoonup \tilde{\bG}:\nabla \bv^1\quad \mbox{ weakly in } L^1(E_j), 
    \end{equation}
    and thus, for all $\psi\in L^\infty(Q_T)$ (and a further not relabeled subsequence), there holds
    \begin{equation}
    \label{L11key}
        \lim_{k\to \infty} \int_{E_j} (\nabla \bv_k^1-\nabla \bv^1): g(\theta_k) \bF_k \bF_k^T \psi=-\lim_{k\to \infty} \int_{E_j} |\bD_k^1-\bD^1|^2\psi.
    \end{equation}
    Next, since $\theta_k\to \theta$ almost everywhere in $Q_T$ and $\{\ln \theta_k\}_{k\in \N}$ is uniformly bounded in $L^1(Q_T)$, which follows from \eqref{est3a0} (with $\theta_k$ in the role of $\theta$), by the Fatou lemma it holds that (recall that $\delta$ is defined in \eqref{growthg01}, use the convention $\ln 0=-\infty$)
    \small\begin{equation}
        \sup_{t\in (0,T)} \int_{\Omega}|\ln \theta(t)|
        \leq \sup_{t\in (0,T)}\liminf_{k\to \infty}\int_{\Omega}|\ln \theta_k(t)|\leq \tilde{C}(\delta, T, \Omega, g, \bv_0, \bB_0, \theta_0),
    \end{equation}
    \normalsize which directly gives that $\theta>0$ almost everywhere in $Q_T$ and that $\ln \theta\in L^1(Q_T)$. Recall (from the footnote on page 9) that $g(\theta)>0$ whenever $\theta>0$. Hence, due to the  continuity of $g$, we have
    \begin{equation}
    \label{ae1lt}\frac{1}{g(\theta_k)}\to \frac{1}{g(\theta)}\quad \mbox{ a.e. in } Q_T.\end{equation} 
    Moreover, thanks to the property that if $g(\theta)\to 0+$ as $\theta\to 0+$, then $g(\theta)/\theta\to g'(0)>0$ (see again the footnote on page 9 and recall that $g$ is nondecreasing, continuous and bounded), one can deduce, for any $M\in (0,\infty)$ and some finite positive $\varepsilon=\varepsilon(g,M)$, that (in the first inequality we use again the Fatou lemma)
    \small\begin{equation*}
    \begin{split}
        &\int_{Q_T} |\ln g(\theta)|\leq \sup_{k\in \N} \int_{Q_T}|\ln g(\theta_k)|=\sup_{k\in \N} \int_{\{\theta_k\leq M\}} |\ln g(\theta_k)|+\int_{\{\theta_k> M\}}|\ln g(\theta_k)|\\ &\qquad\leq \sup_{k\in \N}\int_{\{\theta_k\leq M\}} \left|\ln \left(\frac{g(\theta_k)}{\theta_k}\right)+\ln \theta_k\right|+|\{\theta_k>M\}| (|\ln g(M)|+ |\ln g_\infty|)\\ &\qquad \leq \left\{\begin{matrix}
            \begin{split}
                \sup_{k\in N}\int_{\{\theta_k\leq M\}} |\ln \theta_k|+|\{\theta_k\leq M\}|(|\ln g'(0)|+\varepsilon(g,M))\\ \quad+|\{\theta_k>M\}| (|\ln g(M)|+ |\ln g_\infty|)\end{split}& \mbox{ if } g(0)=0
                \\ \ \ 
                \\
            |Q_T| (|\ln g(0)|+ |\ln g_\infty|+ |\ln g(M)|)& \mbox{ if } g(0)>0
        \end{matrix}\right.,    
    \end{split}
    \end{equation*} \normalsize
    which, thanks to \eqref{est3a0} and the properties of $g$, implies that $\ln g(\theta)\in L^1(Q_T)$ and $\{\ln g(\theta_k)\}_{k\in \N}$ is uniformly bounded in $L^1(Q_T)$. 
   Due to these facts we see that $\sup_{k\in \N}|\{g(\theta_k)\leq \frac{1}{l}\}|\to 0$ and $|\{g(\theta)\leq \frac{1}{l}\}|\to 0$ as $l\to \infty$. 
    Hence, the Egoroff theorem applied on \eqref{ae1lt} then implies that there exists a non-decreasing sequence of measurable sets $\{\mathcal{E}_l\}_{l\in \N}$ such that $|Q_T\setminus\mathcal{E}_l|\to 0$ as $l\to \infty$ and
    \begin{equation}
    \label{1ltunif}
        \frac{1}{g(\theta_k)} \to \frac{1}{g(\theta)} \quad \mbox{ strongly in } L^\infty(\mathcal{E}_l).
    \end{equation}
    Combining \eqref{L11key} and \eqref{1ltunif}, we obtain, for all $\varphi\in C_c^\infty((-\infty,T)\times \Omega)$, that
    \begin{equation}
        \label{L11key2}
        \lim_{k\to \infty} \int_{E_j\cap \mathcal{E}_l} (\nabla \bv_k^1-\nabla \bv^1): \bF_k \bF_k^T \varphi= -\lim_{k\to \infty} \int_{E_j\cap \mathcal{E}_l}\frac{|\bD^1_k-\bD^1|^2}{g(\theta)}\varphi.
    \end{equation}
    Hence, due to the uniform equi-integrability of the integrand on the left-hand side of \eqref{L11key2} and also of $|\bD_k^1-\bD^1|^2$ on $E_j$ (see \eqref{CISLO} and \eqref{EQI}), \eqref{L11eq}, and the properties of $\mathcal{E}_l$, we  have, for all non-negative $\varphi\in C_c^\infty((-\infty,T)\times\Omega)$, that
    \small\begin{equation}
        \begin{split}
        \label{L11fin}
        &L_1^1= \lim_{l\to \infty} \lim_{k\to \infty} \int_{E_j\cap \mathcal{E}_l} (\nabla \bv_k^1-\nabla \bv^1): \bF_k \bF_k^T \varphi \\&\qquad\leq -\lim_{l\to \infty} \lim_{k\to \infty} \int_{E_j\cap \mathcal{E}_l}\frac{|\bD^1_k-\bD^1|^2\varphi}{R}
        = -\lim_{k\to\infty}\int_{E_j} \frac{|\bD^1_k-\bD^1|^2\varphi}{R},\end{split}
    \end{equation}
    \normalsize
    where $R:= \sup_{\theta>0} \{g(\theta)\}\in (0,\infty)$.

    Next, we introduce an appropriate estimate of $L_3^1$. Using \eqref{vk1weak}, \eqref{Fkweak}, Young's inequality and the localised form of Korn's inequality (see~\cite{BuMaLo22}, Step 2 of the proof of the compactness of $\{\bF_\varepsilon\}$), we have that 
\begin{equation}
    \label{L3j1est}
    \begin{aligned}
&L_3^{1}= \lim_{k\to\infty}\int_{E_j}\left((\nabla\bv^1-\nabla \bv^1_k) (\bF_k-\bF): \bF\right)\varphi
\\&\qquad\le \lim_{k\to\infty}\int_{E_j} \left(\frac{1}{R}|\bD^1-\bD^1_k|^2 + \frac{R}{2}|\bF_k-\bF|^2|\bF|^2\right)\varphi\\&\qquad\le\lim_{k\to \infty} \int_{E_j}\frac{1}{R}|\bD^1-\bD^1_k|^2 \varphi+\int_{Q_T}\frac{R}{2}|\bF|^2(\overline{|\bF|^2}-|\bF|^2)\varphi.
\end{aligned}
\end{equation}
Summing \eqref{L11fin} with \eqref{L3j1est}, we obtain that
\begin{equation}
\label{compFkkey} L_1^1+L_3^1\leq \int_{Q_T} \frac{R}{2} |\bF|^2 (\overline{|\bF|^2}-|\bF|^2)\varphi
\end{equation}
with $R:= \sup_{\theta>0} \{g(\theta)\}$.

Last, we briefly show that 
\begin{equation}
    \label{compFkeyrep} \lim_{j\to \infty} L_1^2=\lim_{j\to \infty} L_3^2=0.
\end{equation}
From \eqref{L1} and \eqref{L13dec} it follows that
\begin{equation}\begin{split}
    \label{L12} L_1^2=\lim_{m\to \infty} \lim_{k\to \infty}\int_{Q_T\setminus E_j }  \left(\nabla \bv^1_k: G_m(|\bF_k|^2)\bF_k \bF_k^T-\nabla \bv^1:\bF_k \bF_k^T\right)\varphi\\
    + \int_{Q_T} \left(\nabla \bv^2_k: G_m(|\bF_k|^2)\bF_k \bF_k^T-\nabla \bv^2:\bF_k \bF_k^T\right)\varphi.
    \end{split}
\end{equation}
The sequence $\{|\nabla \bv^1_k: G_m(|\bF_k|^2)\bF_k \bF_k^T|\}_{k,m\in \N}$ is bounded in $L^1(Q_T)$ thanks to \eqref{vk1weak} and \eqref{est2}. Moreover, due to the properties of $G$ the sequence is bounded uniformly with respect to $k\in \N$ in $L^2(Q_T)$ and  it is nondecreasing in $m$ almost everywhere in $Q_T$. Hence, for any fixed $m\in \N$ and a suitable not relabeled subsequence, there exists
$\overline{|\nabla \bv^1: G_m(|\bF|^2) \bF \bF^T|}$ satisfying as $k\to\infty$
\begin{equation}
    |\nabla \bv_k^1: G_m(|\bF_k|^2)\bF_k \bF_k^T| \rightharpoonup \overline{|\nabla \bv^1: G_m(|\bF|^2) \bF \bF^T|}\quad \mbox{ weakly in } L^2(Q_T)
\end{equation}
and ($l, n\in \N$)
$$\overline{|\nabla \bv^1: G_l(|\bF|^2) \bF \bF^T|}\leq \overline{|\nabla \bv^1: G_n(|\bF|^2) \bF \bF^T|} \quad \mbox{ a.e. in } Q_T \quad\mbox{ whenever } l\leq n.$$
By the monotone convergence theorem we then obtain the existence of the element $\overline{|\nabla \bv^1: G_\infty(|\bF|^2) \bF \bF^T|}$ satisfying as $m\to \infty$
$$\overline{|\nabla \bv^1: G_m(|\bF|^2) \bF \bF^T|}\to\overline{|\nabla \bv^1: G_\infty(|\bF|^2) \bF \bF^T|}\quad \mbox{ strongly in } L^1(Q_T).$$
The modulus of the limit as $k\to\infty$ and $m\to\infty$ of the first integral in \eqref{L12} is thus less than or equal to 
\begin{equation}\label{estL12}
    \int_{Q_T\setminus E_j} \left(\overline{|\nabla \bv^1: G_\infty(|\bF|^2) \bF \bF^T|}+|\nabla \bv^1:\overline{\bF \bF^T}|\right)\varphi.
\end{equation}
Since the integrand in \eqref{estL12} belongs to $L^1(Q_T)$ and $|Q_T\setminus E_j|\to 0$ as $j\to \infty$, we observe that the limit as $k\to \infty$, $m\to \infty$ and $j\to \infty$ of the first integral in \eqref{L12} is equal to zero. Moreover, from \eqref{Fkweak}, \eqref{vk2weak}, \eqref{nablavk2strong},   \eqref{Gmn1} and the definition of $G_m$, one can deduce that
\begin{equation}
    \label{L12concl}
    \begin{split}
    \lim_{m\to \infty} \lim_{k\to \infty}\int_{Q_T} \left(\nabla \bv^2_k: G_m(|\bF_k|^2)\bF_k \bF_k^T-\nabla \bv^2:\bF_k \bF_k^T\right)\varphi\\ \leq \lim_{m\to\infty}\left(2m \lim_{k\to \infty} \int_{Q_T} |\nabla \bv^2_k-\nabla \bv^2| \varphi\right) = 0,
    \end{split}
\end{equation}
i.e. the limit as $k\to \infty$ and $m\to \infty$ of the second integral in \eqref{L12} is also equal to zero. To conclude \eqref{compFkeyrep} it remains to prove that
$$\lim_{j\to \infty} L_3^2=0.$$
However, since $L_3^2$ is determined by \eqref{L3} and \eqref{L13dec}, it follows straightforwardly from \eqref{Fkweak}, \eqref{vk1weak}, \eqref{vk2weak}, \eqref{nablavk2strong} and the property $|Q_T\setminus E_j|\to 0$ as $j\to \infty$. Hence \eqref{compFkeyrep} holds true. Note that the procedure how to prove \eqref{compFkeyrep} is in detail described in \cite{BuMaLo24}.

To conclude, since $l_2=L_1+L_2+L_3$, we see, using \eqref{estL2}, \eqref{L13dec}, \eqref{compFkkey} and \eqref{compFkeyrep}, that \eqref{step2compFsigma} holds with 
$$\tilde{L}:= 1+2|\nabla \bv|+R|\bF|^2\in L^2(Q_T),$$
where $R:=\sup_{\theta>0} g(\theta)$. This completes the proof of the compactness of the sequence $\{\bF_k\}_{k\in \N}$ in $(L^2(Q_T))^{3\times 3}$.
\end{proof}

\section{Existence of weak solutions}\label{Sec4}
This section is devoted to the proof of the main result of the paper, namely the global-in-time existence result for large data. Before presenting the proof, we first state the result precisely. We consider here only the three-dimensional case; however, the argument extends to other dimensions with only minor modifications.

\begin{thm}\label{THM2}
   Let the material parameters satisfy \eqref{estnu}--\eqref{growthg01}. Assume that  
\begin{equation}\label{initCCCA}
    \begin{aligned}
       \bv_0\in L^{2}_{\bn, \diver},\quad \bF_0 \in  L^2(\Omega)^{3\times 3}, \quad \theta_0\in L^1(\Omega), \quad \det \bF_0 \in L^1(\Omega)
    \end{aligned}
\end{equation}
and that $\det \bF_0 >0$ almost everywhere in $\Omega$. Then there exists $(\bv,\theta,\bF)$ satisfying \eqref{vincompT}--\eqref{aatt}.
\end{thm}

In the present section, we shall make use of the already proven weak stability result and show the approximation procedure, which gives the existence of weak solutions. We want to remark, that the proof is already quite extensive with many approximating steps, therefore as to not make it even longer we shall assume some preliminary knowledge of the reader in the construction of approximating schemes and obtaining the energy estimates. Without going into much details we will often use the renormalization theory introduced by DiPerna and Lions \cite{DiPerna1989}, to be able to properly test equations even without the proper regularity of the test function. We will also often assume that the reader understands how to apply a mix of weak and strong convergence results to converge in the weak formulations of equations in linear and non-linear terms. Let us begin our discussion with the introduction of an approximation procedure. Denote by $\varepsilon_1, \varepsilon_2, \varepsilon_3, \varepsilon_4, \varepsilon_5,\varepsilon_6, \varepsilon_7\in (0, 1)$ some small constants, that will converge to $0^+$ throughout the approximation. Fix $\{\omega_i\}_{i\in\N}$, $\{\Gamma_i\}_{i\in\N}$ and $\{\beta_i\}_{i\in\N}$ as the standard Galerkin basis', which are orthonormal in $L^2$ and orthogonal in $W^{s, 2}_{\bn, \diver}(\Omega, \R^3)$, $W^{s,2}(\Omega)$, $W^{s,2}(\Omega; \R^{3\times 3})$ respectively, for $s$ large enough so that
\begin{align*}
    W^{s-1, 2}\hookrightarrow L^\infty .
\end{align*}
Note that we also fix the appropriate boundary conditions for them which are related to the variables $\bv, e, \mathbb{F}$, that is
$$
\bv\cdot\bn = 0, \quad\nabla_x\bF\cdot\bn = 0,\quad\text{ and }\nabla_xe\cdot\bn = 0.
$$
Then, define
\begin{align*}
    \bv^l := \sum_{i = 1}^l a_i(t)\omega_i(x),\quad e^m := \sum_{i = 1}^m b_i(t)\Gamma_i(x),\quad \bF^n := \sum_{i = 1}^n c_i(t)\beta_i(x),
\end{align*}
with initial conditions $\bv^l_0 := P^l\bv_0$, $e^m_0 = P^m(e^{\varepsilon_1, \varepsilon_6}_0)$, $\bF^n_0 = P^n(T^{\det}_{\varepsilon_5}(T_{\varepsilon_3}(\bF_0))\star\varphi_{\varepsilon_7})$, where $P^i$ is a projection onto first $i$ vectors of a respective Galerkin basis, $e^{\varepsilon_1, \varepsilon_6}_0$ is a cut-off from below for $e_0$, that is
\begin{align*}
    e_0^{\varepsilon_1, \varepsilon_6}:=\left\{\begin{array}{ll}
         e_0,\text{ whenever }e_0 \geq \min\{\varepsilon_1, \varepsilon_6\}\\
         1, \text{ otherwise}, 
    \end{array}\right.
\end{align*}
$T^{\det}_{\varepsilon_5}$ is a cut-off from below over the determinant, that is
\begin{align*}
    T^{\det}_{\varepsilon_5}(F) := F\,\mathbf{1}_{\{\det F \geq \varepsilon_5\}} + \bI\,\mathbf{1}_{\{\det F < \varepsilon_5\}},
\end{align*}
$T_{\varepsilon_3}$ is the typical truncation at the level $2/\varepsilon_3$ adapted to matrices, that is 
\begin{align*}
    T_{\varepsilon_3}(F) := F\,\mathbf{1}_{\{|F|\leq 2/\varepsilon_3\}} + \bI\,\mathbf{1}_{\{|F| > 2/\varepsilon_3\}},
\end{align*}
$\varphi_{\varepsilon_7}$ is the standard mollification kernel, i.e. $\varphi$ is a smooth, non-negative function, compactly supported in a ball of radius one and fulfills $\int_{\mathbb{R}^d} \varphi(x) \diff x = 1$. Then, we set $\varphi_{\varepsilon_7}(x) = \frac{1}{{\varepsilon_7}^d} \varphi\left(\frac{x}{\varepsilon_7}\right)$. Moreover, to define an appropriate approximating system let
\begin{align*}
    g_{\varepsilon_1}(z) := \left\{\begin{array}{ll}
          \text{linear} &\text{on }(0, \varepsilon_1), \\
          \text{smooth} &\text{on }[\varepsilon_1, 2\varepsilon_1],\\
          g(z) &\text{on }(2\varepsilon_1, +\infty),
    \end{array}\right.
\end{align*}
satisfying the condition $g_{\varepsilon_1}(0) = 0$. With this, define
\begin{align}\label{eq:step0_eq_for_theta}
    e^*_{\varepsilon_1, \varepsilon_2}(\theta, F) = \theta + \underbrace{(g_{\varepsilon_1}(\theta) - \theta g'_{\varepsilon_1}(\theta))}_{\text{extended by }0\text{ for }\theta < 0}(\mathrm{tr}(F\,F^T) - 3 - \ln((\det(F\, F^T) - \varepsilon_2)_+ + \varepsilon_2)),
\end{align}
and by $\theta^*_{\varepsilon_1, \varepsilon_2}(e, F)$ denote the inverse of the above function with respect to $\theta$, i. e.:
\begin{align*}
    \theta^*_{\varepsilon_1, \varepsilon_2}(e^*_{\varepsilon_1, \varepsilon_2}(\theta, F), F) \equiv \theta .
\end{align*}
Lastly, we define as well
\begin{align*}
    \Lambda_{\varepsilon_3}(s) := \left\{\begin{array}{ll}
         1, & |s| \leq 1/\varepsilon_3, \\
         0, & |s| \geq 2/\varepsilon_3, \\
         \text{smooth, } &\text{otherwise},
    \end{array}\right.
\end{align*}
and
\begin{align*}
    \bT_{\varepsilon_1, \varepsilon_3, \varepsilon_6} := 2\Lambda_{\varepsilon_3}(|F|)g_{\varepsilon_1}(\theta)F\,F^T\frac{(\theta - \varepsilon_6)_+}{\theta} + \nu(\theta) \mathbb{D}\bv.
\end{align*}
Having all the definitions we are ready to state our approximating system
\begin{align}
    &\partial_t \bv^l + \diver_x (\Lambda_{\varepsilon_3}(|\bv^l|^2)\bv^l\otimes \bv^l) - \diver_x \bT_{\varepsilon_1, \varepsilon_3, \varepsilon_6}(\theta^*_{\varepsilon_1,\varepsilon_2}(e^m, \bF^n), \bF^n,  \bD\bv^l) = 0, \label{eq:galerkin_velocity}\\
    &\partial_t \bF^n + \Diver_x(\bF^n\otimes\bv^l)- \Lambda_{\varepsilon_3}(|\bF^n|)\nabla_x \bv^l \bF^n\frac{(\theta^*_{\varepsilon_1, \varepsilon_2}(e^m, \bF^n) - \varepsilon_6)_+}{\theta^*_{\varepsilon_1, \varepsilon_2}(e^m, \bF^n)}\nonumber\\
    &\qquad\,\,\,- \varepsilon_4\Delta_x \bF^n + \frac{\tau(\theta^*_{\varepsilon_1, \varepsilon_2}(e^m, \bF^n))}{2}\frac{(\det\bF^n - \varepsilon_5)_+}{\det\bF^n}(\bF^n\,\bF^{n, T}\, \bF^n - \bF^n) = 0,\\
    &\partial_te^m + \diver_x(e^m \bv^l) - \varepsilon_7\Delta_x e^m - \diver_x(\kappa(\theta^*_{\varepsilon_1, \varepsilon_2}(e^m, \bF^n))\nabla_x \theta^*_{\varepsilon_1, \varepsilon_2}(e^m, \bF^n))\nonumber\\
    &\qquad\qquad\qquad\qquad\qquad\qquad - \bT_{\varepsilon_1, \varepsilon_3, \varepsilon_6}(\theta^*_{\varepsilon_1, \varepsilon_2}(e^m, \bF^n), \bF^n, \bD\bv^l) : \bD\bv^l = 0.\label{eq:galerkin_internal_energy}
\end{align}
Note that the equations above should be understood as the projection onto the first, respectively $l, n, m$ vectors of the Galerkin basis, which is a standard Galerkin formulation. For the sake of simplicity we have written them in the differential form. For brevity we note that the existence of the solutions is a standard procedure involving the Carath\'{e}odory theorem of existence. Moreover, every function actually depends on all of the parameters $l, n, m, \varepsilon_1,..., \varepsilon_7$, but to avoid the overload of indices we shall only keep the index related to the Galerkin approximation (when applicable), and the one with respect to which, we are currently considering the convergence.\\

\noindent\underline{Step 1: convergence with $m\to +\infty$.} It is easy to see that since we keep the Galerkin index for $\bF^n_m$ and $\bv^l_m$ fixed, both of those sequences are bounded in $W^{1,\infty}(0, T; W^{1, \infty}(\Omega))$, for interested readers we also refer to \cite[Section 4, under (4.16)]{bathory2024coupling}. Hence, by the Banach--Alaoglu theorem and the Aubin--Lions lemma (up to the subsequence which we do not relabel),
\begin{align*}
    \partial_t \bv^l_m &\rightharpoonup^* \partial_t \bv^l, &&\text{ weakly* in }L^\infty(0, T; W^{1,\infty}_{\mathbf{n},\diver}(\Omega)),\\
    \partial_t \bF^n_m &\rightharpoonup^* \partial_t \bF^n, &&\text{ weakly* in }L^\infty(0, T; (W^{1,\infty}(\Omega))^{3\times 3}),\\
    \bv^l_m &\rightarrow \bv^l, &&\text{ strongly in }C([0, T]; W^{1, \infty}_{\mathbf{n},\diver}(\Omega)),\\
    \bF^n_m &\rightarrow \bF^n, &&\text{ strongly in }C([0, T]; (W^{1,\infty}(\Omega))^{3\times 3}),\\
    \bv^l_m &\rightarrow \bv^l, &&\text{ a.e. in }Q_T,\\
    \bF^n_m &\rightarrow \bF^n, &&\text{ a.e. in }Q_T.
\end{align*}
Thus, we move forward to consider the convergence of $e^m$. Since we work with Galerkin approximation we can test \eqref{eq:galerkin_internal_energy} by $e^m$, and obtain the following energy equality
\begin{equation}\label{ineq:step1_energy}
\begin{split}
    \frac{d}{dt}\int_{\Omega}|e^m(t)|^2&\diff x + \varepsilon_7\int_{\Omega}|\nabla_x e^m|^2\diff x\\
    &+ \int_{\Omega}\kappa\nabla_x\theta^*_{\varepsilon_1, \varepsilon_2}\cdot \nabla_x e^m\diff x = \int_{\Omega}(\bT_{\varepsilon_1, ..., \varepsilon_6} : \bD \bv^l_m)e^m\diff x.
\end{split}
\end{equation}
Due to the uniform bounds on $\bF^n_m, \bv^l_m$, we deduce a straightforward bound using H\"{o}lder's inequality
\begin{align}\label{ineq:step1_weak_ineq_1}
\left|\int_{\Omega}(\bT_{\varepsilon_1, ..., \varepsilon_6} : \bD \bv^l_m)e^m\diff x \right|\leq \|\bT_{\varepsilon_1, ..., \varepsilon_6} : \bD\bv^l_m\|_{L^2(\Omega)}\|e^m\|_{L^2(\Omega)}.
\end{align}
Let us move forward to a more problematic term involving the gradient of the function $\theta^*_{\varepsilon_1, \varepsilon_2}$. First, note that
\begin{align}\label{eq:step1_grad_theta}
\nabla_x\theta^*_{\varepsilon_1, \varepsilon_2}(e^m, \bF^n_m) = \nabla_e\theta^*_{\varepsilon_1, \varepsilon_2}(e^m, \bF^n_m)\nabla_xe^m + \nabla_F\theta^*_{\varepsilon_1, \varepsilon_2}(e^m, \bF^n_m) \nabla_x\bF^n_m.
\end{align}
We will focus on both the terms separately. By the Inverse Function theorem
\begin{align*}
    &\nabla_e\theta^*_{\varepsilon_1, \varepsilon_2}(e^m, \bF^n_m) = (\nabla_\theta e^*_{\varepsilon_1, \varepsilon_2}(\theta^*_{\varepsilon_1, \varepsilon_2}(e^m, \bF^n_m), \bF^n_m))^{-1}\\
    &= \frac{1}{1 - \theta^*_{\varepsilon_1, \varepsilon_2}(e^m, \bF^n_m)g''_{\varepsilon_1}(\theta_{\varepsilon_1, \varepsilon_2}^*(e^m, \bF^n_m))\tilde{\psi}_{\varepsilon_2}(\bF^n_m\bF^{n, T}_m)}, 
\end{align*}
where
\begin{align}\label{eq:modified_psi_varepsilon2}
\tilde{\psi}_{\varepsilon_2}(\bF^n_m \bF^{n, T}_m) = \mathrm{tr}(\bF^n_m\,\bF^{n,T}_m) - 3 - \ln((\det(\bF^n_m\, \bF^{n, T}_m) - \varepsilon_2)_+ + \varepsilon_2).
\end{align}
Note that whenever $\det(\bF^n_m\, \bF^{n, T}_m) \leq \varepsilon_2$, then (fixing $\varepsilon_2 > 0$ small enough)
\begin{align*}
    &\mathrm{tr}(\bF^n_m\,\bF^{n,T}_m) - 3 - \ln((\det(\bF^n_m\, \bF^{n, T}_m) - \varepsilon_2)_+ + \varepsilon_2))\\
    &= \mathrm{tr}(\bF^n_m\,\bF^{n,T}_m) - 3 - \ln\varepsilon_2 \geq 0,
\end{align*}
and for $\det(\bF^n_m\, \bF^{n, T}_m) > \varepsilon_2$ we can use a classical inequality true for any symmetric matrix (cf. \cite[Lemma A.3]{bathory2024coupling})
\begin{align*}
    &\mathrm{tr}(\bF^n_m\,\bF^{n,T}_m) - 3 - \ln((\det(\bF^n_m\, \bF^{n, T}_m) - \varepsilon_2)_+ + \varepsilon_2))\\
    &= \mathrm{tr}(\bF^n_m\,\bF^{n,T}_m) - 3 - \ln\det(\bF^n_m\, \bF^{n, T}_m) \geq 0.
\end{align*}
Moreover by the definition of $g_{\varepsilon_1}$ and the concavity of $g$ we have
$$
g''_{\varepsilon_1}(\theta)\leq 0,\text{ and }g''_{\varepsilon_1}(\theta) = 0,\text{ whenever }\theta <\varepsilon_1.
$$
Hence,
$$
\theta^*_{\varepsilon_1, \varepsilon_2}(e^m, \bF^n_m)g''_{\varepsilon_1}(\theta_{\varepsilon_1, \varepsilon_2}^*(e^m, \bF^n_m)) \leq 0.
$$
Thus,
\begin{align}\label{ineq:step1_weak_ineq_2}
    0\leq \nabla_e\theta^*_{\varepsilon_1, \varepsilon_2}(e^m, \bF^n_m)\leq 1.
\end{align}
For the second term we obtain the equality
\begin{equation}\label{ineq:step1_weak_ineq_3}
\begin{split}
    &\left|\nabla_F\theta^*_{\varepsilon_1, \varepsilon_2}(e^m, \bF^n_m)\right| = \left|\frac{\nabla_F e^*_{\varepsilon_1, \varepsilon_2}(\theta^*_{\varepsilon_1, \varepsilon_2}(e^m, \bF^n_m), \bF^n_m)}{\nabla_\theta e^*_{\varepsilon_1, \varepsilon_2}(\theta^*_{\varepsilon_1, \varepsilon_2}(e^m, \bF^n_m), \bF^n_m)}\right|\\
    &\leq \left|\mathbb{I} - (\bF^n_m\, \bF^{n, T}_m)^{-1}\mathbf{1}_{\det(\bF^n_m\, \bF^{n, T}_m) > \varepsilon_2}\right| \leq 3 + \frac{C}{\varepsilon_2}\|\bF^n_m\|^4_\infty.
\end{split}
\end{equation}
Combining \eqref{ineq:step1_weak_ineq_1}, \eqref{ineq:step1_weak_ineq_2}, \eqref{ineq:step1_weak_ineq_3}, as well as the bounds on $\kappa$ and the known bounds on $\bF^n_m, \bv^l_m$, together with \eqref{ineq:step1_energy} and the Young inequality we obtain
\begin{equation*}
    \begin{split}
    &\frac{d}{dt}\int_{\Omega}|e^m(t)|^2\diff x + \varepsilon_7\int_{\Omega}|\nabla_x e^m|^2\diff x\\
    &\qquad= -\int_{\Omega}\kappa\nabla_x\theta^*_{\varepsilon_1, \varepsilon_2}\cdot \nabla_x e^m\diff x + \int_{\Omega}(\bT_{\varepsilon_1, ..., \varepsilon_6} : \bD \bv^l_m)e^m\diff x\\
    &\leq \int_{\Omega}\kappa\left|\nabla_F\theta^*_{\varepsilon_1, \varepsilon_2}(e^m, \bF^n_m) \nabla_x\bF^n_m\cdot \nabla_x e^m\right|\diff x + \|\bT_{\varepsilon_1, ..., \varepsilon_6} : \bD\bv^l_m\|_{L^2(\Omega)}\|e^m\|_{L^2(\Omega)}\\
    &\leq C(K, \varepsilon_2, \|\bF^n_m\|_\infty)\int_{\Omega}\left|\nabla_x\bF^n_m\cdot \nabla_x e^m\right|\diff x + C(|\Omega|, \|\bF^n_m\|_{L^4}, \|\bv^l_m\|_{W^{1,\infty}})\|e^m\|_{L^2(\Omega)}\\
    &\leq C(K, \varepsilon_2, \|\bF^n_m\|_\infty, \varepsilon_7)\int_{\Omega}|\nabla_x \bF^n_m|^2\diff x + \frac{\varepsilon_7}{2}\int_{\Omega}|\nabla_x e^m|^2\diff x\\
    &\qquad\qquad + C(|\Omega|, \|\bF^n_m\|_{L^4}, \|\bv^l_m\|_{W^{1,\infty}}) + \|e^m\|_{L^2(\Omega)}^2\\
    &\leq C(K, |\Omega|, \varepsilon_2, \|\bF^n_m\|_\infty, \varepsilon_7, \|\bv^l_m\|_{W^{1,\infty}}) + \frac{\varepsilon_7}{2}\int_{\Omega}|\nabla_x e^m|^2\diff x + \|e^m\|_{L^2(\Omega)}^2
    \end{split}
\end{equation*}
The inequality above together with Gr\"{o}nwall's inequality imply readily that $\{e^m\}_{m\in\N}$ is bounded in $L^\infty(0, T; L^2(\Omega))$, and $L^2(0, T; W^{1,2}(\Omega))$. Moreover, one can test \eqref{eq:galerkin_internal_energy} by $\partial_t e^m$ to obtain $\{\partial_t e^m\}_{m\in\N}$ bounded in $L^\infty(0, T; L^2(\Omega))$. Since the calculation are analogous to the made above, we skip them. Thus, by the Banach--Alaoglu theorem as well as the Aubin--Lions lemma we may deduce the existence of $e\in L^\infty(0, T; L^2(\Omega))\cap L^2(0, T; W^{1,2}(\Omega))$ with $\partial_t e\in L^\infty(0, T; L^2(\Omega))$ such that (up to the subsequence, which we do not relabel)
\begin{align*}
    \partial_t e^m &\rightharpoonup^* \partial_t e, &&\text{ weakly* in }L^\infty(0, T; L^2(\Omega)),\\
    e^m &\rightharpoonup^* e, &&\text{ weakly* in }L^\infty(0, T; L^2(\Omega)),\\
    e^m &\rightharpoonup e, &&\text{ weakly in }L^2(0, T; W^{1,2}(\Omega)),\\
    e^m &\rightarrow e, &&\text{ strongly in }L^2(Q_T),\\
    e^m &\rightarrow e, &&\text{ a.e. in }Q_T.
\end{align*} 
All that is left to do is to identify the limit of $\theta^*_{\varepsilon_1, \varepsilon_2}(e^m, \bF^n_m)$. Note that by \eqref{eq:step1_grad_theta}, \eqref{ineq:step1_weak_ineq_2}, \eqref{ineq:step1_weak_ineq_3}, and the already established bounds on $\{\nabla_x e^m\}_{m\in\N}$ and $\{\nabla_x\bF^n_m\}_{m\in\N}$ we know that $\{\nabla_x\theta^*_{\varepsilon_1, \varepsilon_2}(e^m, \bF^n_m)\}_{m\in\N}$ is bounded in $L^2(Q_T)$. Moreover by a similar argument we can see that $\{\partial_t \theta^*_{\varepsilon_1, \varepsilon_2}(e^m, \bF^n_m)\}_{m\in\N}$ is bounded in $L^2(Q_T)$ as well. Using the equation
\begin{align*}
    \theta^*_{\varepsilon_1, \varepsilon_2}(e^m, \bF^n_m) = \theta^*_{\varepsilon_1, \varepsilon_2}(e^m_0, \bF^n_{m, 0}) + \int_0^t\partial_t\theta^*_{\varepsilon_1, \varepsilon_2}(e^m(s, x), \bF^n_m(s, x))\diff s,
\end{align*}
we can deduce that $\{\theta^*_{\varepsilon_1, \varepsilon_2}(e^m, \bF^n_m)\}_{m\in\N}$ is bounded in $L^2(Q_T)$ as long as $\{\theta^*_{\varepsilon_1, \varepsilon_2}(e^m_0, \bF^n_{m, 0})\}_{m\in\N}$ is. By \eqref{ineq:step1_weak_ineq_2}, \eqref{ineq:step1_weak_ineq_3} we know that $\theta^*_{\varepsilon_1, \varepsilon_2}(e^m, \bF^n_m)$ is Lipschitz with respect to its variables. Hence,
\begin{align*}
    |\theta^*_{\varepsilon_1, \varepsilon_2}(e^m_0, \bF^n_{m, 0})| \leq C(1 + |e^m_0| + |\bF^n_{m, 0}|)
\end{align*}
and by the boundedness of $\{e^m_0\}_{m\in\N}$ and $\{\bF^n_{m, 0}\}_{m\in\N}$ we get what is needed. Thus, by the Banach--Alaoglu theorem and the Aubin--Lions lemma we can imply that there exists $\theta\in L^2(0, T; W^{1,2}(\Omega))$ with $\partial_t\theta\in L^2(Q_T)$ such that (up to the subsequence which we do not relabel)
\begin{align*}
    \partial_t\theta^*_{\varepsilon_1, \varepsilon_2}(e^m, \bF^n_m)  &\rightharpoonup^* \partial_t \theta, &&\text{ weakly in }L^2(Q_T),\\
    \theta^*_{\varepsilon_1, \varepsilon_2}(e^m, \bF^n_m) &\rightharpoonup \theta, &&\text{ weakly in }L^2(0, T; W^{1,2}(\Omega)),\\
    \theta^*_{\varepsilon_1, \varepsilon_2}(e^m, \bF^n_m) &\rightarrow \theta, &&\text{ strongly in }L^2(Q_T),\\
    \theta^*_{\varepsilon_1, \varepsilon_2}(e^m, \bF^n_m) &\rightarrow \theta, &&\text{ a.e. in }Q_T.
\end{align*}
At the same time, we can identify $\theta$ in another way. Since $\{e^m\}_{m\in\N}$ and $\{\bF^n_m\}_{m\in\N}$ converge a.e. to their limits and $\theta^*_{\varepsilon_1, \varepsilon_2}$ is continuous with respect to the variables we can see that
\begin{align*}
    \theta(t, x) = \theta^*_{\varepsilon_1, \varepsilon_2}(e, \bF^n).
\end{align*}
Moreover, converging in \eqref{eq:step0_eq_for_theta} grants us
\begin{equation}\label{eq:step1_e_with_theta}
\begin{split}
    e(t, x) &= \theta^*_{\varepsilon_1, \varepsilon_2}(e, \bF^n) + \underbrace{(g_{\varepsilon_1}(\theta^*_{\varepsilon_1, \varepsilon_2}(e, \bF^n)) - \theta^*_{\varepsilon_1, \varepsilon_2}(e, \bF^n) g'_{\varepsilon_1}(\theta^*_{\varepsilon_1, \varepsilon_2}(e, \bF^n)))}_{\text{extended by }0\text{ for }\theta < 0}\tilde{\psi}_{\varepsilon_2}(\bF^n\bF^{n, T})\\
    &= \theta(t, x) + \underbrace{(g_{\varepsilon_1}(\theta(t, x)) - \theta(t, x) g'_{\varepsilon_1}(\theta(t, x)))}_{\text{extended by }0\text{ for }\theta < 0}\tilde{\psi}_{\varepsilon_2}(\bF^n\bF^{n, T}),
\end{split}
\end{equation}
where $\tilde{\psi}_{\varepsilon_2}$ is defined in \eqref{eq:modified_psi_varepsilon2}. Having \eqref{eq:step1_e_with_theta} we will drop $e^*_{\varepsilon_1, \varepsilon_2}$ and $\theta^*_{\varepsilon_1, \varepsilon_2}$ notation, as we introduced it only to obtain this equation in the limit to connect the internal energy and the temperature. Having all of the proven convergences it is easy to see that we may converge in \eqref{eq:galerkin_velocity}-\eqref{eq:galerkin_internal_energy} to obtain
\begin{align}
    &\partial_t \bv^l + \diver_x (\Lambda_{\varepsilon_3}(|\bv^l|^2)\bv^l\otimes \bv^l) - \diver_x \bT_{\varepsilon_1, \varepsilon_3, \varepsilon_6}(\theta, \bF^n,  \bD\bv^l) = 0, \label{eq:step2_velocity}\\
    &\partial_t \bF^n + \Diver_x(\bF^n\otimes\bv^l)- \Lambda_{\varepsilon_3}(|\bF^n|)\nabla_x \bv^l \bF^n\frac{(\theta - \varepsilon_6)_+}{\theta}\nonumber\\
    &\qquad\,\,\,- \varepsilon_4\Delta_x \bF^n + \frac{\tau(\theta)}{2}\frac{(\det\bF^n - \varepsilon_5)_+}{\det\bF^n}(\bF^n\,\bF^{n, T}\, \bF^n - \bF^n) = 0,\label{eq:step2_F}\\
    &\partial_te + \diver_x(e \bv^l) - \varepsilon_7\Delta_x e - \diver_x(\kappa(\theta)\nabla_x \theta)\nonumber\\
    &\qquad\qquad\qquad\qquad\qquad\qquad - \bT_{\varepsilon_1, \varepsilon_3, \varepsilon_6}(\theta, \bF^n, \bD\bv^l) : \bD\bv^l = 0,\label{eq:step2_internal_energy}
\end{align}
where \eqref{eq:step2_velocity} and \eqref{eq:step2_F} are to be understood as projections onto the Galerkin basis, and \eqref{eq:step2_internal_energy} is to be understood in the weak sense in the sense of an operator on a predual space to the time derivative (in this case $\partial_t e\in L^\infty(0, T; L^2(\Omega))$, but we shall refer the reader to the particular convergences in the later steps to verify the predual space, as it changes multiple times).\\

\noindent\underline{Step 2: convergence with $n\to +\infty$.} As in the previous step, it is easy to obtain (sine $l$ is fixed) a subsequence (which we do not relabel) that satisfies
\begin{align*}
    \partial_t \bv^l_n &\rightharpoonup^* \partial_t \bv^l, &&\text{ weakly* in }L^\infty(0, T; W^{1,\infty}_{\mathbf{n},\diver}(\Omega)),\\
    \bv^l_n &\rightarrow \bv^l, &&\text{ strongly in }C([0, T]; W^{1, \infty}_{\mathbf{n},\diver}(\Omega)),\\
    \bv^l_n &\rightarrow \bv^l, &&\text{ a.e. in }Q_T.
\end{align*}
Moreover, as we work with Galerkin projections, we can test \eqref{eq:step2_F} by $\bF^n$ to get $\{\bF^n\}_{n\in\N}$ bounded in $L^\infty(0, T; (L^2(\Omega))^{3\times 3})\cap L^2(0, T; (W^{1,2}(\Omega))^{3\times 3})$. As the argument here is fairly simple as long as $\{\bv^l_n\}_{n\in\N}$ is bounded in $L^\infty(0, T; W^{1,\infty}(\Omega))$, we skip it for the sake of brevity. We will instead focus on showing the bound $\{\bF^n\}_{n\in\N}\subset (L^\infty(Q_T))^{3\times 3}$. To do so let
\begin{align*}
    C(t) := M\,e^{Kt},
\end{align*}
where $K$ comes from \eqref{propest}, and $M > \frac{2}{\varepsilon_3}$ is to be fixed later, and define $\bH^n(t, x):= (|\bF^n(t, x)| - C(t))_+$, $\bG^n(t, x) :=\frac{\bF^n}{|\bF^n|}\bH^n$. Note that $\bG^n$ belongs to $L^2(0, T; (W^{1,2}(\Omega))^{3\times 3})$, which means we may test \eqref{eq:step2_F} by $P^n\bG^n$. Let us treat all the terms appearing there separately. For the time derivative we get 
\begin{align}\label{eq:step2_some_eq}
    \int_{\Omega}\frac{d}{dt}\bF^n:P^n\bG^n\diff x = \int_{\Omega}\frac{d}{dt}\bF^n:\bG^n\diff x = \frac{1}{2}\frac{d}{dt}\int_{\Omega}|\bH^n|^2\diff x + C'(t)\int_{\Omega}\bH^n\diff x.
\end{align}
Moving forward, we obtain
\begin{align*}
    &\int_{\Omega}\bF^n\otimes \bv^l_n:\nabla_xP^n\bG^n\diff x = \int_{\Omega}\bF^n\otimes \bv^l:\nabla_x\bG^n\diff x = \sum_{i, j, k}\int_{\Omega}\bv^l_{n,i}\partial_{x_i}\bF^n_{jk}\bG^n_{j k}\diff x\\
    &= \sum_{i, j, k}\int_{\Omega}\bv^l_{n,i}\partial_{x_i}|\bF^n_{jk}|\bH^n_{j k}\diff x = \frac{1}{2}\sum_{i}\int_{\Omega}\bv^l_{n, i}\partial_{x_i}|\bH^n|^2\diff x= 0,
\end{align*}
and as $\Lambda_3(|\bF^n|) = 0$, whenever $|\bF^n| > 2/\varepsilon_3$
\begin{equation*}
\begin{split}
    &\left|\int_{\Omega}\Lambda_{\varepsilon_3}(|\bF^n|)\nabla_x \bv^l_n \bF^n\frac{(\theta - \varepsilon_6)_+}{\theta}:P^n\bG^n\diff x\right|\\
    & \quad= \left|\int_{\Omega}\Lambda_{\varepsilon_3}(|\bF^n|)\nabla_x \bv^l_n \bF^n\frac{(\theta - \varepsilon_6)_+}{\theta}:\bG^n\diff x\right|= 0,
\end{split}
\end{equation*}
as well as
\begin{equation*}
    \begin{split}
        &\int_{\Omega}\nabla_x\bF^n : \nabla_xP^n\bG^n\diff x = \int_{\Omega}\nabla_x\bF^n:\nabla_x\bG^n\diff x\\
        &= \int_{\Omega}\left(1 - \frac{C(t)}{|\bF^n|}\right)_+|\nabla_x\bF^n|^2\diff x + \int_{\Omega}\mathbf{1}_{\{|\bF^n| \geq C(t)\}}\frac{C(t)}{|\bF^n|^3}\sum_{i}|\bF^n:\partial_{x_i}\bF^n|^2\diff x \geq 0
    \end{split}
\end{equation*}
Furthermore
\begin{equation*}
    \begin{split}
        &\int_\Omega \frac{\tau(\theta^n)}{2}\frac{(\det\bF^n - \varepsilon_5)_+}{\det\bF^n}\bF^n\bF^{n, T}\bF^n : P^n\bG^n \diff x\\
        &\qquad\qquad= \int_{\Omega}\frac{\tau(\theta^n)}{2}\frac{(\det\bF^n - \varepsilon_5)_+}{\det\bF^n}\frac{|\bF^n\bF^{n, T}|^2}{|\bF^n|}\bH^n\diff x\geq 0,
    \end{split}
\end{equation*}
and using \eqref{esttau}
\begin{equation*}
    \begin{split}
        &\int_{\Omega}\frac{\tau(\theta^n)}{2}\frac{(\det\bF^n - \varepsilon_5)_+}{\det\bF^n}\bF^n : P^n\bG^n\diff x = \int_{\Omega}\frac{\tau(\theta^n)}{2}\frac{(\det\bF^n - \varepsilon_5)_+}{\det\bF^n}|\bF^n|\bH^n\diff x\\
        &\leq K\int_\Omega|\bH^n|^2\diff x + KC(t)\int_{\Omega}\bH^n\diff x.
    \end{split}
\end{equation*}
Combining all of the above, we obtain an inequality
\begin{equation*}
    \begin{split}
        &\frac{1}{2}\frac{d}{dt}\int_{\Omega}|\bH^n|^2\diff x + C'(t)\int_{\Omega}\bH^n\diff x\\
        &\leq K\int_\Omega|\bH^n|^2\diff x + KC(t)\int_{\Omega}\bH^n\diff x\\
        &\leq .K\int_\Omega|\bH^n|^2\diff x + MKC(t)\int_{\Omega}\bH^n\diff x
    \end{split}
\end{equation*}
Since $C'(t) =  MKC(t)$ the linear terms cancel and we are left with
\begin{align*}
    \frac{1}{2}\frac{d}{dt}\int_{\Omega}|\bH^n|^2\diff x\leq K\int_\Omega|\bH^n|^2\diff x,
\end{align*}
which by Gr\"{o}nwall's inequality implies
\begin{align*}
    \int_{\Omega}|\bH^n|^2\diff x \leq e^{2TK}\int_{\Omega}|\bH^n_0|^2\diff x.
\end{align*}
Note that since 
\begin{align*}
&\|\bF^n_0\|_\infty = \|P^n(T^{\det}_{\varepsilon_5}(T_{\varepsilon_3}(\bF_0))\star\varphi_{\varepsilon_7})\|_\infty\\
&\leq \|P^n(T^{\det}_{\varepsilon_5}(T_{\varepsilon_3}(\bF_0))\star\varphi_{\varepsilon_7})\|_{W^{s, 2}} \leq \|T^{\det}_{\varepsilon_5}(T_{\varepsilon_3}(\bF_0))\star\varphi_{\varepsilon_7}\|_{W^{s, 2}},
\end{align*}
by defining $M := \max\{2/\varepsilon_3, \|T^{\det}_{\varepsilon_5}(T_{\varepsilon_3}(\bF_0))\star\varphi_{\varepsilon_7}\|_{W^{s, 2}}\}$, we know that $\bH^n_0 = 0$. Thus, we deduce
$$
\bH^n \equiv 0,\text{ a.e. in }Q_T.
$$
Hence,
\begin{align}\label{ineq:step2_linfty_bound}
|\bF^n(t, x)| \leq C(t) = M\,e^{Kt},
\end{align}
and we have what we postulated. As a comment we want to mention that a similar bound will hold for a weak formulation of $\bF$ in further arguments, as without the projection $P^n$ one can simply utilize the Young convolution inequality to deduce 
$$
\|\bF_0\|_\infty = \|T^{\det}_{\varepsilon_5}(T_{\varepsilon_3}(\bF_0))
\star\varphi_{\varepsilon_7}\|_\infty \leq \|\varphi_{\varepsilon_7}\|_1\|T^{\det}_{\varepsilon_5}(T_{\varepsilon_3}(\bF_0))\|_\infty \leq \frac{2}{\varepsilon_3}.
$$
In consequence, one will obtain a similar bound but with a constant
$$
\frac{2}{\varepsilon_3}\,e^{Kt}.
$$
Thus, whenever $\varepsilon_3$ is fixed, we can utilize the bound in $L^\infty$ on any sequence $\bF$. Going back to the main argument, we can obtain in a similar manner as in the proof of Corollary \ref{corpropest} that $\{\partial_t\bF^n\}_{n\in\N}$ is bounded in $L^{\infty}(0, T; ((W^{1,2}(\Omega))^{3\times 3})^*)$. Thus, by the Banach--Alaoglu theorem and the Aubin--Lions lemma we can find a subsequence (which we do not relabel) for which
\begin{align*}
    \partial_t \bF^n &\rightharpoonup^* \partial_t \bF, &&\text{ weakly* in }L^{\infty}(0, T; ((W^{1,2}(\Omega))^{3\times 3})^*),\\
    \bF^n &\rightharpoonup \bF, &&\text{ weakly in }L^2(0, T; (W^{1,2}(\Omega))^{3\times 3}),\\
    \bF^n &\rightharpoonup^* \bF, &&\text{ weakly* in }(L^\infty(Q_T))^{3\times 3},\\
    \bF^n &\rightarrow \bF, &&\text{ strongly in }(L^q(Q_T))^{3\times 3},\ q< +\infty,\\
    \bF^n &\rightarrow \bF, &&\text{ a.e. in }Q_T.
\end{align*}
To move forward, note that the bounds established for $e^n$ in Step 1 are independent of $n$. The only one that does not hold is $\{\partial_t e^n\}_{n\in\N}$ bounded in $L^\infty(0, T; L^2(\Omega))$. Here, using the equation \eqref{eq:step2_internal_energy} one deduces that $\{\partial_t e^n\}_{n\in\N}$ is bounded in $L^2(0, T; (W^{1,2}(\Omega))^*)$. Thus, similarly as before we imply the convergences
\begin{align*}
    \partial_t e^n &\rightharpoonup^* \partial_t e, &&\text{ weakly* in } L^2(0, T; (W^{1,2}(\Omega))^*),\\
    e^n &\rightharpoonup^* e, &&\text{ weakly* in }L^\infty(0, T; L^2(\Omega)),\\
    e^n &\rightharpoonup e, &&\text{ weakly in }L^2(0, T; W^{1,2}(\Omega)),\\
    e^n &\rightarrow e, &&\text{ strongly in }L^2(Q_T),\\
    e^n &\rightarrow e, &&\text{ a.e. in }Q_T.
\end{align*}
Analogously
\begin{align*}
    \partial_t \theta^n &\rightharpoonup^* \partial_t \theta, &&\text{ weakly* in } L^2(0, T; (W^{1,2}(\Omega))^*),\\
    \theta^n &\rightharpoonup \theta, &&\text{ weakly in }L^2(0, T; W^{1,2}(\Omega)),\\
    \theta^n &\rightarrow \theta, &&\text{ strongly in }L^2(Q_T),\\
    \theta^n &\rightarrow \theta, &&\text{ a.e. in }Q_T.
\end{align*}
All of the convergences allow us to converge in \eqref{eq:step2_velocity}--\eqref{eq:step2_internal_energy} and obtain
\begin{align}
    &\partial_t \bv^l + \diver_x (\Lambda_{\varepsilon_3}(|\bv^l|^2)\bv^l\otimes \bv^l) - \diver_x \bT_{\varepsilon_1, \varepsilon_3, \varepsilon_6}(\theta, \bF,  \bD\bv^l) = 0, \label{eq:step3_velocity}\\
    &\partial_t \bF + \Diver_x(\bF\otimes\bv^l)- \Lambda_{\varepsilon_3}(|\bF|)\nabla_x \bv^l \bF\frac{(\theta - \varepsilon_6)_+}{\theta}\nonumber\\
    &\qquad\,\,\,- \varepsilon_4\Delta_x \bF + \frac{\tau(\theta)}{2}\frac{(\det\bF - \varepsilon_5)_+}{\det\bF}(\bF\,\bF^{T}\, \bF - \bF) = 0,\label{eq:step3_F}\\
    &\partial_te + \diver_x(e \bv^l) - \varepsilon_7\Delta_x e - \diver_x(\kappa(\theta)\nabla_x \theta)\nonumber\\
    &\qquad\qquad\qquad\qquad\qquad\qquad - \bT_{\varepsilon_1, \varepsilon_3, \varepsilon_6}(\theta, \bF, \bD\bv^l) : \bD\bv^l = 0,\label{eq:step3_internal_energy}
\end{align}
where \eqref{eq:step3_velocity} is to be understood as projections onto the Galerkin basis, and \eqref{eq:step3_F}, \eqref{eq:step3_internal_energy} are to be understood in the weak sense in the sense of an operator on a predual space to the time derivative. Note that from \eqref{eq:step1_e_with_theta} we keep
\begin{equation}\label{eq:step3_e_with_theta}
\begin{split}
    e(t, x) = \theta(t, x) + \underbrace{(g_{\varepsilon_1}(\theta(t, x)) - \theta(t, x) g'_{\varepsilon_1}(\theta(t, x)))}_{\text{extended by }0\text{ for }\theta < 0}\tilde{\psi}_{\varepsilon_2}(\bF\bF^{T}),
\end{split}
\end{equation}
for $\tilde{\psi}_{\varepsilon_2}$ defined as in \eqref{eq:modified_psi_varepsilon2}.\\

\noindent\underline{Step 3: $\theta \geq\min\{\varepsilon_6, \varepsilon_1\}$.} Here, note that whenever $\theta < \min\{\varepsilon_6, \varepsilon_1\}$ by \eqref{eq:step3_e_with_theta} we know that $e = \theta$. Now test \eqref{eq:step3_internal_energy} by $-\mathbf{1}_{(0, \tau)}(\theta - \min\{\varepsilon_6, \varepsilon_1\})_-$. We can do so only formally but we skip the approximation procedure for the sake of brevity. We get
\begin{equation*}
    \begin{split}
        &\|(\theta - \min\{\varepsilon_6, \varepsilon_1\})_-(\tau)\|_{2}^2 - \|(\theta - \min\{\varepsilon_6, \varepsilon_1\})_-(0)\|_{2}^2\\
        &\qquad\qquad+ \int_0^\tau\int_\Omega(\varepsilon_7 + \kappa(\theta))|\nabla_x(\theta-\min\{\varepsilon_6, \varepsilon_1\})_-|^2\diff x\diff t\\
        &=-\int_0^\tau\int_\Omega\bT_{\varepsilon_1, \varepsilon_3, \varepsilon_6}(\theta, \bF, \bD\bv^l) : \bD\bv^l(\theta-\min\{\varepsilon_6, \varepsilon_1\})_-\diff t\diff x \leq 0.
    \end{split}
\end{equation*}
As by definition $e_0^n \geq \min\{\varepsilon_6, \varepsilon_1\}$ (again by \eqref{eq:step3_e_with_theta} if there would be points at which $\theta_0^n < \min\{\varepsilon_6, \varepsilon_1\}$, then it would follow that $e_0^n<\min\{\varepsilon_6, \varepsilon_1\}$ as well), we deduce
$$
\|(\theta - \min\{\varepsilon_6, \varepsilon_1\})_-(\tau)\|_{2} = 0,
$$
and in consequence $\theta \geq \min\{\varepsilon_6, \varepsilon_1\}$.\\

\noindent\underline{Step 4: convergence with $\varepsilon_7\to 0^+$.} As in the previous step, we have
\begin{align*}
    \partial_t \bv^l_{\varepsilon_7} &\rightharpoonup^* \partial_t \bv^l, &&\text{ weakly* in }L^\infty(0, T; W^{1,\infty}_{\mathbf{n},\diver}(\Omega)),\\
    \bv^l_{\varepsilon_7} &\rightarrow \bv^l, &&\text{ strongly in }C([0, T]; W^{1, \infty}_{\mathbf{n},\diver}(\Omega)),\\
    \bv^l_{\varepsilon_7} &\rightarrow \bv^l, &&\text{ a.e. in }Q_T.
\end{align*}
Moreover, it is easy to see that \eqref{eq:step3_F} can be tested by $\bF^{\varepsilon_7}$ (the only slight issue is the regularity in time, but one can mollify in the time variable, we skip the classical argument for the sake of brevity), which means that we again obtain the bounds of $\{\bF^{\varepsilon_7}\}_{\varepsilon_7 > 0}$ in $L^\infty(0, T; (L^2(\Omega))^{3\times 3})\cap L^2(0, T; (W^{1,2}(\Omega))^{3\times 3})$. To see that this sequence is bounded in $(L^\infty(Q_T))^{3\times 3}$ one can repeat the argument given from \eqref{eq:step2_some_eq} to \eqref{ineq:step2_linfty_bound}, see also comments under \eqref{ineq:step2_linfty_bound}. Hence,  
\begin{align*}
    \partial_t \bF^{\varepsilon_7} &\rightharpoonup^* \partial_t \bF, &&\text{ weakly* in }L^{\infty}(0, T; ((W^{1,2}(\Omega))^{3\times 3})^*),\\
    \bF^{\varepsilon_7} &\rightharpoonup \bF, &&\text{ weakly in }L^2(0, T; (W^{1,2}(\Omega))^{3\times 3}),\\
    \bF^{\varepsilon_7} &\rightharpoonup^* \bF, &&\text{ weakly* in }(L^\infty(Q_T))^{3\times 3},\\
    \bF^{\varepsilon_7} &\rightarrow \bF, &&\text{ strongly in }(L^q(Q_T))^{3\times 3},\ q< +\infty,\\
    \bF^{\varepsilon_7} &\rightarrow \bF, &&\text{ a.e. in }Q_T.
\end{align*}
To move forward, we will obtain proper bounds on $\{\theta^{\varepsilon_7}\}_{\varepsilon_7 > 0}$. First, test \eqref{eq:step3_internal_energy} by $1$, which gives us $\{e^{\varepsilon_7}\}_{\varepsilon_7>0}$ bounded in $L^\infty(0, T; L^1(\Omega))$, thus from \eqref{eq:step3_e_with_theta}, the non-negativity of the function $\theta\mapsto g_{\varepsilon_1}(\theta) - \theta g'_{\varepsilon_1}(\theta)$, and non-negativity of $\tilde{\psi}_{\varepsilon_2}$, we obtain that the same bound is true for $\{\theta^{\varepsilon_7}\}_{\varepsilon_7 > 0}$. Furthermore, by Step 3 $\frac{1}{\theta}$ is a proper test function, hence we can follow Proposition \ref{otherform} to obtain the analogue of \eqref{evetalambda}, i.e.
\begin{equation}\label{eq:step4_lambda_entr_equality}
    \begin{split}
        &\frac{d}{dt}\int_{\Omega}\eta^{\varepsilon_7}_\lambda\diff x + \varepsilon_7\int_{\Omega}\nabla_x e^{\varepsilon_7}\cdot\nabla_x(\theta^{\varepsilon_7})^{\lambda-1}\diff x\\
        &\quad- \varepsilon_4\int_{\Omega}(h_\lambda(\theta^{\varepsilon_7}) - g'_{\varepsilon_1}(\theta^{\varepsilon_7})(\theta^{\varepsilon_7})^{\lambda} + g_{\varepsilon_1}(\theta^{\varepsilon_7})(\theta^{\varepsilon_7})^{\lambda -1})\tilde{\psi}'_{\varepsilon_2}(\bB^{\varepsilon_7}):(\nabla_x\bF^{\varepsilon_7} \cdot \nabla_x\bF^{\varepsilon_7, T})\diff x\\
        &\quad- \varepsilon_4\int_{\Omega}\nabla_x(\tilde{\psi}'_{\varepsilon_2}(\bB^{\varepsilon_7})(h_\lambda(\theta^{\varepsilon_7}) - g'_{\varepsilon_1}(\theta^{\varepsilon_7})(\theta^{\varepsilon_7})^{\lambda} + g_{\varepsilon_1}(\theta^{\varepsilon_7})(\theta^{\varepsilon_7})^{\lambda -1})):(\nabla_x\bF^{\varepsilon_7}\bF^{\varepsilon_7, T} + \bF^{\varepsilon_7}\nabla_x\bF^{\varepsilon_7, T})\diff x\\
        &\quad-\int_{\Omega}\tau\frac{(\det\bF^{\varepsilon_7} - \varepsilon_5)_+}{\det\bF^{\varepsilon_7}}((\bB^{\varepsilon_7})^2 - \bB^{\varepsilon_7}):\tilde{\psi}'_{\varepsilon_2}(\bB^{\varepsilon_7})(h_\lambda(\theta^{\varepsilon_7}) - g'_{\varepsilon_1}(\theta^{\varepsilon_7})(\theta^{\varepsilon_7})^{\lambda})\diff x\\
        &\quad+\int_{\Omega}2(\bB^{\varepsilon_7} - \bI):\bD\bv^l_{\varepsilon_7}\Lambda_{\varepsilon_3}(|\bF^{\varepsilon_7}|)\frac{(\theta^{\varepsilon_7} - \varepsilon_6)_+}{\theta^{\varepsilon_7}}(h_\lambda(\theta^{\varepsilon_7}) - g'_{\varepsilon_1}(\theta^{\varepsilon_7})(\theta^{\varepsilon_7})^{\lambda})\diff x\\
        &=(1-\lambda)\int_{\Omega}\kappa\frac{|\nabla_x\theta^{\varepsilon_7}|^2}{(\theta^{\varepsilon_7})^{2 - \lambda}}\diff x + \int_{\Omega}2\nu\frac{|\bD\bv^l_{\varepsilon_7}|^2}{(\theta^{\varepsilon_7})^{2 - \lambda}} + \tau\frac{(\det\bF^{\varepsilon_7} - \varepsilon_5)_+}{\det\bF^{\varepsilon_7}}\frac{g_{\varepsilon_1}(\theta^{\varepsilon_7})((\bB^{\varepsilon_7})^2 - \bB^{\varepsilon_7}):\tilde{\psi}'_{\varepsilon_2}(\bB^{\varepsilon_7})}{(\theta^{\varepsilon_7})^{1 - \lambda}}\diff x,
    \end{split}
\end{equation}
where $\lambda\in (0, 1)$, and $h_\lambda$ is defined in Lemma \ref{suithlambdatheta}. Note that by boundedness of $\bB^{\varepsilon_7}$ and $\bD\bv^l_{\varepsilon_7}$ in $L^\infty(Q_T)$ we have \eqref{entrinequality} for free (in fact, $L^2(Q_T)$ bounds for both of them would be enough), and one can operate the same way as in the proof of \eqref{pom2est3} to bound any terms without $\varepsilon_7, \varepsilon_4$ in front of them. Indeed, without loss of generality we shall assume $\varepsilon_2 < \varepsilon_5^2$ (as we will converge with $\varepsilon_2\to 0^+$ before doing so with $\varepsilon_5$), which implies
\begin{align}\label{eq:step4_helpful_eq}
\frac{(\det\bF^{\varepsilon_7} - \varepsilon_5)_+}{\det\bF^{\varepsilon_7}}((\bB^{\varepsilon_7})^2 - \bB^{\varepsilon_7}):\tilde{\psi}'_{\varepsilon_2}(\bB^{\varepsilon_7}) = \frac{(\det\bF^{\varepsilon_7} - \varepsilon_5)_+}{\det\bF^{\varepsilon_7}}|\bB^{\varepsilon_7} - \bI|^2.
\end{align}
Thus, we may bound
\begin{equation}\label{ineq:step4_ineq_entr_lambda_0}
    \begin{split}
        &\left|\int_{\Omega}\tau\frac{(\det\bF^{\varepsilon_7} - \varepsilon_5)_+}{\det\bF^{\varepsilon_7}}((\bB^{\varepsilon_7})^2 - \bB^{\varepsilon_7}):\tilde{\psi}'_{\varepsilon_2}(\bB^{\varepsilon_7})(h_\lambda(\theta^{\varepsilon_7}) - g'_{\varepsilon_1}(\theta^{\varepsilon_7})(\theta^{\varepsilon_7})^{\lambda})\diff x\right|\\
        &\leq K\|(h_\lambda(\theta^{\varepsilon_7}) - g'_{\varepsilon_1}(\theta^{\varepsilon_7})(\theta^{\varepsilon_7})^{\lambda})\|_\infty\int_\Omega|\bB^{\varepsilon_7} - \bI|^2\diff x,
    \end{split}
\end{equation}
and
\begin{equation}\label{ineq:step4_ineq_entr_lambda_00}
    \begin{split}
        &\left|\int_{\Omega}2(\bB^{\varepsilon_7} - \bI):\bD\bv^{\varepsilon_7}\frac{(\theta^{\varepsilon_7} - \varepsilon_6)_+}{\theta^{\varepsilon_7}}(h_\lambda(\theta^{\varepsilon_7}) - g'_{\varepsilon_1}(\theta^{\varepsilon_7})(\theta^{\varepsilon_7})^{\lambda})\diff x\right|\\
        &\leq \|(h_\lambda(\theta^{\varepsilon_7}) - g'_{\varepsilon_1}(\theta^{\varepsilon_7})(\theta^{\varepsilon_7})^{\lambda})\|_\infty\left(\int_\Omega |\bB^{\varepsilon_7} - \bI|^2\diff x + \int_{\Omega}|\bD\bv^l_{\varepsilon_7}|^2\diff x\right).
    \end{split}
\end{equation}
Now, let us focus on those additional terms. For the first one, using \eqref{eq:step3_e_with_theta} we get
\begin{equation*}
    \begin{split}
        &\varepsilon_7\int_{\Omega}\nabla_x e^{\varepsilon_7}\cdot\nabla_x(\theta^{\varepsilon_7})^{\lambda-1}\diff x = -\varepsilon_7(1 - \lambda)\int_{\Omega}(\theta^{\varepsilon_7})^{\lambda-2}|\nabla_x\theta^{\varepsilon_7}|^2\diff x\\
        &\qquad\qquad+ \varepsilon_7(1 - \lambda)\int_{\Omega}(\theta^{\varepsilon_7})^{\lambda-1}g''_{\varepsilon_1}(\theta^{\varepsilon_7})\tilde{\psi}_{\varepsilon_2}(\bB^{\varepsilon_7})|\nabla_x\theta^{\varepsilon_7}|^2\diff x\\
        &\qquad\qquad -\varepsilon_7(1 - \lambda)\int_{\Omega}(\theta^{\varepsilon_7})^{\lambda-2}(g_{\varepsilon_1}(\theta^{\varepsilon_7}) - \theta^{\varepsilon_7}g'_{\varepsilon_1}(\theta^{\varepsilon_7}))\tilde{\psi}'_{\varepsilon_2}(\bB^{\varepsilon_7})\nabla_x\bB^{\varepsilon_7}\cdot\nabla_x\theta^{\varepsilon_7}\diff x.
    \end{split}
\end{equation*}
The first term can be put on the right-hand side of \eqref{eq:step4_lambda_entr_equality} as it is negative, the same can be said for the second term as $g_{\varepsilon_1}$ is concave. For the third one we can estimate using Young's inequality
\begin{equation}\label{ineq:step4_bound_for_entr_lambda}
    \begin{split}
        &\varepsilon_7(1-\lambda)\left|\int_{\Omega}(\theta^{\varepsilon_7})^{\lambda-2}(g_{\varepsilon_1}(\theta^{\varepsilon_7}) - \theta^{\varepsilon_7}g'_{\varepsilon_1}(\theta^{\varepsilon_7}))\tilde{\psi}'_{\varepsilon_2}(\bB^{\varepsilon_7})\nabla_x\bB^{\varepsilon_7}\cdot\nabla_x\theta^{\varepsilon_7}\diff x\right|\\
        &\leq (1-\lambda)\frac{\varepsilon_7}{2}\int_{\Omega}\frac{|\nabla_x\theta^{\varepsilon_7}|^2}{(\theta^{\varepsilon_7})^{2 - \lambda}}\diff x + 2\varepsilon_7(1-\lambda)\|(\theta^{\varepsilon_7})^{\lambda/2-1}(g_{\varepsilon_1}(\theta^{\varepsilon_7}) - \theta^{\varepsilon_7}g'_{\varepsilon_1}(\theta^{\varepsilon_7}))\tilde{\psi}'_{\varepsilon_2}(\bB^{\varepsilon_7})\|^2_\infty\int_\Omega|\nabla_x\bB^{\varepsilon_7}|^2\diff x\\
        &=: I_1 + I_2.
    \end{split}
\end{equation}
The term $I_1$ can be absorbed into the right-hand side of \eqref{eq:step4_lambda_entr_equality}. For $I_2$ we observe, that from Step 3 the sequence $(\theta^{\varepsilon_7})^{\lambda/2 - 1}$ is bounded, by definition $g_{\varepsilon_1}(\theta^{\varepsilon_7})$ is bounded, and from \eqref{growthg01} $\theta^{\varepsilon_7}g'_{\varepsilon_1}(\theta^{\varepsilon_7})$ is bounded in $L^\infty(Q_T)$ as well. Moreover, by the boundedness of $\{\bF^{\varepsilon_7}\}_{\varepsilon_7 > 0}$ in $(L^\infty(Q_T))^{3\times 3}$ the same is true for $\tilde{\psi}'_{\varepsilon_2}(\bB^{\varepsilon_7})$, and from boundedness of $\{\bF^{\varepsilon_7}\}_{\varepsilon_7 > 0}$ in $(L^\infty(Q_T))^{3\times 3}\cap L^2(0, T; (W^{1,2}(\Omega))^{3\times 3})$, the integral $\int_{\Omega}|\nabla_x\bB^{\varepsilon_7}|^2\diff x$ is bounded as well. Hence, $I_2$ is uniformly bounded in $\varepsilon_7$. Now we can move to bounding the terms with $\varepsilon_4$ in front of them in \eqref{eq:step4_lambda_entr_equality}. Again, we note that by analogous arguments as above the functions $\tilde{\psi}'_{\varepsilon_2}(\bB^{\varepsilon_7})$ and $h_\lambda(\theta^{\varepsilon_7}) - g'_{\varepsilon_1}(\theta^{\varepsilon_7})(\theta^{\varepsilon_7})^{\lambda} + g_{\varepsilon_1}(\theta^{\varepsilon_7})(\theta^{\varepsilon_7})^{\lambda -1}$ are bounded in $L^\infty(Q_T)$, hence
\begin{equation}\label{ineq:step4_ineq_entr_lambda_1}
    \begin{split}
        &\varepsilon_4\left|\int_{\Omega}(h_\lambda(\theta^{\varepsilon_7}) - g'_{\varepsilon_1}(\theta^{\varepsilon_7})(\theta^{\varepsilon_7})^{\lambda} + g_{\varepsilon_1}(\theta^{\varepsilon_7})(\theta^{\varepsilon_7})^{\lambda -1})\tilde{\psi}'_{\varepsilon_2}(\bB^{\varepsilon_7}):(\nabla_x\bF^{\varepsilon_7} \cdot \nabla_x\bF^{\varepsilon_7, T})\diff x\right| \\
        &\leq \varepsilon_4\|\tilde{\psi}'_{\varepsilon_2}(\bB^{\varepsilon_7})(h_\lambda(\theta^{\varepsilon_7}) - g'_{\varepsilon_1}(\theta^{\varepsilon_7})(\theta^{\varepsilon_7})^{\lambda} + g_{\varepsilon_1}(\theta^{\varepsilon_7})(\theta^{\varepsilon_7})^{\lambda -1})\|_{\infty}\int_{\Omega}|\nabla_x\bF^{\varepsilon_7}|^2\diff x\\
        &\leq C(\varepsilon_2, \varepsilon_3)\varepsilon_4\int_{\Omega}|\nabla_x\bF^{\varepsilon_7}|^2\diff x,
    \end{split}
\end{equation}
which is bounded uniformly with respect to $\varepsilon_7$. Moving further, we bound using Young's inequality and \eqref{estkappa}
\begin{equation}\label{ineq:step4_ineq_entr_lambda_2}
    \begin{split}
        &\varepsilon_4\left|\int_{\Omega}\nabla_x(\tilde{\psi}'_{\varepsilon_2}(\bB^{\varepsilon_7})(h_\lambda(\theta^{\varepsilon_7}) - g'_{\varepsilon_1}(\theta^{\varepsilon_7})(\theta^{\varepsilon_7})^{\lambda} + g_{\varepsilon_1}(\theta^{\varepsilon_7})(\theta^{\varepsilon_7})^{\lambda -1})):(\nabla_x\bF^{\varepsilon_7}\bF^{\varepsilon_7, T} + \bF^{\varepsilon_7}\nabla_x\bF^{\varepsilon_7, T})\diff x\right|\\
        &\leq \varepsilon_4\int_{\Omega}|\tilde{\psi}''_{\varepsilon_2}(\bB^{\varepsilon_7})(h_\lambda(\theta^{\varepsilon_7}) - g'_{\varepsilon_1}(\theta^{\varepsilon_7})(\theta^{\varepsilon_7})^{\lambda} + g_{\varepsilon_1}(\theta^{\varepsilon_7})(\theta^{\varepsilon_7})^{\lambda -1}))||\nabla_x\bB^{\varepsilon_7}|^2\diff x\\
        &\qquad\qquad +2\varepsilon_4\int_{\Omega}|\tilde{\psi}'_{\varepsilon_2}(\bB^{\varepsilon_7})||\bF^{\varepsilon_7}||(1-\lambda)g'_{\varepsilon_1}(\theta^{\varepsilon_7})(\theta^{\varepsilon_7})^{\lambda-1} + (\lambda-1)g_{\varepsilon_1}(\theta^{\varepsilon_7})(\theta^{\varepsilon_7})^{\lambda-2}||\nabla_x\theta^{\varepsilon_7}||\nabla_x\bF^{\varepsilon_7}|\diff x\\
        &\leq \varepsilon_4\|\tilde{\psi}''_{\varepsilon_2}(\bB^{\varepsilon_7})(h_\lambda(\theta^{\varepsilon_7}) - g'_{\varepsilon_1}(\theta^{\varepsilon_7})(\theta^{\varepsilon_7})^{\lambda} + g_{\varepsilon_1}(\theta^{\varepsilon_7})(\theta^{\varepsilon_7})^{\lambda -1}))\|_{\infty}\int_{\Omega}|\nabla_x\bB^{\varepsilon_7}|^2\diff x\\
        &\qquad\qquad + \frac{4\varepsilon_4}{K^{-1}}(1-\lambda)\|\tilde{\psi}'_{\varepsilon_2}(\bB^{\varepsilon_7})\bF^{\varepsilon_7}(g'_{\varepsilon_1}(\theta^{\varepsilon_7})(\theta^{\varepsilon_7})^{\lambda/2} - g_{\varepsilon_1}(\theta^{\varepsilon_7})(\theta^{\varepsilon_7})^{\lambda/2-1})\|_\infty^2\int_{\Omega}|\nabla_x\bF^{\varepsilon_7}|^2\diff x\\
        &\qquad\qquad + (1-\lambda)\frac{K^{-1}}{2}\int_{\Omega}\frac{|\nabla_x\theta^{\varepsilon_7}|^2}{(\theta^{\varepsilon_7})^{2 - \lambda}}\diff x\\
        &\leq \varepsilon_4\|\tilde{\psi}''_{\varepsilon_2}(\bB^{\varepsilon_7})(h_\lambda(\theta^{\varepsilon_7}) - g'_{\varepsilon_1}(\theta^{\varepsilon_7})(\theta^{\varepsilon_7})^{\lambda} + g_{\varepsilon_1}(\theta^{\varepsilon_7})(\theta^{\varepsilon_7})^{\lambda -1}))\|_{\infty}\int_{\Omega}|\nabla_x\bB^{\varepsilon_7}|^2\diff x\\
        &\qquad\qquad + \frac{4\varepsilon_4}{K^{-1}}(1-\lambda)\|\tilde{\psi}'_{\varepsilon_2}(\bB^{\varepsilon_7})\bF^{\varepsilon_7}(g'_{\varepsilon_1}(\theta^{\varepsilon_7})(\theta^{\varepsilon_7})^{\lambda/2} - g_{\varepsilon_1}(\theta^{\varepsilon_7})(\theta^{\varepsilon_7})^{\lambda/2-1})\|_\infty^2\int_{\Omega}|\nabla_x\bF^{\varepsilon_7}|^2\diff x\\
        &\qquad\qquad + (1-\lambda)\int_{\Omega}\frac{\kappa}{2}\frac{|\nabla_x\theta^{\varepsilon_7}|^2}{(\theta^{\varepsilon_7})^{2 - \lambda}}\diff x\\
        &\leq C(\varepsilon_1,\varepsilon_2, \varepsilon_3,\varepsilon_6, K, \lambda)\varepsilon_4\int_{\Omega}|\nabla_x\bF^{\varepsilon_7}|^2\diff x + (1-\lambda)\int_{\Omega}\frac{\kappa}{2}\frac{|\nabla_x\theta^{\varepsilon_7}|^2}{(\theta^{\varepsilon_7})^{2 - \lambda}}\diff x
    \end{split}
\end{equation}
Here, we have bounded the first two terms with similar arguments as below \eqref{ineq:step4_bound_for_entr_lambda}, and the third one can be absorbed into the right-hand side of \eqref{eq:step4_lambda_entr_equality}. Combining all of our considerations from \eqref{ineq:step4_bound_for_entr_lambda} to \eqref{ineq:step4_ineq_entr_lambda_2} together with \eqref{eq:step4_lambda_entr_equality} we obtain an inequality similar to \eqref{est1nablathetalambda}, that is
\begin{align}\label{ineq:step4_nabla_theta_bound}
    \int_0^T\int_{\Omega}\frac{|\nabla_x\theta^{\varepsilon_7}|^2}{(\theta^{\varepsilon_7})^{\lambda}}\diff x\diff t \leq C, \qquad \lambda\in (1, 2),
\end{align}
and $C > 0$ is independent of $\varepsilon_7$. With this, we may finally obtain proper bounds on $\{e^{\varepsilon_7}\}_{\varepsilon_7>0}$. We test \eqref{eq:step3_internal_energy} by $e^{\varepsilon_7}$ to obtain
\begin{equation}\label{eq:step4_internal_energy_energy}
\begin{split}
    \frac{d}{dt}\int_{\Omega}|e^{\varepsilon_7}(t)|^2&\diff x + \varepsilon_7\int_{\Omega}|\nabla_x e^{\varepsilon_7}|^2\diff x\\
    &+ \int_{\Omega}\kappa\nabla_x\theta^{\varepsilon_7}\cdot \nabla_x e^{\varepsilon_7}\diff x = \int_{\Omega}(\bT_{\varepsilon_1, \varepsilon_3, \varepsilon_6} : \bD \bv^l_{\varepsilon_7})e^{\varepsilon_7}\diff x.
\end{split}
\end{equation}
The term on the right-hand side follows the same inequality as in \eqref{ineq:step1_weak_ineq_1}. Moreover, using \eqref{eq:step3_e_with_theta} we may write
\begin{equation}\label{eq:step4_energy_e_some_eq}
    \begin{split}
        &\int_{\Omega}\kappa\nabla_x\theta^{\varepsilon_7}\cdot \nabla_x e^{\varepsilon_7}\diff x = \int_{\Omega}\kappa|\nabla_x e^{\varepsilon_7}|^2\diff x +\int_{\Omega}\kappa\,\theta^{\varepsilon_7}g''_{\varepsilon_1}(\theta^{\varepsilon_7})\nabla_x\theta^{\varepsilon_7}\cdot\nabla_x e^{\varepsilon_7}\diff x\\
        &\qquad\qquad - \int_{\Omega}\kappa\,(g_{\varepsilon_1}(\theta^{\varepsilon_7}) - \theta^{\varepsilon_7} g'_{\varepsilon_1}(\theta^{\varepsilon_7}))\tilde{\psi}'_{\varepsilon_2}(\bF^{\varepsilon_7}\,\bF^{\varepsilon_7, T})\nabla_x\bF^{\varepsilon_7}\cdot\nabla_xe^{\varepsilon_7}\diff x.
    \end{split}
\end{equation}
Then, using \eqref{estkappa}, \eqref{ineq:step4_nabla_theta_bound}, the properties of the function $g_{\varepsilon_1}$, as well as the $L^2(0, T; W^{1,2}(\Omega))$ bounds on $\bF^{\varepsilon_7}$ and the Young inequality, we may deduce
\begin{equation}\label{ineq:step4_energy_e_some_bounds}
    \begin{split}
        &\left|\int_{\Omega}\kappa\,\theta^{\varepsilon_7}g''_{\varepsilon_1}(\theta^{\varepsilon_7})\nabla_x\theta^{\varepsilon_7}\cdot\nabla_x e^{\varepsilon_7}\diff x- \int_{\Omega}\kappa\,(g_{\varepsilon_1}(\theta^{\varepsilon_7}) - \theta^{\varepsilon_7} g'_{\varepsilon_1}(\theta^{\varepsilon_7}))\tilde{\psi}'_{\varepsilon_2}(\bF^{\varepsilon_7}\,\bF^{\varepsilon_7, T})\nabla_x\bF^{\varepsilon_7}\cdot\nabla_xe^{\varepsilon_7}\diff x\right|\\
        &\leq 16K\|(\theta^{\varepsilon_7})^{\lambda/2 + 1}g''_{\varepsilon_1}(\theta^{\varepsilon_7})\|^2_\infty\int_{\Omega}\frac{|\nabla_x\theta^{\varepsilon_7}|^2}{(\theta^{\varepsilon_7})^{\lambda}}\diff x +\int_{\Omega}\frac{\kappa}{4}|\nabla_xe^{\varepsilon_7}|^2\diff x \\
        &\qquad\qquad+ 16K\|(g_{\varepsilon_1}(\theta^{\varepsilon_7}) - \theta^{\varepsilon_7} g'_{\varepsilon_1}(\theta^{\varepsilon_7}))\tilde{\psi}'_{\varepsilon_2}(\bF^{\varepsilon_7})\|^2_\infty\int_{\Omega}|\nabla_x\bF^{\varepsilon_7}|^2\diff x + \int_{\Omega}\frac{\kappa}{4}|\nabla_xe^{\varepsilon_7}|^2\diff x\\
        &\leq C + \int_{\Omega}\frac{\kappa}{2}|\nabla_xe^{\varepsilon_7}|^2\diff x,
    \end{split}
\end{equation}
where $C> 0$ is independent of $\varepsilon_7$, and $\lambda$ in the considerations was taken from the interval $(1, 2)$. Hence, combining \eqref{eq:step4_energy_e_some_eq} and \eqref{ineq:step4_energy_e_some_bounds} together with \eqref{eq:step4_internal_energy_energy} we get
\begin{equation*}
    \begin{split}
        \frac{d}{dt}\int_{\Omega}|e^{\varepsilon_7}(t)|^2&\diff x + \varepsilon_7\int_{\Omega}|\nabla_x e^{\varepsilon_7}|^2\diff x + \int_{\Omega}\frac{\kappa}{2}|\nabla_x e^{\varepsilon_7}|^2\diff x \leq C + \|e^{\varepsilon_7}\|_{L^2(\Omega)}^2,
    \end{split}
\end{equation*}
which by Gr\"{o}nwall's inequality provides the bounds on $\{e^{\varepsilon_7}\}_{\varepsilon_7}$ in $L^\infty(0, T; L^2(\Omega))\cap L^2(0, T; W^{1,2}(\Omega))$. Note also that the bound on $\varepsilon_7\int_{Q_T}|\nabla_xe^{\varepsilon_7}|^2\diff x\diff t$ will allow us to converge $\varepsilon_7\Delta_x e^{\varepsilon_7}\rightarrow 0$, in the sense of distributions in \eqref{eq:step3_internal_energy}. Similarly as in Step 2, we may deduce from the equation boundedness of $\{\partial_t e^{\varepsilon_7}\}_{\varepsilon_7 > 0}$ in $L^2(0, T; (W^{1,2}(\Omega))^*)$. Having this, we can see that from the equation \eqref{eq:step3_e_with_theta}, we can deduce the bounds on $\{\theta^{\varepsilon_7}\}_{\varepsilon_7>0}$ in $L^\infty(0, T; L^2(\Omega))\cap L^2(0, T; W^{1,2}(\Omega))$ and $\{\partial_t\theta^{\varepsilon_7}\}_{\varepsilon_7>0}$ in $L^2(0, T; (W^{1,2}(\Omega))^*)$. At last, using the Banach--Alaoglu theorem and the Aubin--Lions lemma we obtain the following convergences
\begin{align*}
    \partial_t e^{\varepsilon_7} &\rightharpoonup^* \partial_t e, &&\text{ weakly* in } L^2(0, T; (W^{1,2}(\Omega))^*),\\
    e^{\varepsilon_7} &\rightharpoonup^* e, &&\text{ weakly* in }L^\infty(0, T; L^2(\Omega)),\\
    e^{\varepsilon_7} &\rightharpoonup e, &&\text{ weakly in }L^2(0, T; W^{1,2}(\Omega)),\\
    e^{\varepsilon_7} &\rightarrow e, &&\text{ strongly in }L^2(Q_T),\\
    e^{\varepsilon_7} &\rightarrow e, &&\text{ a.e. in }Q_T,
\end{align*}
and
\begin{align*}
    \partial_t \theta^{\varepsilon_7} &\rightharpoonup^* \partial_t \theta, &&\text{ weakly* in } L^2(0, T; (W^{1,2}(\Omega))^*),\\
    \theta^{\varepsilon_7} &\rightharpoonup \theta, &&\text{ weakly in }L^2(0, T; W^{1,2}(\Omega)),\\
    \theta^{\varepsilon_7} &\rightarrow \theta, &&\text{ strongly in }L^2(Q_T),\\
    \theta^{\varepsilon_7} &\rightarrow \theta, &&\text{ a.e. in }Q_T.
\end{align*}
Similarly as before, the obtained convergences allow us to converge from \eqref{eq:step3_velocity}-\eqref{eq:step3_internal_energy} to the system
\begin{align}
    &\partial_t \bv^l + \diver_x (\Lambda_{\varepsilon_3}(|\bv^l|^2)\bv^l\otimes \bv^l) - \diver_x \bT_{\varepsilon_1, \varepsilon_3, \varepsilon_6}(\theta, \bF,  \bD\bv^l) = 0, \label{eq:step4_velocity}\\
    &\partial_t \bF + \Diver_x(\bF\otimes\bv^l)- \Lambda_{\varepsilon_3}(|\bF|)\nabla_x \bv^l \bF\frac{(\theta - \varepsilon_6)_+}{\theta}- \varepsilon_4\Delta_x \bF + \frac{\tau(\theta)}{2}\frac{(\det\bF - \varepsilon_5)_+}{\det\bF}(\bF\,\bF^{T}\, \bF - \bF)= 0,\label{eq:step4_F}\\
    &\partial_te + \diver_x(e \bv^l) - \diver_x(\kappa(\theta)\nabla_x \theta)- \bT_{\varepsilon_1, \varepsilon_3,\varepsilon_6}(\theta, \bF, \bD\bv^l) : \bD\bv^l = 0,\label{eq:step4_internal_energy}
\end{align}
where \eqref{eq:step4_velocity} is to be understood as projections onto the Galerkin basis, and \eqref{eq:step4_F}, \eqref{eq:step4_internal_energy} are to be understood in the weak sense in the sense of an operator on a predual space to the time derivative. Note that from \eqref{eq:step3_e_with_theta} we keep
\begin{equation}\label{eq:step4_e_with_theta}
\begin{split}
    e(t, x) = \theta(t, x) + (g_{\varepsilon_1}(\theta(t, x)) - \theta(t, x) g'_{\varepsilon_1}(\theta(t, x)))\tilde{\psi}_{\varepsilon_2}(\bF\bF^{T}),
\end{split}
\end{equation}
for $\tilde{\psi}_{\varepsilon_2}$ defined as in \eqref{eq:modified_psi_varepsilon2}.\\

\noindent\underline{Step 5: convergence with $l\to +\infty$.} Let us start with the bounds on velocity. Since \eqref{eq:step4_velocity} is of Galerkin form, we may test it by $\bv^l$ to obtain the energy equality
\begin{equation}\label{eq:energy_before_limit_step4}
    \begin{split}
        \int_{\Omega}|\bv^l(t)|^2\diff x + \int_0^t\int_{\Omega}\bT_{\varepsilon_1, \varepsilon_3, \varepsilon_6}:\bD\bv^l\diff x\diff s = \int_{\Omega}|\bv^l(0)|^2\diff x - \int_0^t\int_{\partial\Omega}|\bv^l|^2\diff S(x)\diff s,
    \end{split}
\end{equation}
as the boundary term is negative we can drop it, and using \eqref{estkappa} together with Young's inequality we may deduce
\begin{equation*}
    \begin{split}
        \int_{\Omega}|\bv^l(t)|^2\diff x + K^{-1}\int_0^t\int_{\Omega}|\bD\bv^l|^2\diff x\diff s \leq \frac{2K}{\varepsilon^2_3}|Q_T| + \frac{K^{-1}}{2}\int_0^t\int_{\Omega}|\bD\bv^l|^2\diff x\diff s +\int_\Omega|\bv^l(0)|^2\diff x,
    \end{split}
\end{equation*}
which readily implies that $\{\bv^l\}_{l\in\N}$ is bounded in $L^\infty(0, T; L^2(\Omega))\cap L^2(0, T; W^{1,2}_{\bn,\diver}(\Omega))$. Moreover, following the reasoning in Corollary \ref{corpropest} we can see that $\{\partial_t\bv^l\}_{l\in\N}$ is bounded in $L^{2}(0, T; (W^{1,2}_{\bn,\diver}(\Omega))^*)$. Thus, by the Banach--Alaoglu theorem, the Aubin--Lions lemma, and the interpolation, we get that (up to the subsequence)
\begin{align*}
    \partial_t \bv^l &\rightharpoonup^* \partial_t \bv, &&\text{ weakly* in }L^{2}(0, T; (W^{1,2}_{\bn,\diver}(\Omega))^*),\\
    \bv^l &\rightharpoonup^* \bv, &&\text{ weakly* in }L^{\infty}(0, T; L^2(\Omega)),\\
    \bv^l &\rightharpoonup \bv, &&\text{ weakly in }L^2(0, T; W^{1,2}_{\bn,\diver}(\Omega)),\\
    \bv^l &\rightarrow \bv, &&\text{ strongly in }L^q(Q_T),\, q< 10/3,\\
    \bv^l &\rightarrow \bv, &&\text{ a.e. in }Q_T.
\end{align*}
Furthermore, by the same argument as in the previous step, we obtain the convergences for the elastic tensor
\begin{align*}
    \partial_t \bF^l &\rightharpoonup^* \partial_t \bF, &&\text{ weakly* in }L^2(0, T; ((W^{1,2}(\Omega))^{3\times 3})^*),\\
    \bF^l &\rightharpoonup \bF, &&\text{ weakly in }L^2(0, T; (W^{1,2}(\Omega))^{3\times 3}),\\
    \bF^l&\rightharpoonup^* \bF, &&\text{ weakly* in }(L^\infty(Q_T))^{3\times 3},\\
    \bF^l &\rightarrow \bF, &&\text{ strongly in }(L^q(Q_T))^{3\times 3},\ q< +\infty,\\
    \bF^l &\rightarrow \bF, &&\text{ a.e. in }Q_T.
\end{align*}
Note that we only lost here some regularity on the time derivative of $\bF$, due to the loss of the $L^\infty(Q_T)$ bounds on the gradient of the velocity. At this point we need to establish compactness of the temperature sequence $\{\theta^l\}_{l\in\N}$. Due to the loss of the regularity of the symmetric gradient of the velocity, we cannot use the simple technique of testing \eqref{eq:step4_internal_energy} by $e^l$, since the term $\bT_{\varepsilon_1, \varepsilon_3, \varepsilon_6}:\bD\bv^l$ is not bounded in $L^2(Q_T)$. Thus, we come back to the technique presented in the Subsection \ref{subsection:almost_everywhere_conv_theta} about almost everywhere convergence of the temperature. Let us begin with the observation, that one can still derive \eqref{eq:step4_lambda_entr_equality} but with $\varepsilon_7=0$, that is
\begin{equation}\label{eq:step5_lambda_entr_equality}
    \begin{split}
        &\frac{d}{dt}\int_{\Omega}\eta^{l}_\lambda\diff x \\
        &\quad- \varepsilon_4\int_{\Omega}(h_\lambda(\theta^{l}) - g'_{\varepsilon_1}(\theta^{l})(\theta^{l})^{\lambda} + g_{\varepsilon_1}(\theta^{l})(\theta^{l})^{\lambda -1})\tilde{\psi}'_{\varepsilon_2}(\bB^{l}):(\nabla_x\bF^{l} \cdot \nabla_x\bF^{l, T})\diff x\\
        &\quad- \varepsilon_4\int_{\Omega}\nabla_x(\tilde{\psi}'_{\varepsilon_2}(\bB^{l})(h_\lambda(\theta^{l}) - g'_{\varepsilon_1}(\theta^{l})(\theta^{l})^{\lambda} + g_{\varepsilon_1}(\theta^{l})(\theta^{l})^{\lambda -1})):(\nabla_x\bF^{l}\bF^{l, T} + \bF^{l}\nabla_x\bF^{l, T})\diff x\\
        &\quad-\int_{\Omega}\tau\frac{(\det\bF^{l} - \varepsilon_5)_+}{\det\bF^{l}}((\bB^{l})^2 - \bB^{l}):\tilde{\psi}'_{\varepsilon_2}(\bB^{l})(h_\lambda(\theta^{l}) - g'_{\varepsilon_1}(\theta^{l})(\theta^{l})^{\lambda})\diff x\\
        &\quad+\int_{\Omega}2(\bB^{l} - \bI):\bD\bv^l\Lambda_{\varepsilon_3}(|\bF^{l}|)\frac{(\theta^{l} - \varepsilon_6)_+}{\theta^{l}}(h_\lambda(\theta^{l}) - g'_{\varepsilon_1}(\theta^{l})(\theta^{l})^{\lambda})\diff x\\
        &=(1-\lambda)\int_{\Omega}\kappa\frac{|\nabla_x\theta^{l}|^2}{(\theta^{l})^{2 - \lambda}}\diff x + \int_{\Omega}2\nu\frac{|\bD\bv^l|^2}{(\theta^{l})^{2 - \lambda}} + \tau\frac{(\det\bF^{l} - \varepsilon_5)_+}{\det\bF^{l}}\frac{g_{\varepsilon_1}(\theta^{l})((\bB^{l})^2 - \bB^{l}):\tilde{\psi}'_{\varepsilon_2}(\bB^{l})}{(\theta^{l})^{1 - \lambda}}\diff x.
    \end{split}
\end{equation}
Furthermore, the established bounds \eqref{ineq:step4_ineq_entr_lambda_0}, \eqref{ineq:step4_ineq_entr_lambda_00}, \eqref{ineq:step4_ineq_entr_lambda_1}, \eqref{ineq:step4_ineq_entr_lambda_2} are independent of $l$, which means that \eqref{ineq:step4_nabla_theta_bound} holds as well, that is
\begin{align}\label{ineq:step5_nabla_theta_bound}
    \int_0^T\int_{\Omega}\frac{|\nabla_x\theta^{l}|^2}{(\theta^{l})^{\lambda}}\diff x\diff t \leq C(\varepsilon_1,\varepsilon_2, \varepsilon_3,\varepsilon_6, K, \lambda)\varepsilon_4\int_{Q_T}|\nabla_x\bF^{l}|^2\diff x\diff t + C(\lambda, K, \varepsilon_3), \qquad \lambda\in (1, 2),
\end{align}
where $C$ is independent of $l$, and following the argument under \eqref{est1nablathetalambda} we can deduce that $\{\theta^l\}_{l\in\N}$ is bounded in $L^q(Q_T)$ for $q < 5/3$ and $\{\nabla_x\theta^l\}_{l\in\N}$ is bounded in $L^{q}(Q_T)$ for $q < 5/4$. With this, we wish to follow the argument in Subsection \ref{subsection:almost_everywhere_conv_theta}, which means that we need the equations for the entropy $\eta^{l}$ and $\tilde{\psi}_{\varepsilon_2}$. At this point of the approximation, the functions $1/\theta$ and $\tilde{\psi}'_{\varepsilon_2}$ are still proper test functions, hence following the argument in Proposition \ref{otherform} one can verify that
\begin{equation}\label{eq:step5_entropy}
    \begin{split}
        &\partial_t\eta^{l} + \diver_x (\eta^{l}\bv^l) - \diver_x\left(\kappa\frac{\nabla_x\theta^l}{\theta^{l}}\right) + \varepsilon_4\frac{g_{\varepsilon_1}(\theta^{l})}{\theta^{l}}\tilde{\psi}'_{\varepsilon_2}(\bB^{l}) :(\bF^{l}\Delta_x\bF^{l, T} + \Delta_x\bF^{l}\bF^{l, T})\\
        &= \frac{\kappa|\nabla_x\theta^{l}|^2}{(\theta^{l})^2} + \frac{2\nu|\bD\bv^l|^2}{\theta^{l}} + \tau \frac{(\det\bF^{l} - \varepsilon_5)_+}{\det\bF^{l}}\frac{g_{\varepsilon_1}(\theta^{l})\tilde{\psi}_{\varepsilon_2}'(\bB^{l}):((\bB^{l})^2 - \bB^{l})}{\theta^{l}}, 
    \end{split}
\end{equation}
in the sense of operators on $C^1([0, T]\times\overline{\Omega})$, for $\eta^{l} = \ln\theta^{l} - g'_{\varepsilon_1}(\theta^{l})\tilde{\psi}_{\varepsilon_2}(\bB^{l})$, and following the procedure under \eqref{divetakqkcomp} one obtains
\begin{equation}\label{eq:step5_tilde_psi}
    \begin{split}
        &\partial_t\tilde{\psi}_{\varepsilon_2}(\bB^{l}) + \diver_x(\tilde{\psi}_{\varepsilon_2}(\bB^{l})\bv^l) - \varepsilon_4\tilde{\psi}'_{\varepsilon_2}(\bB^{l}) :(\bF^{l}\Delta_x\bF^{l, T} + \Delta_x\bF^{l}\bF^{l, T})\\
        &+ \tau\frac{(\det\bF^{l} - \varepsilon_5)_+}{\det\bF^{l}}\tilde{\psi}_{\varepsilon_2}'(\bB^{l}):((\bB^{l})^2 - \bB^{
        l}) = 2(\bB^{l} - \bI):\bD\bv^l\Lambda_{\varepsilon_3}(|\bF^l|)\frac{(\theta^{l} - \varepsilon_6)_+}{\theta^{l}}.
    \end{split}
\end{equation}
in the sense of operators on $C^1([0, T]\times\overline{\Omega})$, for $\tilde{\psi}_{\varepsilon_2}$ defined as in \eqref{eq:modified_psi_varepsilon2}. Now, the only obstacle to copy verbatum the argument in Subsection \ref{subsection:almost_everywhere_conv_theta} are the properties \eqref{divetakqkcomp} and \eqref{divpsiBkvkcomp}. Note that the terms without $\varepsilon_4$ can be treated the same way as in Subsection \ref{subsection:almost_everywhere_conv_theta}, that is they are $L^1$ functions, which space embeds compactly into $(W^{1,5}(Q_T))^*$. In fact, the term with $\varepsilon_4$ is also proper, as at this point we keep $\varepsilon_4$ fixed, but since in the next step we shall converge with $\varepsilon_4\to 0^+$, let us show the exact bounds at this point, so later it will be clear that those are also independent of $\varepsilon_4$. Hence, let us show that $\varepsilon_4\frac{g_{\varepsilon_1}(\theta^{l})}{\theta^{l}}\tilde{\psi}'_{\varepsilon_2}(\bB^{l}) :(\bF^{l}\Delta_x\bF^{l, T} + \Delta_x\bF^{l}\bF^{l, T})$ defines a compact sequence in $(W^{1,5}(Q_T))^*$, where we choose the non-zero trace Sobolev space to keep in mind the boundary conditions. The term in \eqref{eq:step5_tilde_psi} can be treated the same way. Fix $\varphi\in W^{1,5}(Q_T)$ with $\|\varphi\|_{W^{1,5}}\leq 1$. Then,
\begin{equation}\label{eq:step5_operator}
    \begin{split}
        &-\left\langle\varepsilon_4\frac{g_{\varepsilon_1}(\theta^{l})}{\theta^{l}}\tilde{\psi}'_{\varepsilon_2}(\bB^{l}) :(\bF^{l}\Delta_x\bF^{l, T} + \Delta_x\bF^{l}\bF^{l, T}), \varphi\right\rangle_{((W^{1,5}(Q_T))^*, W^{1,5}(Q_T))}\\
        &= 2\int_{Q_T}\varepsilon_4\varphi\frac{g_{\varepsilon_1}(\theta^{l})}{\theta^{l}}\tilde{\psi}'_{\varepsilon_2}(\bB^{l}):(\nabla_x\bF^{l}\cdot\nabla_x\bF^{l, T})\diff x\diff t\\
        &\qquad + \int_{Q_T}\varepsilon_4\varphi\frac{g_{\varepsilon_1}(\theta^{l})}{\theta^{l}}\mathbf{1}_{\{\det\bF^{l} > \varepsilon_2\}}|(\bB^{l})^{-1}\nabla_x\bB^{l}|^2\diff x\diff t\\
        &\qquad +\int_{Q_T}\varepsilon_4\varphi\frac{g_{\varepsilon_1}(\theta^{l}) - \theta^{l}g'_{\varepsilon_1}(\theta^{l})}{(\theta^{l})^2}\tilde{\psi}'_{\varepsilon_2}(\bB^{l}):(\bF^{l}(\nabla_x\theta^{l}\cdot\nabla_x\bF^{l, T}) + (\nabla_x\theta^{l}\cdot\nabla_x\bF^{l})\bF^{l, T})\diff x\diff t\\
        &\qquad +\int_{Q_T}\varepsilon_4\frac{g_{\varepsilon_1}(\theta^{l})}{\theta^{l}}\tilde{\psi}'_{\varepsilon_2}(\bB^{l}):(\bF^{l}(\nabla_x\varphi\cdot\nabla_x\bF^{l, T}) + (\nabla_x\varphi\cdot\nabla_x\bF^{l})\bF^{l, T})\diff x\diff t\\
        & =: I^{l}_1 + I^{l}_2 + I^{l}_3 + I^{l}_4.
    \end{split}
\end{equation}
Here, we bound
\begin{equation}\label{ineq:step5_bound_operator_1}
    \begin{split}
        |I^{l}_1| \leq 2\left\|\frac{g_{\varepsilon_1}(\theta^{l})}{\theta^{l}}\tilde{\psi}'_{\varepsilon_2}(\bB^{l})\right\|_\infty\varepsilon_4\int_{Q_T}|\varphi||\nabla_x\bF^{l}|^2\diff x\diff t \leq C(\varepsilon_2, \varepsilon_3)\varepsilon_4\int_{Q_T}|\nabla_x\bF^l|^2\diff x\diff t,
    \end{split}
\end{equation}
and
\begin{equation}\label{ineq:step5_bound_operator_2}
    \begin{split}
        |I^{l}_2|\leq \left\|\frac{g_{\varepsilon_1}(\theta^{l})}{\theta^{l}}(\bB^{l})^{-2}\right\|_\infty\|(\bF^{l})^2\|_\infty\,\varepsilon_4\int_{Q_T}|\varphi||\nabla_x\bF^{l}|^2\diff x\diff t\leq C(\varepsilon_2, \varepsilon_3)\varepsilon_4\int_{Q_T}|\nabla_x\bF^l|^2\diff x\diff t,
    \end{split}
\end{equation}
as well as using H\"{o}lder's inequality
\begin{equation}\label{ineq:step5_bound_operator_3}
    \begin{split}
        |I^{l}_3|&\leq \sqrt{\varepsilon_4}\|(g_{\varepsilon_1}(\theta^{l}) - \theta^{l}g'_{\varepsilon_1}(\theta^{l}))\tilde{\psi}'_{\varepsilon_2}(\bB^{l})\|_\infty\|\varphi\|_\infty\left\|\frac{\nabla_x\theta^{l}}{(\theta^{l})^2}\right\|_2 \|\sqrt{\varepsilon_4}\nabla_x\bF^{l}\|_2\\
        &\leq \sqrt{\varepsilon_4}C(\varepsilon_1,\varepsilon_2, \varepsilon_3,\varepsilon_6, K)\left(\varepsilon_4\int_{Q_T}|\nabla_x\bF^{l}|^2\diff x\diff t + 1\right)\|\sqrt{\varepsilon_4}\nabla_x\bF^{l}\|_2,
    \end{split}
\end{equation}
and
\begin{equation}\label{ineq:step5_bound_operator_4}
    \begin{split}
        |I^{l}_4|\leq 2\sqrt{\varepsilon_4}\left\|\frac{g_{\varepsilon_1}(\theta^{l})}{\theta^{l}}\tilde{\psi}'_{\varepsilon_2}(\bB^{l})\right\|_\infty\|\bF^{l}\|_\infty\|\nabla_x\varphi\|_2\|\sqrt{\varepsilon_4}\nabla_x\bF^{l}\|_2 \leq \sqrt{\varepsilon_4}C(\varepsilon_2, \varepsilon_3)\|\sqrt{\varepsilon_4}\nabla_x\bF^{l}\|_2
    \end{split}
\end{equation}
which means that $I^{l}_1$, $I^{l}_2$, $I^{l}_3$, and $I^{l}_4$ are operators defined by $L^1$ functions, which are uniformly bounded in $L^1$ with respect to the $l$. Hence, by Sobolev embeddings, we can deduce that there exists $T\in (W^{1,5}(Q_T))^*$ such that (up to the subsequence)
$$
I^{l}_1 + I^{l}_2 + I^{l}_3 + I^{l}_4 \rightarrow (T, \varphi)_{((W^{1,5}(Q_T))^*, W^{1,5}(Q_T))}\text{, as }l\to +\infty.
$$
Thus, indeed $\varepsilon_4\frac{g_{\varepsilon_1}(\theta^{l})}{\theta^{l}}\tilde{\psi}'_{\varepsilon_2}(\bB^{l}) :(\bF^{l}\Delta_x\bF^{l, T} + \Delta_x\bF^{l}\bF^{l, T})$ defines a compact sequence in $(W^{1,5}(Q_T))^*$, which is what we wanted to show. Here, we may finally follow Subsection \ref{subsection:almost_everywhere_conv_theta} to deduce
\begin{align*}
    \theta^{l}\rightarrow\theta\quad\text{ a.e. in }Q_T,
\end{align*}
which combined with \eqref{ineq:step5_nabla_theta_bound}, and our comments below \eqref{ineq:step5_nabla_theta_bound} gives us
\begin{align*}
    \partial_t \theta^l &\rightharpoonup^* \partial_t \theta, &&\text{ weakly* in }L^q(0, T; (W^{1,q}(\Omega))^*), q > 5\\
    \theta^l &\rightarrow \theta, &&\text{ strongly in }L^q(Q_T),\, q< 5/3,\\
    \nabla_x\theta^l &\rightharpoonup \nabla_x\theta, &&\text{ weakly in }L^q(Q_T),\, q< 5/4,\\
    \theta^l &\rightarrow \theta, &&\text{ a.e. in }Q_T,\\
    (\theta^l)^{\frac{2-\lambda}{2}} &\rightharpoonup\theta^{\frac{2-\lambda}{2}}, &&\text{ weakly in }L^2(0, T; W^{1,2}(\Omega)),\, \lambda\in (0, 1).
\end{align*}
Moreover, analysing \eqref{eq:step5_entropy}, one can obtain similarly to \eqref{estlnt} and \eqref{estgradln} we can deduce
\begin{align*}
    \ln\theta^{l}\rightharpoonup\ln\theta,\text{ weakly in }L^2(0, T; W^{1,2}(\Omega)).
\end{align*}
Now, from \eqref{eq:step4_e_with_theta} we can see that $e^l$ inherits its regularity from $\theta^l$ and $\bF^l$. One obtains
\begin{align*}
    \partial_t e^l &\rightharpoonup^* \partial_t e, &&\text{ weakly* in }L^q(0, T; (W^{1,q}(\Omega))^*), q > 5\\
    e^l &\rightarrow e, &&\text{ strongly in }L^q(Q_T),\, q< 5/3,\\
    e^l &\rightarrow e, &&\text{ a.e. in }Q_T.
\end{align*}
At last, to converge in equations \eqref{eq:step4_velocity}--\eqref{eq:step4_internal_energy} we need to show that
$$
\int_{0}^t\int_{\Omega}\nu(\theta^l)|\bD\bv^l|^2\diff x\diff s \rightarrow \int_{0}^t\int_{\Omega}\nu(\theta)|\bD\bv|^2\diff x\diff s, \text{ as }l\to +\infty.
$$
We shall employ the classical monotonicity method. Using the proven convergences for variables, we can converge with \eqref{eq:step4_velocity} to obtain
$$
\partial_t \bv + \diver_x (\Lambda_{\varepsilon_3}(|\bv|^2)\bv\otimes \bv) - \diver_x \bT_{\varepsilon_1, \varepsilon_3, \varepsilon_6}(\theta, \bF,  \bD\bv) = 0,
$$
in the sense of distributions. Testing the equation above by $\bv$ we get the energy equality in the limiting equation
\begin{equation}\label{eq:energy_limit_step4}
    \begin{split}
        \int_{\Omega}|\bv(t)|^2\diff x + \int_0^t\int_{\Omega}\bT_{\varepsilon_1, \varepsilon_3, \varepsilon_6}:\bD\bv\diff x\diff s = \int_{\Omega}|\bv(0)|^2\diff x - \int_0^t\int_{\partial\Omega}|\bv|^2\diff S(x)\diff s
    \end{split}
\end{equation}
Now, we compare it to the limit of \eqref{eq:energy_before_limit_step4}. We have
\begin{equation}\label{eq:step4_for_mono_trick_gradient_v}
    \begin{split}
        &\limsup_{l\to +\infty}\int_0^t\int_{\Omega}\nu(\theta^l)|\bD\bv^l|^2\diff x\diff s = \limsup_{l\to +\infty}\left(\int_\Omega|\bv^l(0)|^2\diff x - \int_{\Omega}|\bv^l(t)|^2\diff x\right)\\
        &\qquad- \liminf_{l\to +\infty}\left(\int_0^t\int_{\partial\Omega}|\bv^l|^2\diff S(x)\diff s +\int_0^t\int_{\Omega}2\Lambda_{\varepsilon_3}(|\bF^l|)g_{\varepsilon_1}(\theta^l)\bF^l\,\bF^{l,T}\frac{(\theta^l - \varepsilon_6)_+}{\theta^l}: \bD\bv^l\diff x\diff s\right)
    \end{split}
\end{equation}
By weak lower-semi continuity of norms, and convergence of initial datum we deduce
\begin{equation*}
    \begin{split}
        \limsup_{l\to +\infty}\left(\int_\Omega|\bv^l(0)|^2\diff x - \int_{\Omega}|\bv^l(t)|^2\diff x\right) \leq \int_\Omega|\bv(0)|^2\diff x - \int_{\Omega}|\bv(t)|^2\diff x.
    \end{split}
\end{equation*}
Using the trace theorem, and again the lower semi-continuity of norms
\begin{align*}
    - \liminf_{l\to +\infty}\int_0^t\int_{\partial\Omega}|\bv^l|^2\diff S(x)\diff s\leq -\int_0^t\int_{\partial\Omega}|\bv|^2\diff S(x)\diff s.
\end{align*}
At last, by the strong convergence of $\bF^l$, the almost everywhere convergence of $\theta^l$, the Vitali's convergence theorem, and the weak convergence of $\bD\bv^l$ we imply
\begin{equation*}
    \begin{split}
        \lim_{l\to +\infty}\int_0^t\int_{\Omega}2\Lambda_{\varepsilon_3}(|\bF^l|)g_{\varepsilon_1}(\theta^l)\bF^l\,\bF^{l,T}\frac{(\theta^l - \varepsilon_6)_+}{\theta^l}: \bD\bv^l\diff x\diff s = \int_0^t\int_{\Omega}2\Lambda_{\varepsilon_3}(|\bF|)g_{\varepsilon_1}(\theta)\bF\,\bF^{T}\frac{(\theta - \varepsilon_6)_+}{\theta}: \bD\bv\diff x\diff s
    \end{split}
\end{equation*}
Combining the above with \eqref{eq:step4_for_mono_trick_gradient_v} we get
\begin{equation*}
    \begin{split}
        &\limsup_{l\to +\infty}\int_0^t\int_{\Omega}\nu(\theta^l)|\bD\bv^l|^2\diff x\diff s \leq \int_\Omega|\bv(0)|^2\diff x - \int_{\Omega}|\bv(t)|^2\diff x - \int_0^t\int_{\partial\Omega}|\bv|^2\diff S(x)\diff s\\
        &\qquad-\int_0^t\int_{\Omega}2\Lambda_{\varepsilon_3}(|\bF|)g_{\varepsilon_1}(\theta)\bF\,\bF^{T}\frac{(\theta - \varepsilon_6)_+}{\theta}: \bD\bv\diff x\diff s\\
        &=\int_0^t\int_{\Omega}\nu(\theta)|\bD\bv|^2\diff x\diff s.
    \end{split}
\end{equation*}
Here, we note that due to the boundedness of $\nu$ from \eqref{estkappa}, almost everywhere convergence of $\theta^l$ and the weak convergence of $\bD\bv^l$ we can obtain
$$
\sqrt{\nu(\theta^l)}\bD\bv^l\rightharpoonup\sqrt{\nu(\theta)}\bD\bv\text{, weakly in }L^2(Q_T),
$$
thus by weak lower-semicontinuity of norms
$$
\int_0^t\int_{\Omega}\nu(\theta)|\bD\bv|^2\diff x\diff s \leq \liminf_{l\to +\infty}\int_0^t\int_{\Omega}\nu(\theta^l)|\bD\bv^l|^2\diff x\diff s.
$$
Therefore, we have shown what we needed, i.e.
$$
\int_0^t\int_{\Omega}\nu(\theta)|\bD\bv|^2\diff x\diff s = \lim_{l\to +\infty}\int_0^t\int_{\Omega}\nu(\theta^l)|\bD\bv^l|^2\diff x\diff s.
$$
Using all our proven convergences, we may converge from \eqref{eq:step4_velocity}-\eqref{eq:step4_internal_energy} to the system
\begin{align}
    &\partial_t \bv + \diver_x (\Lambda_{\varepsilon_3}(|\bv|^2)\bv\otimes \bv) - \diver_x \bT_{\varepsilon_1, \varepsilon_3, \varepsilon_6}(\theta, \bF,  \bD\bv) = 0, \label{eq:step5_velocity}\\
    &\partial_t \bF + \Diver_x(\bF\otimes\bv)- \Lambda_{\varepsilon_3}(|\bF|)\nabla_x \bv \bF\frac{(\theta - \varepsilon_6)_+}{\theta}- \varepsilon_4\Delta_x \bF + \frac{\tau(\theta)}{2}\frac{(\det\bF - \varepsilon_5)_+}{\det\bF}(\bF\,\bF^{T}\, \bF - \bF)= 0,\label{eq:step5_F}\\
    &\partial_te + \diver_x(e \bv) - \diver_x(\kappa(\theta)\nabla_x \theta)- \bT_{\varepsilon_1, \varepsilon_3,\varepsilon_6}(\theta, \bF, \bD\bv) : \bD\bv^l = 0,\label{eq:step5_internal_energy}
\end{align}
where  \eqref{eq:step5_velocity}--\eqref{eq:step5_internal_energy} are to be understood in the weak sense in the sense of an operator on a predual space to the time derivative. Note that from \eqref{eq:step4_e_with_theta} we keep
\begin{equation}\label{eq:step5_e_with_theta}
\begin{split}
    e(t, x) = \theta(t, x) + (g_{\varepsilon_1}(\theta(t, x)) - \theta(t, x) g'_{\varepsilon_1}(\theta(t, x)))\tilde{\psi}_{\varepsilon_2}(\bF\bF^{T}),
\end{split}
\end{equation}
for $\tilde{\psi}_{\varepsilon_2}$ defined as in \eqref{eq:modified_psi_varepsilon2}.

Beforem moving onto the next step, let us show one more convergence which will help us identify the operator $T$. Consider the equation \eqref{eq:step4_F} tested by $\bF^l$, which gives as energy equality
\begin{equation}\label{eq:step5_energy_F_l}
\begin{split}
    &\varepsilon_4\int_0^t\int_{\Omega}|\nabla_x\bF^l|^2\diff x\diff\tau = -\int_{\Omega}|\bF^l(t, x)|^2\diff x + \int_{\Omega}|\bF^l_0(x)|^2\diff x\\
    &\qquad\qquad+ \int_0^t\int_{\Omega}\Lambda_{\varepsilon_3}(|\bF^l|)\nabla_x \bv^l :\bF^l\bF^{l,T}\frac{(\theta^l - \varepsilon_6)_+}{\theta^l} - \frac{\tau(\theta^l)}{2}\frac{(\det\bF^l - \varepsilon_5)_+}{\det\bF^l}(|\bF^l\bF^{l, T}|^2 - |\bF^l|^2)\diff x\diff \tau,
\end{split}
\end{equation}
as well as the equation \eqref{eq:step5_F} tested by $\bF$ which provides
\begin{equation}\label{eq:step5_energy_F}
\begin{split}
    &\varepsilon_4\int_0^t\int_{\Omega}|\nabla_x\bF|^2\diff x\diff\tau = -\int_{\Omega}|\bF(t, x)|^2\diff x + \int_{\Omega}|\bF_0(x)|^2\diff x\\
    &\qquad\qquad+ \int_0^t\int_{\Omega}\Lambda_{\varepsilon_3}(|\bF|)\nabla_x \bv^l :\bF\bF^{T}\frac{(\theta - \varepsilon_6)_+}{\theta} - \frac{\tau(\theta)}{2}\frac{(\det\bF - \varepsilon_5)_+}{\det\bF}(|\bF\bF^{ T}|^2 - |\bF|^2)\diff x\diff \tau.
\end{split}
\end{equation}
Now using the already established convergences for $\bF^l$, $\nabla_x\bv^l$ and $\theta^l$ we can deduce
\begin{equation*}
    \begin{split}
        \int_0^t\int_{\Omega}\Lambda_{\varepsilon_3}(|\bF^l|)\nabla_x \bv^l &:\bF^l\bF^{l,T}\frac{(\theta^l - \varepsilon_6)_+}{\theta^l} - \frac{\tau(\theta^l)}{2}\frac{(\det\bF^l - \varepsilon_5)_+}{\det\bF^l}(|\bF^l\bF^{l, T}|^2 - |\bF^l|^2)\diff x\diff \tau\\
        &\rightarrow \int_0^t\int_{\Omega}\Lambda_{\varepsilon_3}(|\bF|)\nabla_x \bv :\bF\bF^{T}\frac{(\theta - \varepsilon_6)_+}{\theta} - \frac{\tau(\theta)}{2}\frac{(\det\bF - \varepsilon_5)_+}{\det\bF}(|\bF\bF^{ T}|^2 - |\bF|^2)\diff x\diff \tau.
    \end{split}
\end{equation*}
Thus, applying as well weak lower semicontinuity of norms we obtain
\begin{equation*}
\begin{split}
    &\limsup_{l\to +\infty}\left(\varepsilon_4\int_0^t\int_{\Omega}|\nabla_x\bF^l|^2\diff x\diff\tau \right) \leq  -\int_{\Omega}|\bF(t, x)|^2\diff x + \int_{\Omega}|\bF_0(x)|^2\diff x\\
    &\qquad\qquad+ \int_0^t\int_{\Omega}\Lambda_{\varepsilon_3}(|\bF|)\nabla_x \bv :\bF\bF^{T}\frac{(\theta - \varepsilon_6)_+}{\theta} - \frac{\tau(\theta)}{2}\frac{(\det\bF - \varepsilon_5)_+}{\det\bF}(|\bF\bF^{ T}|^2 - |\bF|^2)\diff x\diff \tau\\
    & = \varepsilon_4\int_0^t\int_{\Omega}|\nabla_x\bF|^2\diff x\diff\tau.
\end{split}
\end{equation*}
Since by weak lower semicontinuity
$$
\liminf_{l\to +\infty}\left(\varepsilon_4\int_0^t\int_{\Omega}|\nabla_x\bF^l|^2\diff x\diff\tau\right) \geq \varepsilon_4\int_0^t\int_{\Omega}|\nabla_x\bF|^2\diff x\diff\tau,
$$
this implies
\begin{align}\label{conv:step5_nabla_F}
\lim_{l\to +\infty}\left(\varepsilon_4\int_0^t\int_{\Omega}|\nabla_x\bF^l|^2\diff x\diff\tau\right) = \varepsilon_4\int_0^t\int_{\Omega}|\nabla_x\bF|^2\diff x\diff\tau,
\end{align}
and in consequence
$$
\bF^{l} \rightarrow \bF,\text{ strongly in }L^2(0, T; (W^{1,2}(\Omega))^{3\times 3}).
$$
Of a particular attention is a fact, that we can now identify the operator $T$ defined under \eqref{ineq:step5_bound_operator_4}. Looking at the definitions of $I^l_i$ in \eqref{eq:step5_operator} the now established strong convergence of the gradient of $\bF^l$ allows us to see that
\begin{align*}
    &I_1^l \rightarrow 2\int_{Q_T}\varepsilon_4\varphi\frac{g_{\varepsilon_1}(\theta^{})}{\theta^{}}\tilde{\psi}'_{\varepsilon_2}(\bB^{}):(\nabla_x\bF^{}\cdot\nabla_x\bF^{ T})\diff x\diff t,\\
    &I_2^l \rightarrow \int_{Q_T}\varepsilon_4\varphi\frac{g_{\varepsilon_1}(\theta^{})}{\theta^{}}\mathbf{1}_{\{\det\bF^{} > \varepsilon_2\}}|(\bB^{})^{-1}\nabla_x\bB^{}|^2\diff x\diff t,\\
    &I_3^l \rightarrow \int_{Q_T}\varepsilon_4\varphi\frac{g_{\varepsilon_1}(\theta^{}) - \theta^{l}g'_{\varepsilon_1}(\theta^{})}{(\theta^{})^2}\tilde{\psi}'_{\varepsilon_2}(\bB^{}):(\bF^{}(\nabla_x\theta^{}\cdot\nabla_x\bF^{ T}) + (\nabla_x\theta^{}\cdot\nabla_x\bF^{})\bF^{ T})\diff x\diff t,\\
    &I_4^l \rightarrow\int_{Q_T}\varepsilon_4\frac{g_{\varepsilon_1}(\theta^{})}{\theta^{}}\tilde{\psi}'_{\varepsilon_2}(\bB^{}):(\bF^{}(\nabla_x\varphi\cdot\nabla_x\bF^{ T}) + (\nabla_x\varphi\cdot\nabla_x\bF^{})\bF^{T})\diff x\diff t,
\end{align*}
hence
\begin{equation}\label{eq:step5_operator_identified}
\begin{split}
    T(\varphi) &=  2\int_{Q_T}\varepsilon_4\varphi\frac{g_{\varepsilon_1}(\theta^{})}{\theta^{}}\tilde{\psi}'_{\varepsilon_2}(\bB^{}):(\nabla_x\bF^{}\cdot\nabla_x\bF^{ T})\diff x\diff t\\
    &\quad +\int_{Q_T}\varepsilon_4\varphi\frac{g_{\varepsilon_1}(\theta^{})}{\theta^{}}\mathbf{1}_{\{\det\bF^{} > \varepsilon_2\}}|(\bB^{})^{-1}\nabla_x\bB^{}|^2\diff x\diff t\\
    &\quad + \int_{Q_T}\varepsilon_4\varphi\frac{g_{\varepsilon_1}(\theta^{}) - \theta^{l}g'_{\varepsilon_1}(\theta^{})}{(\theta^{})^2}\tilde{\psi}'_{\varepsilon_2}(\bB^{}):(\bF^{}(\nabla_x\theta^{}\cdot\nabla_x\bF^{ T}) + (\nabla_x\theta^{}\cdot\nabla_x\bF^{})\bF^{ T})\diff x\diff t\\
    &\quad + \int_{Q_T}\varepsilon_4\frac{g_{\varepsilon_1}(\theta^{})}{\theta^{}}\tilde{\psi}'_{\varepsilon_2}(\bB^{}):(\bF^{}(\nabla_x\varphi\cdot\nabla_x\bF^{ T}) + (\nabla_x\varphi\cdot\nabla_x\bF^{})\bF^{T})\diff x\diff t.
\end{split}
\end{equation}

\noindent\underline{Step 6: convergence with $\varepsilon_4\to 0^+$.} Similarly as in the previous step
\begin{align*}
    \partial_t \bv^{\varepsilon_4} &\rightharpoonup^* \partial_t \bv, &&\text{ weakly* in }L^{2}(0, T; (W^{1,2}_{\bn,\diver}(\Omega))^*),\\
    \bv^{\varepsilon_4} &\rightharpoonup^* \bv, &&\text{ weakly* in }L^{\infty}(0, T; L^2(\Omega)),\\
    \bv^{\varepsilon_4} &\rightharpoonup \bv, &&\text{ weakly in }L^2(0, T; W^{1,2}_{\bn,\diver}(\Omega)),\\
    \bv^{\varepsilon_4} &\rightarrow \bv, &&\text{ strongly in }L^q(Q_T),\, q< 10/3,\\
    \bv^{\varepsilon_4} &\rightarrow \bv, &&\text{ a.e. in }Q_T.
\end{align*}
Again, one can verify that \eqref{eq:step5_F} can be tested by $\bF^{\varepsilon_4}$, which means that we again obtain the bounds on $\{\bF^{\varepsilon_4}\}_{\varepsilon_4 > 0}$ in $L^\infty(0, T; (L^2(\Omega))^{3\times 3})$, as well as the bound on $\{\sqrt{\varepsilon_4}\nabla_x\bF^{\varepsilon_4}\}_{\varepsilon_4 > 0}$ in $(L^2(Q_T))^{3\times 3\times 3}$. Repeating the argument at the beginning of Step 4, we imply that $\{\bF^{\varepsilon_4}\}_{\varepsilon_4 > 0}$ is bounded in $(L^\infty(Q_T))^{3\times 3}$, see also comments under \eqref{ineq:step2_linfty_bound}. A priori this is not enough to converge strongly with the sequence $\{\bF^{\varepsilon_4}\}_{\varepsilon_4>0}$, and we shall utilize the technique presented in Subsection \ref{subsection:strong_convergence_of_F}. But before that, we need to establish compactness of the temperature sequence $\{\theta^{\varepsilon_4}\}_{\varepsilon_4 > 0}$. To do so, note that by weak lower semicontinuity of an $L^2$ norm of $\nabla_x(\theta^{1-\lambda/2})$ and the established convergence \eqref{conv:step5_nabla_F} we can deduce from \eqref{ineq:step5_nabla_theta_bound}
\begin{align}\label{ineq:step6_nabla_theta_bound}
    \int_0^T\int_{\Omega}\frac{|\nabla_x\theta^{\varepsilon_4}|^2}{(\theta^{\varepsilon_4})^{\lambda}}\diff x\diff t \leq C(\varepsilon_1,\varepsilon_2, \varepsilon_3,\varepsilon_6, K, \lambda)\varepsilon_4\int_{Q_T}|\nabla_x\bF^{\varepsilon_4}|^2\diff x\diff t + C(\lambda,K, \varepsilon_3), \qquad \lambda\in (1, 2),
\end{align}
 and following the argument under \eqref{est1nablathetalambda} we can deduce that $\{\theta^{\varepsilon_4}\}_{\varepsilon_4>0}$ is bounded in $L^q(Q_T)$ for $q < 5/3$ and $\{\nabla_x\theta^{\varepsilon_4}\}_{\varepsilon_4>0}$ is bounded in $L^{q}(Q_T)$ for $q < 5/4$. With this, we wish to follow the argument in Subsection \ref{subsection:almost_everywhere_conv_theta}, which means that we need the equations for the entropy $\eta^{\varepsilon_4}$ and $\tilde{\psi}_{\varepsilon_2}$. At this point of the approximation, the functions $1/\theta$ is no longer a proper test function as the duality between $\partial_te^{\varepsilon_4}$ and $1/\theta^{\varepsilon_4}$ is "not correct". Thus, we will converge in some sense in the equality \eqref{eq:step5_entropy}. Using our established convergences it is easy to see that for the first three terms of the left-hand side of \eqref{eq:step5_entropy}
 $$
 \partial_t\eta^{l, \varepsilon_4} + \diver_x (\eta^{l, \varepsilon_4}\bv^{l, \varepsilon_4}) - \diver_x\left(\kappa\frac{\nabla_x\theta^{l, \varepsilon_4}}{\theta^{l, \varepsilon_4}}\right) \wstar \partial_t\eta^{\varepsilon_4} + \diver_x (\eta^{\varepsilon_4}\bv^{\varepsilon_4}) - \diver_x\left(\kappa\frac{\nabla_x\theta^{\varepsilon_4}}{\theta^{\varepsilon_4}}\right),
 $$
 as $l\to +\infty$ in $(W^{1,5}(Q_T))^*$. For the fourth term we already know that it converges in the sense of operators to some $T^{\varepsilon_4}$ defined below \eqref{ineq:step5_bound_operator_4}, and identified in \eqref{eq:step5_operator_identified}. Moreover, using the convergence \eqref{conv:step5_nabla_F}, this operator satisfies a bound
 \begin{multline}\label{ineq:step6_operator_norm_bound}
     \|T^{\varepsilon_4}\|_{(W^{1,5}(Q_T))^*}\\
     \leq C(\varepsilon_2, \varepsilon_3)\varepsilon_4\int_{Q_T}|\nabla_x\bF^{\varepsilon_4}|^2\diff x\diff t + \sqrt{\varepsilon_4}C(\varepsilon_1,\varepsilon_2, \varepsilon_3,\varepsilon_6, K)\left(\varepsilon_4\int_{Q_T}|\nabla_x\bF^{\varepsilon_4}|^2\diff x\diff t + 1\right)\|\sqrt{\varepsilon_4}\nabla_x\bF^{\varepsilon_4}\|_2,
 \end{multline}
 which we will show soon converges to $0$ as $\varepsilon_4\to 0^+$, making sure we get rid of this term after this step. For the right-hand side of \eqref{eq:step5_entropy} we use the fact that all three terms are bounded in $L^1$, hence we can define a measure $\mu\in\mathcal{M}^+(Q_T)$ by
 \begin{align*}
     \mu &:= \text{weak}^*\lim\left(\frac{\kappa|\nabla_x\theta^{l, \varepsilon_4}|^2}{(\theta^{l, \varepsilon_4})^2}- \frac{\kappa|\nabla_x\theta^{\varepsilon_4}|^2}{(\theta^{\varepsilon_4})^2}\right)\\
     &\qquad+\text{weak}^*\lim\left(\frac{2\nu|\bD\bv^l_{\varepsilon_4}|^2}{\theta^{l, \varepsilon_4}} - \frac{2\nu|\bD\bv^{\varepsilon_4}|^2}{\theta^{\varepsilon_4}}\right)\\
     &\qquad+\text{weak}^*\lim\Bigg(\tau \frac{(\det\bF^{l,\varepsilon_4} - \varepsilon_5)_+}{\det\bF^{l, \varepsilon_4}}\frac{g_{\varepsilon_1}(\theta^{l, \varepsilon_4})\tilde{\psi}_{\varepsilon_2}'(\bB^{l, \varepsilon_4}):((\bB^{l, \varepsilon_4})^2 - \bB^{l, \varepsilon_4})}{\theta^{l, \varepsilon_4}}\\
     &\qquad\qquad\qquad\qquad\qquad\qquad\qquad\qquad\qquad- \tau \frac{(\det\bF^{\varepsilon_4} - \varepsilon_5)_+}{\det\bF^{\varepsilon_4}}\frac{g_{\varepsilon_1}(\theta^{\varepsilon_4})\tilde{\psi}_{\varepsilon_2}'(\bB^{\varepsilon_4}):((\bB^{\varepsilon_4})^2 - \bB^{\varepsilon_4})}{\theta^{\varepsilon_4}}\Bigg),
 \end{align*}
 where this measure is positive due to the weak lower semicontinuity of norms. Thus, after the convergence with $l\to +\infty$ in \eqref{eq:step5_entropy} we deduce
\begin{equation}\label{eq:step6_entropy}
    \begin{split}
        &\partial_t\eta^{\varepsilon_4} + \diver_x (\eta^{\varepsilon_4}\bv^{\varepsilon_4}) - \diver_x\left(\kappa\frac{\nabla_x\theta^{\varepsilon_4}}{\theta^{\varepsilon_4}}\right) - T^{\varepsilon_4}\\
        &= \frac{\kappa|\nabla_x\theta^{\varepsilon_4}|^2}{(\theta^{\varepsilon_4})^2} + \frac{2\nu|\bD\bv^{\varepsilon_4}|^2}{\theta^{\varepsilon_4}} + \tau \frac{(\det\bF^{\varepsilon_4} - \varepsilon_5)_+}{\det\bF^{\varepsilon_4}}\frac{g_{\varepsilon_1}(\theta^{\varepsilon_4})\tilde{\psi}_{\varepsilon_2}'(\bB^{\varepsilon_4}):((\bB^{\varepsilon_4})^2 - \bB^{\varepsilon_4})}{\theta^{\varepsilon_4}} + \mu^{\varepsilon_4}, 
    \end{split}
\end{equation}
in the sense of operators in $(W^{1,5}(Q_T))^*$, for $\eta^{\varepsilon_4} = \ln\theta^{\varepsilon_4} - g'_{\varepsilon_1}(\theta^{\varepsilon_4})\tilde{\psi}_{\varepsilon_2}(\bB^{\varepsilon_4})$. The equation for $\tilde{\psi}_{\varepsilon_2}$ still has proper duality, therefore following the procedure under \eqref{divetakqkcomp} one obtains
\begin{equation}\label{eq:step6_tilde_psi}
    \begin{split}
        &\partial_t\tilde{\psi}_{\varepsilon_2}(\bB^{\varepsilon_4}) + \diver_x(\tilde{\psi}_{\varepsilon_2}(\bB^{\varepsilon_4})\bv_{\varepsilon_4}) - \varepsilon_4\tilde{\psi}'_{\varepsilon_2}(\bB^{\varepsilon_4}) :(\bF^{\varepsilon_4}\Delta_x\bF^{\varepsilon_4, T} + \Delta_x\bF^{\varepsilon_4}\bF^{\varepsilon_4, T})\\
        &+ \tau\frac{(\det\bF^{\varepsilon_4} - \varepsilon_5)_+}{\det\bF^{\varepsilon_4}}\tilde{\psi}_{\varepsilon_2}'(\bB^{\varepsilon_4}):((\bB^{\varepsilon_4})^2 - \bB^{\varepsilon_4}) = 2(\bB^{\varepsilon_4} - \bI):\bD\bv^l_{\varepsilon_4}\frac{(\theta^{\varepsilon_4} - \varepsilon_6)_+}{\theta^{\varepsilon_4}}.
    \end{split}
\end{equation}
in the sense of operators on $C^1([0, T]\times \overline{\Omega})$, for $\tilde{\psi}_{\varepsilon_2}$ defined as in \eqref{eq:modified_psi_varepsilon2}. Now, we can copy verbatum the argument in Subsection \ref{subsection:almost_everywhere_conv_theta} as in the previous step. Since \eqref{eq:step6_entropy} can be tested by $1$, we can deduce that the right-hand side is bounded in total variation norm, hence, we may embed it compactly into $(W^{1,5}(Q_T))^*$. Since $T^{\varepsilon_4}$ is defined by \eqref{eq:step5_operator} it is also defined via uniformly bounded $L^1$ functions, and defines a compact sequence in $(W^{1,5}(Q_T))^*$. The term in \eqref{eq:step6_tilde_psi} can be treated the same way. Thus, we may follow Subsection \ref{subsection:almost_everywhere_conv_theta} to deduce
\begin{align*}
    \theta^{\varepsilon_4}\rightarrow\theta\quad\text{ a.e. in }Q_T,
\end{align*}
which combined with \eqref{ineq:step6_nabla_theta_bound}, and our comments below \eqref{ineq:step6_nabla_theta_bound} gives us
\begin{align*}
    \partial_t \theta^{\varepsilon_4} &\rightharpoonup^* \partial_t \theta, &&\text{ weakly* in }L^q(0, T; (W^{1,q}(\Omega))^*), q > 5\\
    \theta^{\varepsilon_4} &\rightarrow \theta, &&\text{ strongly in }L^q(Q_T),\, q< 5/3,\\
    \nabla_x\theta^{\varepsilon_4} &\rightharpoonup \nabla_x\theta, &&\text{ weakly in }L^q(Q_T),\, q< 5/4,\\
    \theta^{\varepsilon_4} &\rightarrow \theta, &&\text{ a.e. in }Q_T,\\
    (\theta^{\varepsilon_4})^{\frac{2-\lambda}{2}} &\rightharpoonup\theta^{\frac{2-\lambda}{2}}, &&\text{ weakly in }L^2(0, T; W^{1,2}(\Omega)),\, \lambda\in (0, 1).
\end{align*}
Moreover, analysing \eqref{eq:step6_entropy}, one can obtain similarly to \eqref{estlnt} and \eqref{estgradln} the convergence
\begin{align*}
    \ln\theta^{\varepsilon_4}\rightharpoonup\ln\theta,\text{ weakly in }L^2(0, T; W^{1,2}(\Omega)).
\end{align*}
Having those convergences we can come back to the problem of strong convergence of $\{\bF^{\varepsilon_4}\}_{\varepsilon_4 > 0}$. At this point we can deduce that
\begin{align*}
    \partial_t \bF^{\varepsilon_4} &\rightharpoonup^* \partial_t \bF, &&\text{ weakly* in }L^2(0, T; ((W^{1,2}(\Omega))^{3\times 3})^*),\\
    \bF^{\varepsilon_4} &\rightharpoonup^* \bF, &&\text{ weakly* in }(L^\infty(Q_T))^{3\times 3},\\
    \varepsilon_4\nabla_x\bF^{\varepsilon_4} &\rightharpoonup 0, &&\text{ weakly in }(L^2(Q_T))^{3\times 3}.
\end{align*}
Thus, we can converge in \eqref{eq:step4_F} to obtain
\begin{equation}\label{eq:step6_almost_F}
    \begin{split}
        \partial_t \bF + \Diver_x(\bF\otimes\bv)- \overline{\Lambda_{\varepsilon_3}(|\bF|)\nabla_x \bv \,\bF}\,\frac{(\theta - \varepsilon_6)_+}{\theta}+ \frac{\tau(\theta)}{2}\left(\overline{\frac{(\det\bF - \varepsilon_5)_+}{\det\bF}\bF\,\bF^{T}\, \bF} - \overline{\frac{(\det\bF - \varepsilon_5)_+}{\det\bF}\bF}\right) = 0,
    \end{split}
\end{equation}
where, in accordance to the convention \eqref{conven} the function $\overline{\Lambda_{\varepsilon_3}(|\bF|)\nabla_x \bv \,\bF}$ is a weak limit of a function $\Lambda_{\varepsilon_3}(|\bF^{\varepsilon_4}|)\nabla_x \bv^{\varepsilon_4} \,\bF^{\varepsilon_4}$ in $L^2(Q_T)$, and $\overline{\frac{(\det\bF - \varepsilon_5)_+}{\det\bF}\bF\,\bF^{T}\, \bF}$ is a weak* limit of the function $\frac{(\det\bF^{\varepsilon_4} - \varepsilon_5)_+}{\det\bF^{\varepsilon_4}}\bF^{\varepsilon_4}\,\bF^{\varepsilon_4, T}\, \bF^{\varepsilon_4}$ in $(L^\infty(Q_T))^{3\times 3}$; analogously we define $\overline{\frac{(\det\bF - \varepsilon_5)_+}{\det\bF}\bF}$. As in Subsection \ref{subsection:strong_convergence_of_F}, one wants to compare \eqref{eq:step5_F} tested by $\varphi\bF^{\varepsilon_4}$ and \eqref{eq:step6_almost_F} tested by $\varphi\bF$, for any non-negative  $\varphi\in \mathcal{C}_c^\infty((-\infty,T)\times\Omega)$ (note that at this point we do not need the function $G_m$ defined in \eqref{defGm}, because the velocity has high enough regularity). Moreover, testing \eqref{eq:step5_F} by $\varphi\bF^{\varepsilon_4}$ gives us two additional terms
\begin{align*}
    \varepsilon_4\int_{Q_T}|\nabla_x\bF^{\varepsilon_4}|^2\varphi\diff x\diff t,\qquad \varepsilon_4\int_{Q_T}(\nabla_x\bF^{\varepsilon_4}\cdot\nabla_x\varphi)\bF^{\varepsilon_4}\diff x\diff t,
\end{align*}
first of which is non-negative, and second of which disappears as $\varepsilon_4\to 0^+$, hence, they are invisible when it comes to inequalities. With this in mind, one can easily verify that in this case we obtain the following analogue of \eqref{step1compFsigma} valid for all non-negative  $\varphi\in \mathcal{C}_c^\infty((-\infty,T)\times\Omega)$:
\begin{equation*}
    \begin{split}
    -&\int_{Q_T} \left(\overline{|\bF|^2}-|\bF|^2\right)\partial_t \varphi - \int_{Q_T} \left(\overline{|\bF|^2}-|\bF|^2\right) \bv\cdot \nabla \varphi \\
    &\qquad \leq \int_{Q_T} \left(\overline{|\bF|^2}-|\bF|^2\right)\varphi-\ell_{1}+2\ell_{2},
\end{split}
\end{equation*}
where
\begin{align*}
 \ell_{1}&:= \lim_{\varepsilon_4\to 0^+} \int_{Q_T} \tau(\theta)\frac{(\det\bF^{\varepsilon_4} - \varepsilon_5)_+}{\det\bF^{\varepsilon_4}}\left(|\bF^{\varepsilon_4} \bF^{\varepsilon_4, T}|^2 -\bF^{\varepsilon_4}\bF^{\varepsilon_4, T}\bF^{\varepsilon_4}:\bF\right)\varphi,
 \\
    \ell_{2}&:= \lim_{\varepsilon_4\to 0^+} \int_{Q_T} \left(\Lambda_{\varepsilon_3}(|\bF^{\varepsilon_4}|)\nabla_x \bv^{\varepsilon_4} \,\bF^{\varepsilon_4}\frac{(\theta^{\varepsilon_4} - \varepsilon_6)_+}{\theta^{\varepsilon_4}} : \bF^{\varepsilon_4}- \Lambda_{\varepsilon_3}(|\bF^{\varepsilon_4}|)\nabla_x\bv^{\varepsilon_4}\bF^{\varepsilon_4}\frac{(\theta^{\varepsilon_4} - \varepsilon_6)_+}{\theta^{\varepsilon_4}}:\bF\right)\varphi.
\end{align*}
We will verify, that $\ell_1$ is a sum of two integrals, one of which is non-negative, and the second one which satisfies an inequality similar to the one for $\ell_2$ below, while $\ell_2$ satisfies
\begin{equation*}
    \ell_{2}\leq \int_{Q_T} \tilde{L} \left(\overline{|\bF|^2}-|\bF|^2\right)\varphi
\end{equation*}
for all non-negative  $\varphi\in \mathcal{C}_c^\infty((-\infty,T)\times\Omega)$,
and some $\tilde{L}$ -- an $L^2(Q_T)$ function specified below, which will allow us to follow Subsection \ref{subsection:strong_convergence_of_F} and conclude that $\bF^{\varepsilon_4}$ converges strongly in $(L^2(Q_T))^{3\times 3}$ to $\bF$. For $\ell_1$ we may write
\begin{equation*}
    \begin{split}
        &\int_{Q_T} \tau(\theta)\frac{(\det\bF^{\varepsilon_4} - \varepsilon_5)_+}{\det\bF^{\varepsilon_4}}\left(|\bF^{\varepsilon_4} \bF^{\varepsilon_4, T}|^2 -\bF^{\varepsilon_4}\bF^{\varepsilon_4, T}\bF^{\varepsilon_4}:\bF\right)\varphi\diff x\diff t\\
        &= \int_{Q_T}\tau(\theta)(\bF^{\varepsilon_4}\bF^{\varepsilon_4, T}\bF^{\varepsilon_4} - \bF\,\bF^{T}\,\bF):(\bF^{\varepsilon_4} - \bF)\frac{(\det\bF^{\varepsilon_4} - \varepsilon_5)_+}{\det\bF^{\varepsilon_4}}\varphi\diff x\diff t\\
        &\qquad +\int_{Q_T}\tau(\theta)\bF\,\bF^T\,\bF:(\bF^{\varepsilon_4} - \bF)\left(\frac{(\det\bF^{\varepsilon_4} - \varepsilon_5)_+}{\det\bF^{\varepsilon_4}} - \frac{(\det\bF - \varepsilon_5)_+}{\det\bF}\right)\varphi\diff x\diff t\\
        &\qquad +\int_{Q_T}\tau(\theta)\bF\,\bF^T\,\bF:(\bF^{\varepsilon_4} - \bF) \frac{(\det\bF - \varepsilon_5)_+}{\det\bF}\varphi\diff x\diff t\\
        &=: I_1^{\varepsilon_4} + I_2^{\varepsilon_4} + I_3^{\varepsilon_4}.
    \end{split}
\end{equation*}
The fact that $I_1^{\varepsilon_4} \geq 0$ follows from the monotonicity of a function $\bF\mapsto \bF\bF^T\bF$ (see for example \cite[Lemma 4.2]{BuMaLo22}). Let us move into handling $I_2^{\varepsilon_4}$. Notice that since $\{\bF\}_{\varepsilon_4}$ is bounded in $(L^{\infty}(Q_T))^{3\times 3}$, there exists a $C> 0$ for which $\bF^{\varepsilon_4}$ is contained in the set $\{F\in \R^{3\times 3}: |F|\leq C\}$. Now, we may calculate
$$
\partial\frac{(\det F - \varepsilon_5)_+}{\det F} = \mathbf{1}_{\{\det F > \varepsilon_5\}}\frac{\varepsilon_5}{\det F}F^{-T}:\partial F,
$$
and on the set $\{F\in \R^{3\times 3}: |F|\leq C\}$
$$
 \left|\mathbf{1}_{\{\det F > \varepsilon_5\}}\frac{\varepsilon_5}{\det F}F^{-T}\right|\leq 3\frac{|F|^2}{\det F}\leq \frac{3C^2}{\varepsilon_5}.
$$
Hence,
$$
F\mapsto \frac{(\det F - \varepsilon_5)_+}{\det F}
$$
is Lipschitz on this set and one can bound
\begin{equation*}
    \begin{split}
        |I_2^{\varepsilon_4}|\leq \frac{3K}{\varepsilon_5}C^5\int_{Q_T}|\bF^{\varepsilon_4} - \bF|^2\varphi\diff x\diff t \rightarrow \frac{3K}{\varepsilon_5}C^5\int_{Q_T}(\overline{|\bF|^2} - |\bF|^2)\varphi\diff x\diff t,
    \end{split}
\end{equation*}
as $\varepsilon_4\to 0^+$. For the term $I_3^{\varepsilon_4}$ we can deduce from the weak* convergence of $\bF^{\varepsilon_4}$ to $\bF$ in $(L^\infty(Q_T))^{3\times 3}$ that 
$$
I_3^{\varepsilon_4}\rightarrow 0^+, \quad\varepsilon_4\to 0^+.
$$
To handle $\ell_2$ let us first recall, that analogously to the arguments below \eqref{stokes1} and \eqref{stokes2} we can split the velocity into two parts $\bv^{\varepsilon_4} = \bv_1^{\varepsilon_4} + \bv^{\varepsilon_4}_2$, where
\begin{align*}
    \bv_1^{\varepsilon_4} & \rightharpoonup \bv_1 && \mbox{weakly in } L^2(0,T; W_{\bn, \diver}^{1,2}),\\
    \partial_t \bv_1^{\varepsilon_4} & \rightharpoonup \partial_t\bv_1 && \mbox{weakly in } L^2(0,T; (W_{\bn, \diver}^{1,2})^*),\\
    \bv_1^{\varepsilon_4} &\to \bv_1 \quad &&\mbox{strongly in } (L^{\frac{10}{3})}(Q_T))^3,
\end{align*}
and
\begin{align*}
    \bv_2^{\varepsilon_4} & \rightharpoonup \bv_2 && \mbox{weakly in } L^2(0,T; W_{\bn, \diver}^{1,2}),\\
    \nabla \bv_2^{\varepsilon_4} & \to \nabla\bv_2 && \mbox{strongly in } (L^{2)}(Q_T))^{3\times 3}.
\end{align*}
Moreover, by Proposition \ref{bitingT}
\begin{align*}
    &\left(2\Lambda_{\varepsilon_3}(|\bF^{\varepsilon_4}|)g_{\varepsilon_1}(\theta^{\varepsilon_4})\bF^{\varepsilon_4}\,\bF^{\varepsilon_4,T}\frac{(\theta^{\varepsilon_4} - \varepsilon_6)_+}{\theta^{\varepsilon_4}} + \bD\bv^{\varepsilon_4}_1\right):\bD\bv^{\varepsilon_4}_1\\
    &\qquad\qquad\qquad\rightharpoonup \left(2\overline{\Lambda_{\varepsilon_3}(|\bF|)g_{\varepsilon_1}(\theta^{})\bF^{}\,\bF^{T}\frac{(\theta - \varepsilon_6)_+}{\theta}}+\bD\bv_1\right):\bD\bv_1\text{, in a biting sense in }L^1(Q_T),
\end{align*}
see also Proposition \ref{biting} for the definition of the biting convergence. Let us denote by $E_j$ the sequence of sets for which $|Q_T\setminus E_j|\rightarrow 0$ as $j\to +\infty$, and
\begin{align}
    &\left(2\Lambda_{\varepsilon_3}(|\bF^{\varepsilon_4}|)g_{\varepsilon_1}(\theta^{\varepsilon_4})\bF^{\varepsilon_4}\,\bF^{\varepsilon_4,T}\frac{(\theta^{\varepsilon_4} - \varepsilon_6)_+}{\theta^{\varepsilon_4}} + \bD\bv^{\varepsilon_4}_1\right):\bD\bv^{\varepsilon_4}_1\nonumber\\
    &\qquad\qquad\qquad\rightharpoonup \left(2\overline{\Lambda_{\varepsilon_3}(|\bF|)g_{\varepsilon_1}(\theta)\bF\,\bF^{T}\frac{(\theta - \varepsilon_6)_+}{\theta}}+\bD\bv_1\right):\bD\bv_1\text{, weakly in }L^1(E_j).\label{conv:step6_weak_product_1}
\end{align}
Furthermore, following \eqref{1ltunif} let $\mathcal{E}_k$ be the sequence of sets for which $|Q_T\setminus\mathcal{E}_k|\rightarrow 0$ and
$$
\frac{1}{g_{\varepsilon_1}(\theta^{\varepsilon_4})} \rightarrow \frac{1}{g_{\varepsilon_1}(\theta)},\text{ strongly in }L^\infty(\mathcal{E}_k)
$$
Note that, due to \eqref{conv:step6_weak_product_1}
\begin{align}
    &\lim_{\varepsilon_4\to 0^+}\int_{E_j\cap\mathcal{E}_k}\left(2\Lambda_{\varepsilon_3}(|\bF^{\varepsilon_4}|)g_{\varepsilon_1}(\theta^{\varepsilon_4})\bF^{\varepsilon_4}\,\bF^{\varepsilon_4,T}\frac{(\theta^{\varepsilon_4} - \varepsilon_6)_+}{\theta^{\varepsilon_4}}:\bD\bv^{\varepsilon_4} - 2\overline{\Lambda_{\varepsilon_3}(|\bF|)g_{\varepsilon_1}(\theta)\bF\,\bF^{T}\frac{(\theta - \varepsilon_6)_+}{\theta}}:\bD\bv\right)\varphi\diff x\diff t\nonumber\\
    & = - \lim_{\varepsilon_4\to 0^+}\int_{E_j\cap\mathcal{E}_k}|\bD\bv^{\varepsilon_4}_1 - \bD\bv|^2\varphi\diff x\diff t.\label{conv:step6_weak_product_2}
\end{align}
Having this, we may write
\begin{equation*}
    \begin{split}
        &\int_{E_j\cap\mathcal{E}_k} \left(\Lambda_{\varepsilon_3}(|\bF^{\varepsilon_4}|)\nabla_x \bv^{\varepsilon_4} \bF^{\varepsilon_4}\frac{(\theta^{\varepsilon_4} - \varepsilon_6)_+}{\theta^{\varepsilon_4}} : \bF^{\varepsilon_4}- \Lambda_{\varepsilon_3}(|\bF^{\varepsilon_4}|)\nabla_x\bv^{\varepsilon_4}\bF^{\varepsilon_4}\frac{(\theta^{\varepsilon_4} - \varepsilon_6)_+}{\theta^{\varepsilon_4}}:\bF\right)\varphi\diff x\diff t\\
        & = \int_{E_j\cap\mathcal{E}_k} \left(\Lambda_{\varepsilon_3}(|\bF^{\varepsilon_4}|)\nabla_x \bv^{\varepsilon_4} \bF^{\varepsilon_4}\frac{(\theta^{\varepsilon_4} - \varepsilon_6)_+}{\theta^{\varepsilon_4}} : \bF^{\varepsilon_4}- \Lambda_{\varepsilon_3}(|\bF^{\varepsilon_4}|)\nabla_x\bv\,\bF^{\varepsilon_4}\frac{(\theta^{\varepsilon_4} - \varepsilon_6)_+}{\theta^{\varepsilon_4}}:\bF^{\varepsilon_4}\right)\varphi\diff x\diff t\\
        &\qquad + \int_{E_j\cap\mathcal{E}_k}\frac{(\theta^{\varepsilon_4} - \varepsilon_6)_+}{\theta^{\varepsilon_4}}(\Lambda_{\varepsilon_3}(|\bF^{\varepsilon_4}|)\bF^{\varepsilon_4} - \Lambda_{\varepsilon_3}(|\bF|)\bF)\nabla_x\bv:(\bF^{\varepsilon_4} - \bF)\varphi\diff x\diff t\\
        &\qquad -\int_{E_j\cap\mathcal{E}_k}\frac{(\theta^{\varepsilon_4} - \varepsilon_6)_+}{\theta^{\varepsilon_4}}(\Lambda_{\varepsilon_3}(|\bF^{\varepsilon_4}|)\bF^{\varepsilon_4} - \Lambda_{\varepsilon_3}(|\bF|)\bF)(\nabla_x\bv^{\varepsilon_4} - \nabla_x\bv):\bF\varphi\diff x\diff t\\
        &\qquad +\int_{E_j\cap\mathcal{E}_k}\frac{(\theta^{\varepsilon_4} - \varepsilon_6)_+}{\theta^{\varepsilon_4}}\left(\nabla_x\bv\Lambda_{\varepsilon_3}(|\bF|)\bF:\bF^{\varepsilon_4} - \nabla_x\bv^{\varepsilon_4}\Lambda_{\varepsilon_3}(|\bF|)\bF:\bF\right)\diff x\diff t\\
        &=: J_1^{\varepsilon_4} + J_2^{\varepsilon_4} - J_3^{\varepsilon_4} + J_4^{\varepsilon_4}.
    \end{split}
\end{equation*}
Let us treat all the terms separately. For $J_1^{\varepsilon_4}$ we can write by \eqref{conv:step6_weak_product_2} and the strong convergence of the gradient of $\bv_2^{\varepsilon_4}$
\begin{equation*}
    \begin{split}
        &\lim_{\varepsilon_4\to 0^+}\int_{E_j\cap\mathcal{E}_k} \left(\Lambda_{\varepsilon_3}(|\bF^{\varepsilon_4}|)\nabla_x \bv^{\varepsilon_4} \bF^{\varepsilon_4}\frac{(\theta^{\varepsilon_4} - \varepsilon_6)_+}{\theta^{\varepsilon_4}} : \bF^{\varepsilon_4}- \Lambda_{\varepsilon_3}(|\bF^{\varepsilon_4}|)\nabla_x\bv\,\bF^{\varepsilon_4}\frac{(\theta^{\varepsilon_4} - \varepsilon_6)_+}{\theta^{\varepsilon_4}}:\bF^{\varepsilon_4}\right)\varphi\diff x\diff t\\
        &\leq -\lim_{\varepsilon_4\to 0^+}\int_{E_j\cap\mathcal{E}_k}\frac{|\bD\bv_1^{\varepsilon_4} - \bD\bv_1|^2}{g_{\varepsilon_1}(\theta)}\varphi\diff x\diff t + \frac{4}{\varepsilon_3^2}\lim_{\varepsilon_4\to 0^+}\int_{E_j\cap\mathcal{E}_k}|\nabla_x\bv_2^{\varepsilon_4} - \nabla_x\bv_2|\varphi\diff x\diff t\\
        &\leq -\lim_{\varepsilon_4\to 0^+}\int_{E_j\cap\mathcal{E}_k}\frac{|\bD\bv_1^{\varepsilon_4} - \bD\bv_1|^2}{R}\varphi\diff x\diff t, 
    \end{split}
\end{equation*}
where $R := \sup_{\theta > 0}g(\theta) > 0$. When it comes to $J_4^{\varepsilon_4}$ we can see that due to the almost everywhere convergence of $\theta^{\varepsilon_4}$ and the Vitali's convergence theorem
\begin{align*}
    \frac{(\theta^{\varepsilon_4} - \varepsilon_6)_+}{\theta^{\varepsilon_4}}\to \frac{(\theta - \varepsilon_6)_+}{\theta},\text{ strongly in }L^q(Q_T),\,\, q< +\infty,
\end{align*}
hence using weak convergence of $\bF^{\varepsilon_4}$ and $\nabla_x\bv^{\varepsilon_4}$
$$
\lim_{\varepsilon_4\to 0^+}J_4^{\varepsilon_4} = 0.
$$
Before moving to $J_2^{\varepsilon_4}$ and $J_3^{\varepsilon_4}$ notice that $F\to \Lambda_{\varepsilon_3}(|F|)$ can be chosen in such a way that $0\leq|\partial\Lambda_{\varepsilon_3}(|F|)|\leq 2\varepsilon_3$. Hence, the function
$$
F\mapsto \Lambda_{\varepsilon_3}(|F|)F
$$
is Lipschitz with (without a loss of generality) a Lipschitz constant smaller than $1$, and we may write
\begin{equation*}
    \begin{split}
        &\lim_{\varepsilon_4\to 0^+}J_2^{\varepsilon_4} = \lim_{\varepsilon_4\to 0^+}\int_{E_j\cap\mathcal{E}_k}\frac{(\theta^{\varepsilon_4} - \varepsilon_6)_+}{\theta^{\varepsilon_4}}(\Lambda_{\varepsilon_3}(|\bF^{\varepsilon_4}|)\bF^{\varepsilon_4} - \Lambda_{\varepsilon_3}(|\bF|)\bF)\nabla_x\bv:(\bF^{\varepsilon_4} - \bF)\varphi\diff x\diff t\\
        &\leq \int_{Q_T}|\nabla_x\bv|(\overline{|\bF|^2} - |\bF|^2)\varphi\diff x\diff t,
    \end{split}
\end{equation*}
Moving further, for $J_3^{\varepsilon_4}$ using Young's inequality, localized version of the Korn's inequality for sequences (see for example \cite{BuMaLo24} in the Appendix), and strong convergence of $\nabla_x\bv^{\varepsilon_4}_2$ in $L^1(Q_T)$, we deduce
\begin{equation*}
    \begin{split}
        &\lim_{\varepsilon_4\to 0^+}J_3^{\varepsilon_4} = \lim_{\varepsilon_4\to 0^+}\int_{E_j\cap\mathcal{E}_k}\frac{(\theta^{\varepsilon_4} - \varepsilon_6)_+}{\theta^{\varepsilon_4}}(\Lambda_{\varepsilon_3}(|\bF^{\varepsilon_4}|)\bF^{\varepsilon_4} - \Lambda_{\varepsilon_3}(|\bF|)\bF)(\nabla_x\bv^{\varepsilon_4} - \nabla_x\bv):\bF\varphi\diff x\diff t\\
        &\qquad\leq \lim_{\varepsilon_4\to 0^+}\int_{E_j\cap\mathcal{E}_k}|\bF||\bF^{\varepsilon_4} - \bF||\nabla_x\bv^{\varepsilon_4}_1 - \nabla_x\bv_1|\varphi\diff x\diff t\\
        &\qquad \leq C(R)\int_{E_j\cap\mathcal{E}_k}|\bF|^2(\overline{|\bF|^2} - |\bF|^2)\varphi\diff x\diff t + \lim_{\varepsilon_4\to 0^+}\int_{E_j\cap\mathcal{E}_k}\frac{|\bD\bv^{\varepsilon_4}_1 - \bD\bv_1|^2}{2R}\varphi\diff x\diff t
    \end{split}
\end{equation*}
Adding everything together we obtain
$$
\lim_{\varepsilon_4\to 0^+}J_1^{\varepsilon_4} + J_2^{\varepsilon_4} + J_3^{\varepsilon_4} + J_4^{\varepsilon_4} \leq \int_{Q_T}(|\nabla_x\bv| + C(R)|\bF|^2)(\overline{|\bF|^2} - |\bF|^2)\varphi\diff x\diff t.
$$
All that is left to do is consider the integral over the set $|Q_T\setminus(E_j\cap\mathcal{E}_k)|$. Here, we easily see by H\"{o}lder's inequality
\begin{equation*}
    \begin{split}
        &\lim_{\varepsilon_4\to 0^+}\left|\int_{Q_T\setminus(E_j\cap\mathcal{E}_k)} \left(\Lambda_{\varepsilon_3}(|\bF^{\varepsilon_4}|)\nabla_x \bv^{\varepsilon_4} \bF^{\varepsilon_4}\frac{(\theta^{\varepsilon_4} - \varepsilon_6)_+}{\theta^{\varepsilon_4}} : \bF^{\varepsilon_4}- \Lambda_{\varepsilon_3}(|\bF^{\varepsilon_4}|)\nabla_x\bv^{\varepsilon_4}\bF^{\varepsilon_4}\frac{(\theta^{\varepsilon_4} - \varepsilon_6)_+}{\theta^{\varepsilon_4}}:\bF\right)\varphi\diff x\diff t \right|\\
        &\qquad\leq \lim_{\varepsilon_4\to 0^+}\frac{8}{\varepsilon^2_3}\|\varphi\|_\infty\int_{Q_T\setminus(E_j\cap\mathcal{E}_k)} |\nabla_x\bv^{\varepsilon_4}|\diff x\diff t \\
        &\qquad \leq C(K, \varepsilon_3)\|\varphi\|_\infty\sqrt{|Q_T\setminus(E_j\cap\mathcal{E}_k)|}.
    \end{split}
\end{equation*}
Since $|Q_T\setminus(E_j\cap\mathcal{E}_k)|\to 0$ as $j, k\to +\infty$, we finally obtain the needed
$$
\ell_2 \leq \int_{Q_T}(|\nabla_x\bv| + C(R)|\bF|^2)(\overline{|\bF|^2} - |\bF|^2)\varphi\diff x\diff t.
$$
Thus, by Subsection \ref{subsection:strong_convergence_of_F} 
$$
\bF^{\varepsilon_4}\rightarrow\bF,\text{ strongly in }L^2(Q_T),
$$
which by interpolation implies
$$
\bF^{\varepsilon_4}\rightarrow\bF,\text{ strongly in }L^q(Q_T),\, q< +\infty.
$$
Now, from \eqref{eq:step5_e_with_theta} we can see that $e^{\varepsilon_4}$ inherits its regularity from $\theta^{\varepsilon_4}$ and $\bF^{\varepsilon_4}$. One obtains
\begin{align*}
    \partial_t e^{\varepsilon_4} &\rightharpoonup^* \partial_t e, &&\text{ weakly* in }L^q(0, T; (W^{1,q}(\Omega))^*), q > 5\\
    e^{\varepsilon_4} &\rightarrow e, &&\text{ strongly in }L^q(Q_T),\, q< 5/3,\\
    e^{\varepsilon_4} &\rightarrow e, &&\text{ a.e. in }Q_T.
\end{align*}
To keep the equality in the internal energy equation one needs to now show that $\nabla_x\bv^{\varepsilon_4}$ converges strongly in $L^2(Q_T)$. As we have already shown strong convergence of $\bF^{\varepsilon_4}$ in $L^2(Q_T)$ one can repeat the argument given in the previous step (see the discussion under \eqref{eq:energy_limit_step4}). All of the convergences are enough to go from \eqref{eq:step5_velocity}--\eqref{eq:step5_internal_energy} to
\begin{align}
    &\partial_t \bv + \diver_x (\Lambda_{\varepsilon_3}(|\bv|^2)\bv\otimes \bv) - \diver_x \bT_{\varepsilon_1, \varepsilon_3, \varepsilon_6}(\theta, \bF,  \bD\bv) = 0, \label{eq:step6_velocity}\\
    &\partial_t \bF + \Diver_x(\bF\otimes\bv)- \Lambda_{\varepsilon_3}(|\bF|)\nabla_x \bv\, \bF\frac{(\theta - \varepsilon_6)_+}{\theta}+ \frac{\tau(\theta)}{2}\frac{(\det\bF - \varepsilon_5)_+}{\det\bF}(\bF\,\bF^{T}\, \bF - \bF)= 0,\label{eq:step6_F}\\
    &\partial_te + \diver_x(e \bv) - \diver_x(\kappa(\theta)\nabla_x \theta)- \bT_{\varepsilon_1, \varepsilon_3, \varepsilon_6}(\theta, \bF, \bD\bv) : \bD\bv = 0,\label{eq:step6_internal_energy}
\end{align}
where \eqref{eq:step6_velocity}--\eqref{eq:step6_internal_energy} are to be understood in the weak sense in the sense of an operator on a predual space to the time derivative. Note that from \eqref{eq:step5_e_with_theta} we keep
\begin{equation}\label{eq:step6_e_with_theta}
\begin{split}
    e(t, x) = \theta(t, x) + (g_{\varepsilon_1}(\theta(t, x)) - \theta(t, x) g'_{\varepsilon_1}(\theta(t, x)))\tilde{\psi}_{\varepsilon_2}(\bF\bF^{T}),
\end{split}
\end{equation}
for $\tilde{\psi}_{\varepsilon_2}$ defined as in \eqref{eq:modified_psi_varepsilon2}.

Here, let us get rid of the operator $T^{\varepsilon_4}$ from the entropy equality \eqref{eq:step6_entropy}. We will show that
\begin{align}\label{conv:step6_nabla_F}
\lim_{\varepsilon_4\to 0^+}\left(\varepsilon_4\int_{Q_T}|\nabla_x\bF^{\varepsilon_4}|^2\diff x\diff t\right) = 0,
\end{align}
which by \eqref{ineq:step6_operator_norm_bound} will imply $T^{\varepsilon_4} \to 0$ strongly in $(W^{1,5}(Q_T))^*$.  Consider the equation \eqref{eq:step5_F} tested by $\bF^{\varepsilon_4}$, which gives as energy equality
\begin{equation}\label{eq:step6_energy_F_eps}
\begin{split}
    &\varepsilon_4\int_0^t\int_{\Omega}|\nabla_x\bF^{\varepsilon_4}|^2\diff x\diff\tau = -\int_{\Omega}|\bF^{\varepsilon_4}(t, x)|^2\diff x + \int_{\Omega}|\bF^{\varepsilon_4}_0(x)|^2\diff x\\
    &\qquad\qquad+ \int_0^t\int_{\Omega}\Lambda_{\varepsilon_3}(|\bF^{\varepsilon_4}|)\nabla_x \bv^{\varepsilon_4} :\bF^{\varepsilon_4}\bF^{\varepsilon_4,T}\frac{(\theta^{\varepsilon_4} - \varepsilon_6)_+}{\theta^{\varepsilon_4}} - \frac{\tau(\theta^{\varepsilon_4})}{2}\frac{(\det\bF^{\varepsilon_4} - \varepsilon_5)_+}{\det\bF^{\varepsilon_4}}(|\bF^{\varepsilon_4}\bF^{\varepsilon_4, T}|^2 - |\bF^{\varepsilon_4}|^2)\diff x\diff \tau,
\end{split}
\end{equation}
as well as the equation \eqref{eq:step6_F} tested by $\bF$ which provides
\begin{equation}\label{eq:step6_energy_F}
\begin{split}
    &0 = -\int_{\Omega}|\bF(t, x)|^2\diff x + \int_{\Omega}|\bF_0(x)|^2\diff x\\
    &\qquad\qquad+ \int_0^t\int_{\Omega}\Lambda_{\varepsilon_3}(|\bF|)\nabla_x \bv^l :\bF\bF^{T}\frac{(\theta - \varepsilon_6)_+}{\theta} - \frac{\tau(\theta)}{2}\frac{(\det\bF - \varepsilon_5)_+}{\det\bF}(|\bF\bF^{ T}|^2 - |\bF|^2)\diff x\diff \tau.
\end{split}
\end{equation}
Note, that the testing can be easily justified for the term $\Diver(F\otimes v)$, even though the gradient of $\bF$ no longer exists, via the Commutator lemma (cf. \cite[Lemma 4.1]{BuMaLo22} or \cite[Lemma II.1]{DiPerna1989}), since $\bF\in (L^\infty(Q_T))^{3\times 3}$, and $\bv\in L^2(0, T; W^{1,2}(\Omega))$. Now using the already established convergences for $\bF^{\varepsilon_4}$, $\nabla_x\bv^{\varepsilon_4}$ and $\theta^{\varepsilon_4}$ we can deduce
\begin{equation*}
    \begin{split}
        \int_0^t\int_{\Omega}\Lambda_{\varepsilon_3}(|\bF^{\varepsilon_4}|)\nabla_x \bv^{\varepsilon_4} &:\bF^{\varepsilon_4}\bF^{\varepsilon_4,T}\frac{(\theta^{\varepsilon_4} - \varepsilon_6)_+}{\theta^{\varepsilon_4}} - \frac{\tau(\theta^{\varepsilon_4})}{2}\frac{(\det\bF^{\varepsilon_4} - \varepsilon_5)_+}{\det\bF^{\varepsilon_4}}(|\bF^{\varepsilon_4}\bF^{\varepsilon_4, T}|^2 - |\bF^{\varepsilon_4}|^2)\diff x\diff \tau\\
        &\rightarrow \int_0^t\int_{\Omega}\Lambda_{\varepsilon_3}(|\bF|)\nabla_x \bv :\bF\bF^{T}\frac{(\theta - \varepsilon_6)_+}{\theta} - \frac{\tau(\theta)}{2}\frac{(\det\bF - \varepsilon_5)_+}{\det\bF}(|\bF\bF^{ T}|^2 - |\bF|^2)\diff x\diff \tau.
    \end{split}
\end{equation*}
Thus, applying as well weak lower semicontinuity of norms we obtain
\begin{equation*}
\begin{split}
    &\limsup_{\varepsilon_4\to 0^+}\left(\varepsilon_4\int_0^t\int_{\Omega}|\nabla_x\bF^{\varepsilon_4}|^2\diff x\diff\tau \right) \leq  -\int_{\Omega}|\bF(t, x)|^2\diff x + \int_{\Omega}|\bF_0(x)|^2\diff x\\
    &\qquad\qquad+ \int_0^t\int_{\Omega}\Lambda_{\varepsilon_3}(|\bF|)\nabla_x \bv :\bF\bF^{T}\frac{(\theta- \varepsilon_6)_+}{\theta} - \frac{\tau(\theta)}{2}\frac{(\det\bF - \varepsilon_5)_+}{\det\bF}(|\bF\bF^{ T}|^2 - |\bF|^2)\diff x\diff \tau\\
    & = 0,
\end{split}
\end{equation*}
which is what we wanted.

For further reference, we mention that one can converge with $l\to +\infty$ and $\varepsilon_4\to 0^+$ in lambda entropy equality \eqref{eq:step5_lambda_entr_equality} analogously as for the standard entropy \eqref{eq:step5_entropy}, and obtain
\begin{equation}\label{eq:step6_lambda_entr_equality}
    \begin{split}
        &\partial_t\eta^{}_\lambda -\tau\frac{(\det\bF^{} - \varepsilon_5)_+}{\det\bF^{}}((\bB^{})^2 - \bB^{}):\tilde{\psi}'_{\varepsilon_2}(\bB^{})(h_\lambda(\theta^{}) - g'_{\varepsilon_1}(\theta^{})(\theta^{})^{\lambda})\\
        &\qquad+ 2(\bB^{} - \bI):\bD\bv\Lambda_{\varepsilon_3}(|\bF^{}|)\frac{(\theta^{} - \varepsilon_6)_+}{\theta^{}}(h_\lambda(\theta^{}) - g'_{\varepsilon_1}(\theta^{})(\theta^{})^{\lambda})\\
        &=(1-\lambda)\kappa\frac{|\nabla_x\theta^{l}|^2}{(\theta^{})^{2 - \lambda}}\diff x + 2\nu\frac{|\bD\bv|^2}{(\theta^{})^{2 - \lambda}} + \tau\frac{(\det\bF^{} - \varepsilon_5)_+}{\det\bF^{}}\frac{g_{\varepsilon_1}(\theta^{})((\bB^{})^2 - \bB^{}):\tilde{\psi}'_{\varepsilon_2}(\bB^{})}{(\theta^{})^{1 - \lambda}} + \nu_{\lambda},
    \end{split}
\end{equation}
as an operator in $(W^{1,5}(Q_T))^*$, for $\nu_{\lambda}\in \mathcal{M}^+(Q_T)$.\\

\noindent\underline{Step 7: $\det\bF\geq \varepsilon_5$, and $\det\bF > 0$ uniformly.}
Let
$$
\bG_\delta := \bF^{-T}\frac{(\det\bF - \delta)_+}{\det\bF}.
$$
Note that
\begin{align*}
    \|\bG_\delta\|_\infty\leq \frac{C\|\bF\|_\infty^2}{\det\bF}\frac{(\det\bF - \delta)_+}{\det\bF} \leq \frac{C\|\bF\|_\infty^2}{\delta},
\end{align*}
hence $\bG_\delta\in (L^\infty(Q_T))^{3\times 3}$. Now, we would like to test \eqref{eq:step6_F} by $\bG_\delta$, which requires mollification. As the argument is standard, we again skip it for the sake of brevity, and only mention that to treat the term with the divergence one should utilize the Commutator lemma \cite[Lemma 4.1]{BuMaLo22}, \cite[Lemma II.1]{DiPerna1989} (see also \cite[Section 6.5]{BuMaLo22} for the full description of the method for the interested readers). Before writing down the full expression after the testing, let us denote the primitive function of $s\mapsto \frac{(s - \delta)_+}{s^2}$ as
\begin{align*}
    S_\delta(s) := \left\{\begin{array}{ll}
         \ln(s) + \frac{\delta}{s},\text{ whenever }s > \delta, \\
          \ln\delta + 1,\text{ otherwise}.
    \end{array}\right.
\end{align*}
Then, using the identity (cf. \cite[Lemma A.3]{bathory2024coupling}) $\partial\ln(\det\bF) = \partial\bF:\bF^{- T}$ and the fact that $\diver_x\bv = 0$, one obtains
\begin{align}\label{eq:step7_for_detF_0_1}
    &\int_\Omega(S_\delta(\det\bF(t)) - S_\delta(\det\bF_0))\diff x\nonumber\\
    &\qquad-\int_0^t\int_\Omega\Lambda_{\varepsilon_3}(|\bF|)\frac{(\theta - \varepsilon_6)_+}{\theta}\nabla_x\bv\,\bF:\bG_\delta + \frac{\tau}{2}\frac{(\det\bF - \varepsilon_5)_+}{\det\bF}(\bF\,\bF^T\,\bF - \bF):\bG_\delta\diff x\diff s = 0.
\end{align}
Due to the fact that $\bv$ is divergence free, we may write
\begin{equation}\label{eq:step7_some_eq_1}
\begin{split}
    &\int_0^t\int_\Omega\Lambda_{\varepsilon_3}(|\bF|)\frac{(\theta - \varepsilon_6)_+}{\theta}\nabla_x\bv\,\bF:\bG_\delta\diff x\diff t = \int_0^t\int_{\Omega}\Lambda_{\varepsilon_3}(|\bF|)\frac{(\theta - \varepsilon_6)_+}{\theta}\frac{(\det\bF - \delta)_+}{\det\bF}\mathrm{tr}(\bF^{-1}\nabla_x\bv\bF)\diff x\diff s\\
    & = \int_0^t\int_{\Omega}\Lambda_{\varepsilon_3}(|\bF|)\frac{(\theta - \varepsilon_6)_+}{\theta}\frac{(\det\bF - \delta)_+}{\det\bF}\mathrm{tr}(\nabla_x\bv\,\bF\bF^{-1})\diff x\diff s = 0.
\end{split}
\end{equation}
Thus, we have from \eqref{eq:step7_for_detF_0_1}
\begin{align*}
    \int_\Omega(S_\delta(\det\bF(t)) - S_\delta(\det\bF_0))\diff x + \int_0^t\int_\Omega \frac{\tau}{2}\frac{(\det\bF - \varepsilon_5)_+}{\det\bF}\frac{(\det\bF - \delta)_+}{\det\bF}(|\bF\,\bF^T| - 3)\diff x\diff s = 0,
\end{align*}
which implies
\begin{equation}\label{eq:step7_for_detF_0_2}
    \begin{split}
        \sup_{t\in(0, T)}\left|\int_\Omega S_\delta(\det\bF(t))\diff x\right| \leq \int_{\Omega}1 + |\ln\det\bF_0|\diff x + \int_0^T\int_\Omega|\bF|^2 + 3\diff x\diff t.
    \end{split}
\end{equation}
Since the right-hand side is bounded uniformly with respect to $\delta\rightarrow0^+$, and
\begin{align*}
    S_\delta(s)\rightarrow S(s) := \left\{\begin{array}{ll}
         \ln(s),\text{ whenever }s > 0, \\
          -\infty,\text{ otherwise},
    \end{array}\right.
\end{align*}
the inequality above implies, that $\det \bF > 0$ a.e. in $Q_T$, which in principle ends the proof of a second part of our step. We will now show the bound on $\|\ln\det\bF\|_1$, which will be useful for keeping this property throughout the limiting procedures, as by Fatou's lemma it will be kept whenever we have an almost everywhere convergent sequence. From \eqref{eq:step7_for_detF_0_2}, by Fatou's lemma
\begin{align*}
    \sup_{t\in(0, T)}\left|\int_\Omega \ln\det\bF(t)\diff x\right| \leq \int_{\Omega}1 + |\ln\det\bF_0|\diff x + \int_0^T\int_\Omega|\bF|^2 + 3\diff x\diff t,
\end{align*}
hence
\begin{align*}
    \int_0^T\int_\Omega(\ln\det\bF)_-\diff x\diff t \leq \int_0^T\int_\Omega(\ln\det\bF)_+\diff x\diff t + T\int_{\Omega}1 + |\ln\det\bF_0|\diff x + \int_0^T\int_\Omega|\bF|^2 + 3\diff x\diff t,
\end{align*}
which is enough to bound the negative part. As for the positive one, we can simply write using the known inequality $\ln\det(\bF\bF^T)\leq \mathrm{tr}(\bF\bF^T) - 3 \leq |\bF\bF^T| - 3$
\begin{equation*}
    \begin{split}
         &\int_0^T\int_\Omega(\ln\det\bF)_+\diff x\diff t =  \int_{\{\det\bF > 1\}}\ln\det\bF\diff x\diff t = \frac{1}{2}\int_{\{\det\bF > 1\}}\ln(\det\bF)^2\diff x\diff t\\
         & = \frac{1}{2}\int_{\{\det\bF > 1\}}\ln(\det\bF\bF^T)\diff x\diff t \leq \frac{1}{2}\int_{\{\det\bF > 1\}}|\bF\bF^T| - 3\diff x\diff t \leq \frac{1}{2}\int_{Q_T}|\bF|^2 + 3\diff x\diff t.
    \end{split}
\end{equation*}
In the end
\begin{equation}\label{ineq:step7_ln_det_F_bound}
    \begin{split}
        \|\ln\det\bF\|_1 = \int_{Q_T}(\ln\det\bF)_+ + (\ln\det\bF)_-\diff x\diff t \leq C(|Q_T|, \|\bF\|_2,\|\ln\det\bF_0\|_1).
    \end{split}
\end{equation}
We can move to proving the non-uniform bound from below on $\det\bF$. Since we already know that $\det\bF > 0$, then $\bF^{-1}$ is well defined, and we can denote
$$
\bH := -\bF^{-T}\det\bF(\det\bF -\varepsilon_5)_-.
$$
Using standard properties of the determinant, we can bound
\begin{align*}
    (\det\bF)^2 = \det(\bF^T\bF)\leq (\mathrm{tr}(\bF^T\bF))^3 = |\bF|^3,
\end{align*}
which implies
\begin{align*}
    \|\det\bF\|_\infty \leq \|\bF\|_\infty^{3/2}.
\end{align*}
Hence,
\begin{align*}
    \|\bH\|_\infty \leq C\|\bF\|_\infty^2\|\det\bF - \varepsilon_5\|_\infty \leq C(\|\bF\|_\infty^2 + \|\bF\|_\infty^{5/2}), 
\end{align*}
and in consequence $\bH\in (L^\infty(Q_T))^{3\times 3}$. Similarly as above we now test \eqref{eq:step6_F} by $\bH$, which necessitates the use of mollification, where the procedure is the same as for $\bG_\delta$ above. One can observe that $\partial(\det\bF - \varepsilon_5)^2_- = -\partial\bF:\bF^{-T}\det\bF(\det\bF - \varepsilon_5)_-$, thus after testing, and using $\diver_x\bv = 0$, we get similarly as before
\begin{align}\label{eq:step7_for_detF_0_3}
    &\int_\Omega((\det\bF(t) - \varepsilon_5)^2_- - (\det\bF_0 - \varepsilon_5)^2_-)\diff x -\int_0^t\int_\Omega\Lambda_{\varepsilon_3}(|\bF|)\frac{(\theta - \varepsilon_6)_+}{\theta}\nabla_x\bv\,\bF:\bH\diff x\diff s\\
    &\qquad+ \int_0^t\int_\Omega\frac{\tau}{2}\frac{(\det\bF - \varepsilon_5)_+}{\det\bF}(\bF\,\bF^T\,\bF - \bF):\bH\diff x\diff s = 0.
\end{align}
By an analogous procedure as in \eqref{eq:step7_some_eq_1}
\begin{align*}
    \int_0^t\int_\Omega\Lambda_{\varepsilon_3}(|\bF|)\frac{(\theta - \varepsilon_6)_+}{\theta}\nabla_x\bv\,\bF:\bH\diff x\diff s = 0,
\end{align*}
and using the fact that
$$
(\det\bF - \varepsilon_5)_-(\det\bF - \varepsilon_5)_+ = 0,
$$
we may deduce
\begin{align*}
    \int_0^t\int_\Omega\frac{\tau}{2}\frac{(\det\bF - \varepsilon_5)_+}{\det\bF}(\bF\,\bF^T\,\bF - \bF):\bH\diff x\diff s = 0.
\end{align*}
Combining this with \eqref{eq:step7_for_detF_0_3} we get
\begin{align*}
    \sup_{t\in (0, T)}\|(\det\bF(t) - \varepsilon_5)_-\|_2 \leq \|(\det\bF_0 - \varepsilon_5)_-\|_2.
\end{align*}
Since (briefly we introduce back our indexes)
\begin{align*}
    \det\bF^{\varepsilon_3,\varepsilon_5}_0 &= \det(T^{\det}_{\varepsilon_5}(T_{\varepsilon_3}(\bF_0)))\\
    &= \det(\bF_0\mathbf{1}_{\{|\bF_0| \leq 2/\varepsilon_3\}}\mathbf{1}_{\{\det\bF_0 > \varepsilon_5\}} + \mathbb{I}\,\mathbf{1}_{\{|\bF_0| > 2/\varepsilon_3\}}\mathbf{1}_{\{\det\bF_0 \leq \varepsilon_5\}} + \mathbb{I}\,\mathbf{1}_{\{|\bF_0| > 2/\varepsilon_3\}}) > \varepsilon_5
\end{align*}
the inequality above implies
\begin{align*}
    \|(\det\bF(t) - \varepsilon_5)_-\|_2 = 0,\text{ for a.e. }t\in (0, T),
\end{align*}
and in consequence $\det\bF \geq \varepsilon_5$ a.e. in $Q_T$.\\

\noindent\underline{Step 8: convergence with $\varepsilon_2\to 0^+$.} As in the previous step, due to the energy bounds on the velocity we obtain
\begin{align*}
    \partial_t \bv^{\varepsilon_2} &\rightharpoonup^* \partial_t \bv, &&\text{ weakly* in }L^{2}(0, T; (W^{1,2}_{\bn,\diver}(\Omega))^*),\\
    \bv^{\varepsilon_2} &\rightharpoonup^* \bv, &&\text{ weakly* in }L^{\infty}(0, T; L^2(\Omega)),\\
    \bv^{\varepsilon_2} &\rightharpoonup \bv, &&\text{ weakly in }L^2(0, T; W^{1,2}_{\bn,\diver}(\Omega)),\\
    \bv^{\varepsilon_2} &\rightarrow \bv, &&\text{ strongly in }L^q(Q_T),\, q< 10/3,\\
    \bv^{\varepsilon_2} &\rightarrow \bv, &&\text{ a.e. in }Q_T.
\end{align*}
Moreover, as is standard already in this proof, one can deduce that $\{\bF^{\varepsilon_2}\}_{\varepsilon_2 > 0}$ is bounded in $(L^\infty(Q_T))^{3\times 3}$. Now, we shall repeat the reasoning given in Step 6 for convergence with $\varepsilon_4\to 0^+$ to augment the convergence of the elastic tensor and temperature. by weak lower semicontinuity of an $L^2$ norm of $\nabla_x(\theta^{1-\lambda/2})$ and the established convergence \eqref{conv:step6_nabla_F} we can deduce from \eqref{ineq:step6_nabla_theta_bound}
\begin{align}\label{ineq:step8_nabla_theta_bound}
    \int_0^T\int_{\Omega}\frac{|\nabla_x\theta^{\varepsilon_2}|^2}{(\theta^{\varepsilon_2})^{\lambda}}\diff x\diff t \leq C(\lambda,K, \varepsilon_3), \qquad \lambda\in (1, 2),
\end{align}
Then, again one can repeat the argumentation under \eqref{est1nablathetalambda} we can deduce that $\{\theta^{\varepsilon_2}\}_{\varepsilon_2>0}$ is bounded in $L^q(Q_T)$ for $q < 5/3$ and $\{\nabla_x\theta^{\varepsilon_2}\}_{\varepsilon_2>0}$ is bounded in $L^{q}(Q_T)$ for $q < 5/4$. Similarly as in Step 6, we wish to follow the argument in Subsection \ref{subsection:almost_everywhere_conv_theta}, which means that we need the equations for the entropy $\eta^{\varepsilon_2}$ and $\tilde{\psi}_{\varepsilon_2}$. As the functions $1/\theta$ and $\tilde{\psi}'_{\varepsilon_2}$ are no longer proper test functions, we may follow the argument about convergence established in Step 6. Similarly as for \eqref{eq:step6_entropy} we obtain 
\begin{equation}\label{eq:step8_entropy}
    \begin{split}
        &\partial_t\eta^{\varepsilon_2} + \diver_x (\eta^{\varepsilon_2}\bv^{\varepsilon_2}) - \diver_x\left(\kappa\frac{\nabla_x\theta^{\varepsilon_2}}{\theta^{\varepsilon_2}}\right)\\
        &= \frac{\kappa|\nabla_x\theta^{\varepsilon_2}|^2}{(\theta^{\varepsilon_2})^2} + \frac{2\nu|\bD\bv^{\varepsilon_2}|^2}{\theta^{\varepsilon_2}} + \tau \frac{(\det\bF^{\varepsilon_2} - \varepsilon_5)_+}{\det\bF^{\varepsilon_2}}\frac{g_{\varepsilon_1}(\theta^{\varepsilon_2})\tilde{\psi}_{\varepsilon_2}'(\bB^{\varepsilon_2}):((\bB^{\varepsilon_2})^2 - \bB^{\varepsilon_2})}{\theta^{\varepsilon_2}} + \mu^{\varepsilon_2}, 
    \end{split}
\end{equation}
in the sense of operators in $(W^{1,5}(Q_T))^*$, for $\eta^{\varepsilon_2} = \ln\theta^{\varepsilon_2} - g'_{\varepsilon_1}(\theta^{\varepsilon_2})\tilde{\psi}_{\varepsilon_2}(\bB^{\varepsilon_2})$. Note that we know that $T^{\varepsilon_4}\to 0$ from the end of Step 6 and discussion around \eqref{conv:step6_nabla_F}. Moreover, we have good enough convergence properties in $\varepsilon_4$ to simply converge from \eqref{eq:step6_tilde_psi} to obtain
\begin{equation}\label{eq:step8_tilde_psi}
    \begin{split}
        &\partial_t\tilde{\psi}_{\varepsilon_2}(\bB^{\varepsilon_2}) + \diver_x(\tilde{\psi}_{\varepsilon_2}(\bB^{\varepsilon_2})\bv^{\varepsilon_2}) \\
        &+ \tau\frac{(\det\bF^{\varepsilon_2} - \varepsilon_5)_+}{\det\bF^{\varepsilon_2}}\tilde{\psi}_{\varepsilon_2}'(\bB^{\varepsilon_2}):((\bB^{\varepsilon_2})^2 - \bB^{\varepsilon_2}) = 2(\bB^{\varepsilon_2} - \bI):\bD\bv^{\varepsilon_2}\frac{(\theta^{\varepsilon_2} - \varepsilon_6)_+}{\theta^{\varepsilon_2}}.
    \end{split}
\end{equation}
in the sense of operators on $C^1([0, T]\times \overline{\Omega})$ for $\tilde{\psi}_{\varepsilon_2}$ defined as in \eqref{eq:modified_psi_varepsilon2}. At this point, there are no additional terms as in Step 6, hence we may simply copy the argument in Subsection \ref{subsection:almost_everywhere_conv_theta} to deduce
\begin{align*}
    \theta^{\varepsilon_2}\rightarrow\theta\quad\text{ a.e. in }Q_T,
\end{align*}
which combined with \eqref{ineq:step8_nabla_theta_bound}, and our comments below \eqref{ineq:step8_nabla_theta_bound} gives us
\begin{align*}
    \partial_t \theta^{\varepsilon_2} &\rightharpoonup^* \partial_t \theta, &&\text{ weakly* in }L^q(0, T; (W^{1,q}(\Omega))^*), q > 5\\
    \theta^{\varepsilon_2} &\rightarrow \theta, &&\text{ strongly in }L^q(Q_T),\, q< 5/3,\\
    \nabla_x\theta^{\varepsilon_4} &\rightharpoonup \nabla_x\theta, &&\text{ weakly in }L^q(Q_T),\, q< 5/4,\\
    \theta^{\varepsilon_2} &\rightarrow \theta, &&\text{ a.e. in }Q_T,\\
    (\theta^{\varepsilon_2})^{\frac{2-\lambda}{2}} &\rightharpoonup\theta^{\frac{2-\lambda}{2}}, &&\text{ weakly in }L^2(0, T; W^{1,2}(\Omega)),\, \lambda\in (0, 1).
\end{align*}
Moreover, analysing \eqref{eq:step8_entropy}, one can obtain similarly to \eqref{estlnt} and \eqref{estgradln} that
\begin{align}\label{ineq:step8_ln_theta_bound}
    \|\ln\theta^{\varepsilon_2}\|_{L^\infty(0, T; L^1(\Omega))} + \|\ln\theta^{\varepsilon_2}\|_{L^2(0, T; W^{1,2}(\Omega))}\leq C,
\end{align}
with a constant independent of $\varepsilon_1, \varepsilon_2, \varepsilon_5, \varepsilon_6$. Hence, by the Banach--Alaoglu theorem
\begin{align*}
    \ln\theta^{\varepsilon_2}\rightharpoonup\ln\theta,\text{ weakly in }L^2(0, T; W^{1,2}(\Omega)).
\end{align*}
When it comes to $\{\bF^{\varepsilon_2}\}_{\varepsilon_2 > 0}$, one can simply follow the same argumentation as in Step 6 to deduce 
\begin{align*}
    \partial_t \bF^{\varepsilon_2} &\rightharpoonup^* \partial_t \bF, &&\text{ weakly* in }L^2(0, T; ((W^{1,2}(\Omega))^{3\times 3})^*),\\
    \bF^{\varepsilon_2} &\rightharpoonup^* \bF, &&\text{ weakly* in }(L^\infty(Q_T))^{3\times 3},\\
    \bF^{\varepsilon_2}&\rightarrow\bF, &&\text{ strongly in }L^q(Q_T),\, q< +\infty,\\
    \bF^{\varepsilon_2} &\rightarrow \bF, &&\text{ a.e. in }Q_T.
\end{align*}
By equation \eqref{eq:step6_e_with_theta}, the sequence of internal energies inherits the regularities and convergence from the temperature and the elastic tensor. We have
\begin{align*}
    \partial_t e^{\varepsilon_2} &\rightharpoonup^* \partial_t e, &&\text{ weakly* in }L^q(0, T; (W^{1,q}(\Omega))^*), q > 5\\
    e^{\varepsilon_2} &\rightarrow e, &&\text{ strongly in }L^q(Q_T),\, q< 5/3,\\
    e^{\varepsilon_2} &\rightarrow e, &&\text{ a.e. in }Q_T.
\end{align*}
And again, to augment the convergence of the velocity gradient, one can repeat the argument given under \eqref{eq:energy_limit_step4} in Step 4. All of the convergences are enough to go from \eqref{eq:step6_velocity}--\eqref{eq:step6_internal_energy} to
\begin{align}
    &\partial_t \bv + \diver_x (\Lambda_{\varepsilon_3}(|\bv|^2)\bv\otimes \bv) - \diver_x \bT_{\varepsilon_1, \varepsilon_3,\varepsilon_6}(\theta, \bF,  \bD\bv) = 0, \label{eq:step8_velocity}\\
    &\partial_t \bF + \Diver_x(\bF\otimes\bv)- \Lambda_{\varepsilon_3}(|\bF|)\nabla_x \bv\,\bF\frac{(\theta - \varepsilon_6)_+}{\theta}+ \frac{\tau(\theta)}{2}\frac{(\det\bF - \varepsilon_5)_+}{\det\bF}(\bF\,\bF^{T}\, \bF - \bF)= 0,\label{eq:step8_F}\\
    &\partial_te + \diver_x(e \bv) - \diver_x(\kappa(\theta)\nabla_x \theta)- \bT_{\varepsilon_1, \varepsilon_3, \varepsilon_6}(\theta, \bF, \bD\bv) : \bD\bv = 0,\label{eq:step8_internal_energy}
\end{align}
where \eqref{eq:step8_velocity}--\eqref{eq:step8_internal_energy} are to be understood in the weak sense in the sense of an operator on a predual space to the time derivative. Note that from \eqref{eq:step6_e_with_theta} we keep
\begin{equation}\label{eq:step8_e_with_theta}
\begin{split}
    e(t, x) = \theta(t, x) + (g_{\varepsilon_1}(\theta(t, x)) - \theta(t, x) g'_{\varepsilon_1}(\theta(t, x)))\tilde{\psi}(\bF\bF^{T}),
\end{split}
\end{equation}
for $\tilde{\psi}$ defined as in \eqref{deftildepsi}.\\

\noindent\underline{Step 9: convergence with $\varepsilon_1, \varepsilon_5, \varepsilon_6\to 0^+$.} The arguments here are analogous to the ones in the previous step. The only reason we have singled out the convergence in $\varepsilon_2$, was the possibility to obtain the equality \eqref{eq:step4_helpful_eq}. Hence, we skip the argument here. We only note that by Fatou's lemma, from \eqref{ineq:step7_ln_det_F_bound} and \eqref{ineq:step8_ln_theta_bound} we may deduce
$$
\|\ln\theta\|_{L^\infty(0, T; L^1(\Omega))}, \|\ln\det\bF\|_{L^\infty(0, T; L^1(\Omega))} < +\infty,
$$
which will imply $\theta > 0$, $\det\bF > 0$, even after we converge with the lower bounds obtained in Step 3 and Step 7 to $0$. Here, we obtain
\begin{align}
    &\partial_t \bv + \diver_x (\Lambda_{\varepsilon_3}(|\bv|^2)\bv\otimes \bv) - \diver_x \bT(\theta, \bF,  \bD\bv) = 0, \label{eq:step9_velocity}\\
    &\partial_t \bF + \Diver_x(\bF\otimes\bv)- \Lambda_{\varepsilon_3}(|\bF|)\nabla_x \bv \bF+ \frac{\tau(\theta)}{2}(\bF\,\bF^{T}\, \bF - \bF)= 0,\label{eq:step9_F}\\
    &\partial_te + \diver_x(e \bv) - \diver_x(\kappa(\theta)\nabla_x \theta)- \bT(\theta, \bF, \bD\bv) : \bD\bv = 0,\label{eq:step9_internal_energy}
\end{align}
where \eqref{eq:step9_velocity}--\eqref{eq:step9_internal_energy} are to be understood in the weak sense in the sense of an operator on a predual space to the time derivative. Note that from \eqref{eq:step8_e_with_theta} we keep
\begin{equation}\label{eq:step9_e_with_theta}
\begin{split}
    e(t, x) = \theta(t, x) + (g(\theta(t, x)) - \theta(t, x) g'(\theta(t, x)))\tilde{\psi}(\bF\bF^{T}),
\end{split}
\end{equation}
for $\tilde{\psi}$ defined as in \eqref{deftildepsi}.\\

\noindent\underline{Step 10: convergence with $\varepsilon_3\to 0^+$.} The last step is in some sense the easiest one, as we can simply copy the arguments given for weak stability of solutions. The only question is whether we have appropriate a priori bounds as in Proposition \ref{propest} and Corollary \ref{corpropest}, where Corollary \ref{corpropest} follows from using the equation together with Proposition \ref{propest}. Note that the only problem that could appear are the a priori bounds related to entropy and lambda entropy, as those are the quantities, in which we cannot keep the equality throughout the approximating sequence. Instead, looking at Step 6 and Step 8, we have
\begin{equation}\label{eq:step10_entropy}
    \begin{split}
        &\partial_t\eta^{\varepsilon_3} + \diver_x (\eta^{\varepsilon_3}\bv^{\varepsilon_3}) - \diver_x\left(\kappa\frac{\nabla_x\theta^{\varepsilon_3}}{\theta^{\varepsilon_3}}\right)\\
        &= \frac{\kappa|\nabla_x\theta^{\varepsilon_3}|^2}{(\theta^{\varepsilon_3})^2} + \frac{2\nu|\bD\bv^{\varepsilon_3}|^2}{\theta^{\varepsilon_3}} + \tau \frac{g(\theta^{\varepsilon_3})|\bB^{\varepsilon_3} - \bI|^2}{\theta^{\varepsilon_3}} + \mu^{\varepsilon_3}, 
    \end{split}
\end{equation}
in the sense of operators in $(W^{1,5}(Q_T))^*$, and
\begin{equation}\label{eq:step10_lambda_entr_equality}
    \begin{split}
        &\partial_t\eta^{\varepsilon_3}_\lambda -\tau|\bB^{\varepsilon_3} - \bI|^2(h_\lambda(\theta^{\varepsilon_3}) - g'(\theta^{\varepsilon_3})(\theta^{\varepsilon_3})^{\lambda})\\
        &\qquad+ 2(\bB^{\varepsilon_3} - \bI):\bD\bv^{\varepsilon_3}\Lambda_{\varepsilon_3}(|\bF^{\varepsilon_3}|)(h_\lambda(\theta^{\varepsilon_3}) - g'(\theta^{\varepsilon_3})(\theta^{\varepsilon_3})^{\lambda})\\
        &=(1-\lambda)\kappa\frac{|\nabla_x\theta^{\varepsilon_3}|^2}{(\theta^{\varepsilon_3})^{2 - \lambda}}\diff x + 2\nu\frac{|\bD\bv^{\varepsilon_3}|^2}{(\theta^{\varepsilon_3})^{2 - \lambda}} + \frac{g(\theta^{\varepsilon_3})|\bB^{\varepsilon_3} - \bI|^2}{(\theta^{\varepsilon_3})^{1 - \lambda}} + \nu^{\varepsilon_3}_{\lambda},
    \end{split}
\end{equation}
as an operator in $(W^{1,5}(Q_T))^*$, for $\mu^{\varepsilon_3}, \nu^{\varepsilon_3}_{\lambda}\in \mathcal{M}^+(Q_T)$. But since both equations can be tested by $1$, and the measures are positive, one can proceed in the arguments.

\bibliographystyle{plainnat}

\bibliography{name}

\end{document}